\documentclass{amsart}       % onecolumn (second format)

\usepackage{graphicx}
\usepackage[all]{xy}
\usepackage{color}
\usepackage{hyperref}
\usepackage{enumerate}
\usepackage{amsmath}
\usepackage{amsthm}
\usepackage{amssymb}
\usepackage{amscd}
\usepackage{epsfig}
\usepackage[english]{babel}
\usepackage[stable]{footmisc}
\usepackage{ucs}
\usepackage[utf8x]{inputenc}

\newcommand{\cD}{\mathcal{D}}
\newcommand{\cE}{\mathcal{E}}
\newcommand{\cF}{\mathcal{F}}
\newcommand{\cG}{\mathcal{G}}

\newcommand{\cL}{\mathcal{L}}

\newcommand{\CC}{C^\infty}

\newcommand{\F}{\mathcal{F}}

\renewcommand{\d}{\mathrm d}               % differential of functions, forms, etc.
\newcommand{\G}{\mathcal{G}}            % Lie groupoid
\renewcommand{\O}{\mathcal{O}}             % Orbit of a Lie groupoid
\renewcommand{\H}{\mathcal{H}}          % Lie subgroupoid
\renewcommand{\gg}{\mathfrak{g}}        % Lie algebra
\newcommand{\tto}{\rightrightarrows}    % Arrows of a groupoid
\DeclareMathOperator{\Ker}{Ker}           % kernel

\numberwithin{equation}{section}
\allowdisplaybreaks

\newtheorem{theorem}{Theorem}[section]
\newtheorem{lemma}[theorem]{Lemma}
\newtheorem{proposition}[theorem]{Proposition}
\newtheorem{corollary}[theorem]{Corollary}

\theoremstyle{definition}
\newtheorem{definition}[theorem]{Definition}
\newtheorem{example}[theorem]{Example}

\newtheorem{remark}[theorem]{Remark}

\begin{document}

\title{Measures on differentiable stacks}

% author one information
\author{Marius Crainic}
\address{Mathematical Institute, Utrecht University,
The Netherlands}
\email{m.crainic@uu.nl}

% author two information
\author{Jo\~ao Nuno Mestre}
\address{Mathematical Institute, Utrecht University,
The Netherlands}
\curraddr{CMUP, Rua do Campo Alegre 687,  4169-007 Porto, Portugal}
\email{jnmestre@gmail.com}

\date{}

\begin{abstract} We introduce and study measures and densities (geometric measures) on differentiable stacks, using a rather straightforward generalization of Haefliger's approach to leaf spaces and to transverse measures for foliations. In general we prove Morita invariance, a Stokes formula which provides reinterpretations in terms of (Ruelle-Sullivan type) algebroid currents, and a Van Est isomorphism. In the proper case we reduce the theory to classical (Radon) measures on the underlying space, we provide explicit (Weyl-type) formulas that shed light on Weinstein's notion of volumes of differentiable stacks; in particular, in the symplectic case, we prove the conjecture left open in Weinstein's ``The volume of a differentiable stack'' \cite{alan-volume}. We also revisit the notion of Haar systems (and the existence of cut-off functions). Our original motivation comes from the study of Poisson manifolds of compact types \cite{PMCT1,PMCT2,PMCT3}, which provide two important examples of such measures: the affine and the Duistermaat-Heckman measures.

\end{abstract}

\maketitle

\textit{Mathematics Subject Classification} (2010). 58H05; 22A22.\newline
\textit{Keywords:} Lie groupoids, differentiable stacks, transverse measures.

\setcounter{tocdepth}{1}
\tableofcontents

%%%%%%%%%%%%%%%%%%%%%%%%%%%
%%%%%%%%%%%%%%%%%%%%%%%%%%%
%%%%%%%%%%%%%%%%%%%%%%%%%%%
%%%%%%%%%%%%%%%%%%%%%%%%%%%
%%%%%%%%%%%%%%%%%%%%%%%%%%%
%%%%%%%%%%%%%%%%%%%%%%%%%%%
%%%%%%%%%%%%%%%%%%%%%%%%%%%
%%%%%%%%%%%%%%%%%%%%%%%%%%%
%%%%%%%%%%%%%%%%%%%%%%%%%%%
%%%%%%%%%%%%%%%%%%%%%%%%%%%
\section{Introduction}\label{intro}
%%%%%%%%%%%%%%%%%%%%%%%%%%%
%%%%%%%%%%%%%%%%%%%%%%%%%%%
%%%%%%%%%%%%%%%%%%%%%%%%%%%
%%%%%%%%%%%%%%%%%%%%%%%%%%%
%%%%%%%%%%%%%%%%%%%%%%%%%%%
%%%%%%%%%%%%%%%%%%%%%%%%%%%
%%%%%%%%%%%%%%%%%%%%%%%%%%%
%%%%%%%%%%%%%%%%%%%%%%%%%%%
%%%%%%%%%%%%%%%%%%%%%%%%%%%
%%%%%%%%%%%%%%%%%%%%%%%%%%%

In this paper we study measures on differentiable stacks or, equivalently, transverse measures on Lie groupoids. We use the point of view that differentiable stacks can be represented by Morita equivalence classes of Lie groupoids; the intuition is that the stack associated to a Lie groupoid $\G$ over a manifold $M$ models the  space of orbits $M/\G$ of $\G$ (very singular in general!). Accordingly, one uses the notation $M//\G$ to denote the stack represented by $\G$. The framework that differentiable stacks provide for studying singular spaces is a rather obvious extension of Haefliger's philosophy \cite{haefliger2} on leaf spaces (and the transverse geometry of foliations), just that one allows now general Lie groupoids and not only \'etale ones. 

We will show that a rather straightforward extension of Haefliger's approach to transverse measures for foliations \cite{haefliger-minimal} allows one to talk about measures and geometric measures ( = densities) for differentiable stacks. For geometric measures (densities), when we compute the resulting volumes, we will recover the formulas that are taken as definition by Weinstein \cite{alan-volume}; also, using our viewpoint, we prove the conjecture left open in {\it loc.cit}.

We would also like to stress that the original motivation for this paper comes from the study of Poisson manifolds of compact types \cite{PMCT1,PMCT2,PMCT3}; there, the space of symplectic leaves comes with two interesting measures: the integral affine and the Duistermaat-Heckman measures. Although the leaf space is an orbifold, the measures are of a ``stacky'' nature that goes beyond Haefliger's framework. From this point of view, this paper puts those measures into the general framework of measures on stacks; see also the comments below and our entire Section \ref{sec-sympl-gpds}.\\

The measures that we have in mind in this paper are Radon measures on locally compact Hausdorff spaces $X$, interpreted as positive linear functionals on the space $C_{c}(X)$ of compactly supported continuous functions, or $C_{c}^{\infty}(X)$ if $X$ is a manifold (the basic
definitions are collected in Section \ref{measures}). Therefore, for a stack $M//\G$ represented by a Lie groupoid $\G\tto M$, the main problem is to define the space $C_{c}^{\infty}(M//\G)$; this is done in Haefliger's style, by thinking about ``compactly supported smooth functions on the orbit space $M/\G$''. To make sure that the resulting notion of measure, called transverse measures for $\G$, is intrinsic to the stack, we have to prove the following (see Theorem  \ref{theorem-ME} for a more precise statement): 

\begin{theorem} The notion of transverse measure for Lie groupoids is invariant under Morita equivalences.
\end{theorem}

We will also look at measures that are geometric. On a manifold $X$, these correspond to densities $\rho\in \mathcal{D}(X)$ on $X$ - for which the resulting integration is based on the standard integration on $\mathbb{R}^n$. We will show that the same discussion applies to differentiable stacks. For a groupoid $\G\tto M$, the object that is central to the discussion (and which replaces the density bundle $\mathcal{D}_{X}$ from the classical case) is the transverse density bundle 
\[ \mathcal{D}_{\G}^{\textrm{tr}}:= \mathcal{D}_{A^*}\otimes \mathcal{D}_{TM} \]
(the tensor product of the density bundle of the dual of the Lie algebroid $A$ of $\G$ and the density bundle of $TM$). This is a representation of $\cG$ or, equivalently, it represents a vector bundle of the stack $M//\G$; morally, this is the density bundle $\mathcal{D}_{M//\G}$ of the stack $M//\G$. 

While geometric measures on a manifold $X$ correspond to sections of $\mathcal{D}_X$, geometric measures on $M//\cG$ correspond to sections of the vector bundle represented by $\mathcal{D}_{\G}^{\textrm{tr}}$ or, more precisely, by 
$\cG$-invariant sections of $\mathcal{D}_{\G}^{\textrm{tr}}$. Following Haefliger's viewpoint, these will be called transverse densities for $\G$. Such a transverse density $\sigma$ can be decomposed  as 
\[ \sigma= \rho^{\vee}\otimes \tau,\] 
where $\rho$ is a density on the Lie algebroid and $\tau$ one on $M$; such decompositions are useful for writing down explicit formulas for the resulting measures. 

While the general theory may seem rather abstract, it becomes more concrete when applied to stacks represented by Lie groupoids $\G\tto M$ which are proper in the sense that the map $(s, t): \G\rightarrow M$ is proper ($s$ is the source map, and $t$ is the target). This notion generalizes compactness of Lie groups and properness of Lie group actions; in Haefliger's picture, it corresponds to leaf spaces that are orbifolds. Section \ref{sec-proper} is devoted to the proper case. The key remark is that  
the orbit space 
\[ B:= M/\G\]
(with the quotient topology) is not so pathological anymore: it is a locally compact Hausdorff space and there is even a well-behaved notion of ``smooth functions on $B$'': functions that, when pulled-back to $M$, are smooth. In particular, one can look at (standard) Radon measure on the topological space $B$. We will prove:

\begin{theorem} For proper groupoids $\G$, averaging induces a 1-1 correspondence: 
\[ 
\left\{\txt{transverse\\ measures for $\cG$ \,}\right\}
\stackrel{1-1}{\longleftrightarrow}
\left\{\txt{standard\ Radon\ measures\\ on\ the\ orbit\ space\ $B$}\right\}
.\]
\end{theorem}

For a more detailed statement, please see Theorem \ref{thm-iso-Cc(B)}. This theorem is rather surprising: it says that, from the point of view of measure theory, proper stacks can be treated as ordinary topological spaces. Although this may sound a bit disappointing, there is a very important gain here: the {\it geometric} measures on $B$, which depend on the full stack structure and not just on the underlying space. More concretely: one obtains abstract Radon measures $\mu_{\sigma}$ on $B$ out of more geometric data, namely the transverse densities $\sigma= \rho^{\vee}\otimes \tau$ mentioned above; the resulting integration can be computed by rather explicit Weyl-type formulas:
\begin{equation}\label{Weyl-type-0} 
\int_M f(x)\ d\mu_{\tau}(x) =\int_B\left(\int_{\O_b}f(y)\ d\mu_{\rho_{\O_b}}(y) \right)\ d\mu_{\sigma}(b),
\end{equation}
This is explained in our Proposition \ref{prop-tr-dens-B}, which is one of the main results of this paper (see also the comment that follows the proposition).

Another digression is provided by the case of symplectic groupoids, which is treated in Section \ref{sec-sympl-gpds}. This case presents several particularities, due to the presence of symplectic/Poisson geometry. In particular, we recast the measures from \cite{PMCT2}:

\begin{theorem} Any regular proper symplectic groupoid $(\G, \Omega)$ over $M$ carries a canonical transverse density. This gives rise to a canonical measure $\mu_{\mathrm{aff}}$ on the leaf space $B= M/\G$, called the affine measure. 

If $\G$ is $s$-proper, but not necessarily regular, then it carries a canonical transverse density. This gives rise to a canonical measure  $\mu_{DH}$ on $B$, called the Duistermaat-Heckman measure; it is the  push forward of the Liouville measure induced by $\Omega$.
\end{theorem}

See Corollary \ref{corr-sympl-gpd-affine} and Subsection \ref{The general case}. In \cite{PMCT2} it is shown that the two measures are related by a Duistermaat-Heckman formula. As application of our framework, we will provide a new proof of the Weyl-type integration formula from \cite{PMCT2} using (\ref{Weyl-type-0}), and we will prove the characterization of the affine measure $\mu_{\mathrm{aff}}$ that is left as a conjecture in \cite{alan-volume}. 

Returning to the general theory, a surprising (at first) feature of transverse measures (geometric or not!) is that:

\begin{theorem} Any transverse measure for a Lie groupoid $\cG$  satisfies the Stokes formula; hence it  gives rise to a closed algebroid current. If $\cG$ is $s$-connected, this gives rise to a 1-1 correspondence:
\[ 
\left\{\txt{transverse\\ measures for $\cG$ \,}\right\}
\stackrel{1-1}{\longleftrightarrow}
\left\{\txt{(positive) closed $A$-currents of top degree\\ with coefficients in the orientation bundle $\mathfrak{o}_A$ \,} \right\}
.\]
\end{theorem}

The precise statement is given in Propositions \ref{pre-VE-iso-deg-0} and \ref{VE-iso-deg-0}. The surprise comes from the fact that Stokes  formulas are usually satisfied only by geometric measures. The explanation is that the Stokes formula in this setting involves derivatives only in the longitudinal direction. Note that, when applied to foliations, this construction gives rise to the standard Ruelle-Sullivan current associated to a transverse measure. 

Conceptually, the space $C_{c}^{\infty}(M//\G)$ that we use to define the notion of transverse measure is best understood in terms of differentiable cohomology with compact supports (which, modulo re-indexing, will also be called ``differentiable homology''); one has, directly from the definitions, 
\[ C_{c}^{\infty}(M//\G)= H_0(\G)= H^{r}_{c}(\G),\]
where $r$ is the dimension of the $s$-fibers of $\G$. We will prove

\begin{theorem} $H_{\bullet}(\G)= H^{r- \bullet}_{c}(\G)$ is Morita invariant and it is related to the compactly supported cohomology of the Lie algebroid by a Van Est map 
\[ VE^{\bullet}: H^{\bullet}_{c}(A, \mathfrak{o}_A) \rightarrow H^{\bullet}_{c} (\G) .\]
If the $s$-fibers of $\G$ are homologically $k$-connected, where $k\in \{0, 1, \ldots, r-1\}$ (and $r$ is the rank of $A$), then $VE^0, \ldots, VE^k$ are isomorphisms.
\end{theorem}

This is Theorem \ref{thm-Mor-inv} and Theorem \ref{thm-Van-Est} in the main body. This theorem provides the tools to prove some of the previous results (e.g. the Stokes formula and the reinterpretation in terms of currents).

The main body of the paper provides the details for the story we have just described. However, along the way, we also have to revisit some of the standard material that is, we believe, not so clearly treated in the existing literature. One such standard material concerns Haar systems for Lie groupoids $\G$, i.e.,  families 
\[ \mu= \{\mu^x\}\]
of measures on the $s$-fibers of $\G$ that are invariant under right translations. In the ``geometric case'', i.e., when the measures come from densities, this boils down to (certain) sections of the density bundle $\mathcal{D}_A$ of the algebroid $A$ of $\G$; we call them Haar densities. When trying to use such Haar systems/densities to perform ``standard averaging'', one is lead to the notion of properness of $\mu$: which means that the source map is proper when restricted to the support of $\mu$. We will show: 

\begin{proposition} For any Lie groupoid $\G\tto M$, the following are equivalent:
\begin{enumerate}
\item[1.] $\G$ is proper.
\item[2.] $\G$ admits a proper Haar system.
\item[3.] $\G$ admits a proper Haar density.
\end{enumerate}
In particular, if $\rho\in C^\infty(M, \mathcal{D}_A)$ is a full Haar density then there exists a function $c\in C^{\infty}(M)$ such that $c\cdot \rho$ is a proper Haar density which, moreover, may be arranged to be normalized. Such a $c$ is called a \textbf{cut-off function for $\rho$}.
\end{proposition}

Cut-off functions are an important technical tool since they allow to 
perform averaging in order to produce invariant objects (functions, metrics, etc); the standard reference for the existence of such functions is \cite{Tu}; our previous proposition is the result of our attempt to understand the proof from {\it loc. cit.} and the actual meaning of such cut-off functions. Given the independent interest on Haar systems, all these results are put together into an appendix.

Finally, we would like to point our that our approach is very much compatible with the non-commutative approach to measures on orbit spaces \cite{Connes2,Connes}, as the ``localization at units'' of the non-commutative measures arising as traces on the convolution algebras (see Proposition \ref{equiv-to-trace} and \ref{prop-equiv-viewpoints}). Of course, the non-commutative viewpoint calls for more: understanding higher traces (Hochschild/cyclic cohomology computations) and other localizations (not just at units) in more geometric terms. We hope that our study of transverse measures/densities is a useful step in that direction; it would be interesting to combine it with the related computations of Pflaum-Posthuma-Tang from \cite{PPT}.\\

%%%%%%%%%%%%%%%%%%%%%%
%%%%%%%%%%%%%%%%%%%%%%
%%%%%%%%%%%%%%%%%%%%%%
%%%%%%%%%%%%%%%%%%%%%%
%%%%%%%%%%%%%%%%%%%%%%
%%%%%%%%%%%%%%%%%%%%%%
\noindent \textbf{Acknowledgements}
%%%%%%%%%%%%%%%%%%%%%%
%%%%%%%%%%%%%%%%%%%%%%
%%%%%%%%%%%%%%%%%%%%%%
%%%%%%%%%%%%%%%%%%%%%%
%%%%%%%%%%%%%%%%%%%%%%
%%%%%%%%%%%%%%%%%%%%%%

This research was supported by the Nederlandse Organisatie voor Wetenschappelijk Onderzoek Vici Grant no. 639.033.312. The second author was supported also by the Funda\c{c}\~{a}o para a Ci\^encia e a Tecnologia grant  SFRH/BD/71257/2010 under the POPH/FSE programmes. We would also like to acknowledge various discussions with Rui Loja Fernandes, Ioan M\unichar{259}rcu\unichar{539} and David Mart\'inez Torres.

%%%%%%%%%%%%%%%%%%%%%%
%%%%%%%%%%%%%%%%%%%%%%
%%%%%%%%%%%%%%%%%%%%%%
%%%%%%%%%%%%%%%%%%%%%%
%%%%%%%%%%%%%%%%%%%%%%
%%%%%%%%%%%%%%%%%%%%%%
\noindent \textbf{Notation and conventions:}
%%%%%%%%%%%%%%%%%%%%%%
%%%%%%%%%%%%%%%%%%%%%%
%%%%%%%%%%%%%%%%%%%%%%
%%%%%%%%%%%%%%%%%%%%%%
%%%%%%%%%%%%%%%%%%%%%%
%%%%%%%%%%%%%%%%%%%%%%

All the vector bundles in this paper are assumed to be of constant rank. Given a vector bundle $E$ over a manifold $M$, we denote by $\Gamma(E)= C^\infty(M,E)$ the space of smooth sections of $E$ and by $C_c^\infty(M,E)$ the space of smooth sections with compact support.

We use the notation $\cG\tto M$ to indicate that $\cG$ is a Lie groupoid over $M$. The source and target maps are denoted by $s$ and $t$, respectively. The fibers of $s$ will be called $s$-fibers. We only 
 work with groupoids whose $s$-fibers have the same dimension or, equivalently, whose algebroid $A$ has constant rank. More notation related to groupoids and algebroids can be found in the Appendix.
 
All the groupoids appearing in this paper are assumed to be Hausdorff.

%%%%%%%%%%%%%%%%%%%%%%
%%%%%%%%%%%%%%%%%%%%%%
%%%%%%%%%%%%%%%%%%%%%%
%%%%%%%%%%%%%%%%%%%%%%
%%%%%%%%%%%%%%%%%%%%%%
%%%%%%%%%%%%%%%%%%%%%%
\section{Measures on manifolds}\label{measures}
%%%%%%%%%%%%%%%%%%%%%%
%%%%%%%%%%%%%%%%%%%%%%
%%%%%%%%%%%%%%%%%%%%%%
%%%%%%%%%%%%%%%%%%%%%%
%%%%%%%%%%%%%%%%%%%%%%
%%%%%%%%%%%%%%%%%%%%%%

\subsection{Radon measures}\label{subsec-Radon measures}
To fix the notations, we start by recalling some background on (Radon) measures on manifolds. 

\begin{definition} Let $X$ be a locally compact Hausdorff topological space. A \textbf{(Radon) measure} on $X$ is a linear functional 
\[ \mu:  C_c(X)\rightarrow \mathbb{R}\]
which is positive, i.e., satisfies $\mu(f)\geq 0$ for all $f\in C_{c}(X)$ with $f\geq 0$.
\end{definition} 

Via the Riesz duality theorem (cf. e.g. \cite[Ch.\ 2]{rudin}), this notion is equivalent to the more intuitive one in terms of set-measures, i.e., measures  
\[ \mu: \mathcal{B}(X) \rightarrow [0, \infty]\]
defined on the $\sigma$-algebra generated by the topology of $X$, with the property that $\mu$ is finite on compacts, and is inner regular in the sense that, for any Borel set $B$,
\begin{equation}\label{inner-reg} 
\mu(B) = \textrm{sup}\, \{ \mu(K)\ :\ K \subset B,\ K- \textrm{compact}\}.
\end{equation}
This condition implies that $\mu$ is uniquely determined by what it does on compacts; it is also determined by what it does on opens, since it follows that 
\begin{equation}\label{outer-reg-on-cpcts} 
\mu(K)= \textrm{inf}\, \{ \mu(U)\ :\ K\subset U,\ U-\textrm{open} \}
\end{equation}
for any compact $K\subset X$. Such a measure gives rise to an integration map
\[ I_{\mu}: C_{c}(X) \rightarrow \mathbb{R}, \ \ I_{\mu}(f)= \int_{X} f(x)\ d\mu(x).\]
defined on the space of compactly supported continuous functions on $X$, i.e., a measure in the sense of the previous definition. The way in which $\mu$ can be recovered from $I_{\mu}$ comes from the formula 
\[ \mu(U)= \textrm{sup} \{ I_{\mu}(f)\ :\ f\in C_{c}(X), \quad 0\leq f \leq 1, \quad \textrm{supp(f)}\subset U \}\]
valid for all open subsets $U\subset X$ (together with the previous remark that $\mu$ is uniquely determined by what it does on open subsets).

\begin{remark}[a handy notation]\label{notation-integrals-0}\rm \ 
Although the point of view of set-measures will not be used in the paper, we will still use 
the formula
\[ \mu(f)= \int_{X} f(x)\ d\mu(x) \]
as a notation that indicates which is the function on $X$ that $\mu$ acts on. For instance, if $\mu$ is a Radon measure on a group $G$, then writing
$\int_{G} f(gh, h^2g^3)\ d\mu(h)$ 
indicates that $g$ is fixed and one applies $\mu$ to the function $G\ni h \mapsto f(gh, h^2g^3)$. 
\end{remark}

\begin{remark}[automatic continuity + a guiding idea for our approach]\label{automatic-continuity} \rm \ 
The definition that we adopted is often convenient to work with because it is purely algebraic. However, $C_{c}(X)$ carries its standard topology: the inductive limit topology that arises by writing $C_{c}(X)$ as the union over all compacts $K\subset X$ of the spaces $C_{K}(X)$ of continuous functions supported in $K$, where each such space is endowed with the sup-norm; convergence 
$f_n\to f$ in this topology means convergence in the sup-norm together with the condition that there exists a compact $K\subset M$ such that $\textrm{supp}(f_n)\subset K$ for all $n$.

The main reason that continuity is not mentioned in the definition is because it follows from the positivity condition (see also below). However, it is very important. For instance, if one wants to allow more general measures (real-valued), then one gives up on the positivity condition but one requests continuity.

In contrast with the continuity property (which is important but not explicitly required), the space $C_c(X)$ on which $\mu$ is defined is not so important. A very good illustration of this remark is the fact that, if $X$ is a manifold, we could use the space $C_{c}^{\infty}(X)$ (which is more natural in the smooth setting) to define the notion of measure. All that matters is that the (test) spaces $C_c(X)$ and $C_{c}^{\infty}(X)$ contain enough ``test functions'' to model the topology of $X$ (hence to characterize completely the Radon measures). This remark is important for our approach to transverse measures, approach that is centred around the question: what is the correct space of ``transverse test functions''? 

The following brings together some known results that we were alluding to; we include the proof for completeness and also for later use (cf. Proposition \ref{1to1measures-orbi}).
\end{remark}

\begin{proposition}\label{1to1measures}
For any smooth manifold $M$, 
\begin{enumerate}
\item[1.] $C_{c}^{\infty}(M)$ is dense in $C_{c}(M)$.
\item[2.] any positive linear functional on $C_{c}(M)$ is automatically continuous; and the same holds for $C_{c}^{\infty}(M)$ with the induced topology.  
\item[3.] the construction $\mu\mapsto \mu|_{C_{c}^{\infty}(M)}$ induces a 1-1 correspondence:
% between (Radon) measures on $M$ and positive linear functionals on $C_{c}^{\infty}(M)$.
\[ \{\textrm{(Radon)\ measures\ on}\ M\} \stackrel{\sim}{\longleftrightarrow} \ \{ \textrm{positive\ linear\ functionals\ on}\ C_{c}^{\infty}(M)\} .\]
\end{enumerate}
\end{proposition}

Sometimes one would like to integrate more than just {\it compactly supported} functions. One way to do so is by working with measures with compact supports. The notion of support of a measure $\mu$, denoted $\textrm{supp}(\mu)$ can be defined since measures on $X$ form a sheaf (due to the existence of partitions of unity). Explicitly, it is the smallest closed subspace of $X$ with the property that
\[ \mu(f)= 0\ \ \ \forall f\in C_{c}(X\setminus \textrm{supp}(\mu)) .\]
Note that, when $X$ is a manifold, one could also use $C_{c}^{\infty}(X)$ to characterize the support. In the next lemma, continuity refers to the topology on $C(X)$ of uniform convergence on compacts. We now assume that $X$ is also second countable. The following result follows along the lines of Chapters 22 and 24 in \cite{Treves}.

\begin{lemma}\label{cpctly-cupp-measures} For a measure $\mu$ on $X$, the following are equivalent:
\begin{enumerate}
\item[1.] $\mu$ is compactly supported.
\item[2.] $\mu$ has an extension (necessarily unique) to a linear continuous map
\[ \mu: C(X) \rightarrow \mathbb{R} \]
\end{enumerate}
(and, for manifolds, the similar statement that uses $C^{\infty}(X)$).
\end{lemma}

%%%%%%%%%%%%%%%%%%%%%%%%%%%
%%%%%%%%%%%%%%%%%%%%%%%%%%%
%%%%%%%%%%%%%%%%%%%%%%%%%%%
%%%%%%%%%%%%%%%%%%%%%%%%%%%
%%%%%%%%%%%%%%%%%%%%%%%%%%%
\subsection{Geometric measures: densities}
\label{Geometric measures: densities}
%%%%%%%%%%%%%%%%%%%%%%%%%%%
%%%%%%%%%%%%%%%%%%%%%%%%%%%
%%%%%%%%%%%%%%%%%%%%%%%%%%%
%%%%%%%%%%%%%%%%%%%%%%%%%%%
%%%%%%%%%%%%%%%%%%%%%%%%%%%
In the case of manifolds there is a distinguished class of measures to consider: the ones for which, locally, the corresponding integration is  the standard integration on $\mathbb{R}^n$. This brings us to the notion of densities on manifolds, that we recall next (cf. \cite{densities} for a textbook account).

Remark first that any group homomorphism $\delta: GL_r\rightarrow \mathbb{R}^*$ allows us to associate to any $r$-dimensional vector space a canonical 1-dimensional vector space
\[ L_{\delta}(V):= \{\xi: \textrm{Fr}(V)\rightarrow \mathbb{R}\ :\  \xi(e\cdot A)= \delta(A) \xi(e) \ \textrm{for\ all} \ e\in \textrm{Fr}(V),\ A\in GL_r\},\]
where $\textrm{Fr}(V)= \textrm{Isom}(\mathbb{R}^r, V)$ is the space of frames on $V$, endowed with the standard (right) action of $GL_r$. When $\delta= \textrm{det}$, we obtain the top exterior power $\Lambda^{\textrm{top}}V^*$, for $\delta= \textrm{sign}\circ \textrm{det}$, we obtain the \textbf{orientation space} $\mathfrak{o}_{V}$ of $V$, and with $\delta= |\textrm{det}|$ we obtain the \textbf{space $\mathcal{D}_{V}$ of densities} of $V$. More generally, for $l\in \mathbb{Z}$, $\delta= |\textrm{det}|^l$ defines the \textbf{space $\mathcal{D}_{V}^{l}$ of $l$-densities} of $V$.
There are canonical isomorphisms:
\[ \mathcal{D}_V\otimes \mathfrak{o}_{V}\cong \Lambda^\mathrm{top} V^*, \quad \mathcal{D}_{V}^{l_1}\otimes \mathcal{D}_{V}^{l_2}\cong \mathcal{D}_{V}^{l_1+ l_2}.\]
When $V= L$ is 1-dimensional, $D_{L}$ is also denoted by $|L|$; so, in general,
\[ \mathcal{D}_{V}= |\Lambda^\mathrm{top} V^*|,\]
which fits well with the fact that for any $\omega\in \Lambda^\mathrm{top} V^*$, $|\omega|$ makes sense as an element of $\mathcal{D}_{V}$.
For 1-dimensional vector spaces $W_1$ and $W_2$, one has a canonical isomorphism $|W_1| \otimes |W_2| \cong |W_1\otimes W_2|$
%\[ |W_1| \otimes |W_2| \cong |W_1\otimes W_2|\]
(in particular $|W^*|\cong |W|^*$). From the properties of $\Lambda^\mathrm{top}V^*$ (or by similar arguments), one obtains canonical isomorphisms:
\begin{enumerate}
\item[1.] $\mathcal{D}_{V}^{*}\cong \mathcal{D}_{V^*}$ for any vector space $V$.
\item[2.] For any short exact sequence of vector spaces
\[ 0\rightarrow V\rightarrow U\rightarrow W\rightarrow 0\]
(e.g. for $U= V\oplus W$), one has an induced isomorphism $\mathcal{D}_{U}\cong \mathcal{D}_{V}\otimes \mathcal{D}_{W}$.
\end{enumerate}

Since the previous discussion is canonical (free of choices), it can applied (fiberwise) to vector bundles over a manifold $M$ so that, for any such vector bundle $E$, one can talk about the associated line bundles over $M$
\[ \mathcal{D}_E,\ \Lambda^\mathrm{top} E^*,\ \mathfrak{o}_{E}\]
and the previous isomorphisms continue to hold at this level. However, at this stage, only $\mathcal{D}_E$ is trivializable, and even that is in a non-canonical way.

\begin{definition} A \textbf{density on a manifold} $M$ is any section of the density bundle $\mathcal{D}_{TM}$. We denote by $\mathcal{D}(M)$ the space of densities on $M$. %, and by $\mathcal{D}_{+}(M)$ the space of positive densities. 
\end{definition}

The main point about densities is that they can be integrated in a canonical fashion: one has an integration map
\begin{equation}\label{eq-can-int}
\int_{M}: C_{c}^{\infty}(M, \mathcal{D}_{TM}) \rightarrow \mathbb{R}.
\end{equation}
Locally, for densities $\rho$ supported in a coordinate chart $(x_1, \ldots, x_n)$, writing 
$\rho=  f |dx_1\ldots dx_n|$, the integral of $\rho$ is the usual integral of $f$. Of course, the integral is well-defined because of 
the change of variables formula 
\[ \int_{h(U)}f=\int_U f\circ h |Jac(h)| \]
and the fact that the coefficient $f$ of $\rho$ changes precisely to $ f\circ h |Jac(h)|$. In what follows, the integration (\ref{eq-can-int}) of compactly supported densities will be called \textbf{the canonical integration on $M$}.

\begin{definition}\label{geom-measures-maniflds}
The measure associated to a positive density $\rho\in \mathcal{D}(M)$ is:
\[ \mu_{\rho}: C_{c}^{\infty}(M)\rightarrow \mathbb{R},\ \ \mu_{\rho}(f)= \int_{M} f\cdot \rho .\]
Radon measures of this type will be called \textbf{geometric measures}. 
\end{definition}

%%%%%%%%%%%%%%%%%%%%%%%%%%%
%%%%%%%%%%%%%%%%%%%%%%%%%%%
%%%%%%%%%%%%%%%%%%%%%%%%%%%
%%%%%%%%%%%%%%%%%%%%%%%%%%%
%%%%%%%%%%%%%%%%%%%%%%%%%%%
\subsection{Some basic examples/constructions}\label{subsection-some basic examples}
%%%%%%%%%%%%%%%%%%%%%%%%%%%
%%%%%%%%%%%%%%%%%%%%%%%%%%%
%%%%%%%%%%%%%%%%%%%%%%%%%%%
%%%%%%%%%%%%%%%%%%%%%%%%%%%
%%%%%%%%%%%%%%%%%%%%%%%%%%%

\begin{example}[Haar measures and densities] \label{ex-Haar0} Let $G$ be a Lie group. Recall that a \textbf{right Haar measure} on $G$ is any non-zero measure on $G$ that is invariant under right translations (and similarly for left Haar measures). This notion makes sense for general  locally compact Hausdorff topological groups, and a highly non-trivial theorem says that such measures exist and are unique up to multiplication by scalars. However, for Lie groups the situation is much easier; for instance, for the existence, one can search among the geometric measures, i.e., the ones induced by densities $\rho$ on $G$. The invariance condition means that $\rho$ is obtained from its value at $1$, $\rho_1\in \mathcal{D}_{\mathfrak{g}}$ (where $\mathfrak{g}$ is the Lie algebra of $G$), by right translation. Hence Haar densities on $G$ are in 1-1 correspondence with non-zero-elements of the $1$-dimensional vector space of densities on $\mathfrak{g}$. 

Recall also that, while the right and left Haar measures are in general different, for compact groups, they coincide. Actually, in the compact case, it follows that there exists a unique right (and left) Haar measure on $G$ for which $G$ has volume $1$. This is called \textbf{the Haar measure}  of the compact group $G$. Of course, to obtain it, one just starts with any non-zero density $\rho_{1}$ on $\mathfrak{g}$ and one rescales it by the resulting volume of $G$. Note that, due to this normalization condition, when dealing with a non-connected compact Lie group $G$, the Haar density of $G$ differs from the one of the identity component $G^{0}$ by a factor which is the number of connected components of $G$. We refer to \cite[Sec.\ 3.13]{DK} for a textbook account on Haar measures and densities on Lie groups.

With the basic understanding that when passing from groups to groupoids the source-fibers of the groupoid are used for making sense of right translations, there is a rather obvious generalization of the notion of Haar measures/densities to the world of groupoids - called Haar systems/densities. Although this is well known, some of the basic facts are hard to find or have been overlooked. They are collected in the (self-contained) Appendix. 
\end{example}

\begin{example}[push-forward measures/densities] \rm \ 
If $\pi: P\rightarrow B$ is a proper map then measures $I$ on $P$ can be pushed-down to measures $\pi_{!}(I)$ on $B$ by:
\[ \pi_{!}(I) (f)= I(f\circ \pi),\]
or, in the integral notation (cf. Remark \ref{notation-integrals-0}),
\[ \int_{B} f(b)\ d\mu_{\pi_{!}(I)}(b):= \int_{P} f(\pi(p))\ d\mu_{I}(p).\]
If $\pi$ is a submersion, then the construction $I\mapsto \pi_!(I)$ takes geometric measures on $P$ to geometric measures on $B$. The key remark is that integration over the fibers makes sense (canonically!) for densities, to give a map 
\[ \pi_{!}= \int_{\textrm{fibers}}: \mathcal{D}(P)\rightarrow \mathcal{D}(B).\]
Here is the definition of $\pi_{!}(\rho)(b)$ for a density $\rho$ on $P$ and for $b\in B$: restrict $\rho$ to the fiber $P_b= \pi^{-1}(b)$; one finds for each $p\in P_b$:
\[ \rho(p)\in \mathcal{D}_{T_pP} \cong \mathcal{D}_{T_p P_b}\otimes \mathcal{D}_{T_bB},\]
where the last isomorphism is the one induced by the short exact sequence 
\[ 0 \rightarrow T_pP_b  \rightarrow T_pP  \rightarrow T_bB \rightarrow  0.\]
Hence we can interpret
\begin{equation}\label{dens-bdle-dec} 
\rho|_{P_b}\in \mathcal{D}(P_b)\otimes \mathcal{D}_{T_bB}; 
\end{equation}
integrating over $P_b$ (compact because $\pi$ is proper) we find the desired element
\[ \pi_{!}(\rho)(b):= \int_{P_b} \rho \in \mathcal{D}_{T_bB}.\]
The fact that this operation is compatible with the one on measures, i.e., that $\pi_{!}(I_{\rho})= I_{\pi_{!}(\rho)}$, follows from the Fubini formula for densities:
\[ \int_{P} \rho = \int_{B} \pi_{!}(\rho) \ \ \left(= \int_{B} \int_{\textrm{fibers}} \rho\right)\]
% (which follows in turn from the standard Fubini formula on Euclidean spaces and the fact that $\pi$ is locally a projection). 
\end{example}

\begin{example}\label{inv-meas}(invariant measures/densities) The previous discussion can be continued further in the case of a (right) principal $G$-bundle
\[ \pi: P\rightarrow B\]
for a compact Lie group $G$ (cf. \cite[Sec.\ 3.13]{DK}). One of the main outcomes will be that the construction  $I \mapsto \pi_{!}(I)$ induces a bijection 
\[ \{\textrm{invariant\ measures/densities\ on}\ P\} \stackrel{\sim}{\longleftrightarrow} \{\textrm{measures/densities\ on}\ B\}.\]
We concentrate here on densities. What is special in this case is that each fiber $P_b$ ($b\in B$) carries a canonical ``Haar density'' $\rho_{\mathrm{Haar}}^{b}$: to see this one 
chooses $p\in P_b$ and one uses the diffeomorphism $m_p: G\rightarrow P_{b}$, $g\mapsto pg$, to transport $\rho_{\mathrm{Haar}}^{G}$ to a density on $P_b$; 
the independence of the choice of $p$ follows from the bi-invariance of $\rho_{\mathrm{Haar}}^{G}$. 
% the bi-invariance of $\rho_{\mathrm{Haar}}^{G}$ implies that the result does not depend on the choice of $p\in P_b$.
Using $\rho_{\mathrm{Haar}}^{b}$, arbitrary densities on $P$ can be decomposed, via  (\ref{dens-bdle-dec}), as
\[ \rho(p)= \rho_{\mathrm{Haar}}^{b}\otimes \rho_B(p) \ \ \textrm{for\ some\ }\ \rho_{B}(p)\in \mathcal{D}_{T_bB}.\]
A simple check shows that $\rho$ is invariant if and only if $\rho_{B}(p)$ does not depend on $p$ but only on $b= \pi(p)$. Of course, in this case $\rho_{B}= \pi_{!}(\rho)$ (because of the normalization of the Haar density of $G$). Therefore one obtains an isomorphism:
\[ \pi_{!}: \mathcal{D}(P)^{G} \cong \mathcal{D}(B).\]
For $\rho_{B}\in \mathcal{D}(B)$, we denote by $\rho_{\mathrm{Haar}}\otimes \rho_{B}$ 
the induced invariant density on $P$. 
\end{example}

%%%%%%%%%%%%%%%%%%%%%%%%%%%
%%%%%%%%%%%%%%%%%%%%%%%%%%%
%%%%%%%%%%%%%%%%%%%%%%%%%%%
%%%%%%%%%%%%%%%%%%%%%%%%%%%
%%%%%%%%%%%%%%%%%%%%%%%%%%%
%%%%%%%%%%%%%%%%%%%%%%%%%%%
%%%%%%%%%%%%%%%%%%%%%%%%%%%
%%%%%%%%%%%%%%%%%%%%%%%%%%%
%%%%%%%%%%%%%%%%%%%%%%%%%%%
%%%%%%%%%%%%%%%%%%%%%%%%%%%
\section{Transverse measures}
\label{Transverse measures}
%%%%%%%%%%%%%%%%%%%%%%%%%%%
%%%%%%%%%%%%%%%%%%%%%%%%%%%
%%%%%%%%%%%%%%%%%%%%%%%%%%%
%%%%%%%%%%%%%%%%%%%%%%%%%%%
%%%%%%%%%%%%%%%%%%%%%%%%%%%
%%%%%%%%%%%%%%%%%%%%%%%%%%%
%%%%%%%%%%%%%%%%%%%%%%%%%%%
%%%%%%%%%%%%%%%%%%%%%%%%%%%
%%%%%%%%%%%%%%%%%%%%%%%%%%%
%%%%%%%%%%%%%%%%%%%%%%%%%%%

In this section we introduce the notion of transverse measures for groupoids using a rather straightforward generalization of Haefliger's interpretation of transverse measures for foliations (and for \'etale groupoids). 

%%%%%%%%%%%%%%%%%%%%%%%%%%%
%%%%%%%%%%%%%%%%%%%%%%%%%%%
%%%%%%%%%%%%%%%%%%%%%%%%%%%
%%%%%%%%%%%%%%%%%%%%%%%%%%%
%%%%%%%%%%%%%%%%%%%%%%%%%%%
\subsection{Making sense of ``orbit spaces''; transverse objects}
\label{Making sense of ``orbit spaces''; transverse objects}
%%%%%%%%%%%%%%%%%%%%%%%%%%%
%%%%%%%%%%%%%%%%%%%%%%%%%%%
%%%%%%%%%%%%%%%%%%%%%%%%%%%
%%%%%%%%%%%%%%%%%%%%%%%%%%%
%%%%%%%%%%%%%%%%%%%%%%%%%%%

Intuitively, given a groupoid $\G$ over $M$, transverse structures on $\G$ are structures that are intrinsic to the geometry of $\G$  
that lives in the direction transverse to the orbits of $\G$; more suggestively (but less precisely), one may think that they are structures associated to the orbit space
\[ M/\G := M/ (x\sim y \ \ \textrm{iff}\ \ \exists \ g\in \G \ \textrm{from}\ x \ \textrm{to} \ y). \]
Of course, the actual quotient {\it topological space} $M/\G$ may be very pathological and uninteresting; for that reason, when one refers to the ``orbit {\it space}'' one often has in mind much more than just the topological space itself. There are several points of view that allow one to make sense of ``singular spaces'' (like $M/\G$) in a satisfactory but precise way. We recall here a few.\\

\textbf{Leaf spaces as \'etale groupoids:} This is Haefliger's approach to the study of the transverse geometry of foliated spaces $(M, \mathcal{F})$, giving a satisfactory meaning to ``the spaces of leaves $M/\mathcal{F}$'' \cite{haefliger,haefliger2}. The first point is that any foliated manifold $(M, \F)$ has an associated \textbf{holonomy groupoid} $\textrm{Hol}(M, \F)$; it is a groupoid over $M$ whose arrows are determined by germs of holonomy transformations.
% ; in particular, two points in $M$ are connected by an arrow if and only if they belong to the same leaf. In this way 
In this way the leaf space $M/\F$ is realised as the orbit space of a groupoid, and one may think that 
$\textrm{Hol}(M, \F)$ represents (as some kind of ``desingularization'') $M/\F$. The second point is that this representative can be simplified by restricting to a complete transversal $T\subset M$ of the foliation:
\[ \textrm{Hol}_{T}(M, \F):= \textrm{Hol}(M, \F)|_{T}\]
will have the same orbit space (at least intuitively) and has a rather special extra property: it is \'etale, in the sense that its source and target maps are local diffeomorphisms. The last property make it possible to handle \'etale groupoids (and their orbit spaces) very much like one handles usual manifolds (think of a manifold $M$ as the \'etale groupoid over $M$ with only identity arrows).

To make the entire story precise, one also has to give a precise meaning to ``two groupoids give rise to (or model) the same orbit space''. This is precisely what the notion of Morita equivalence of groupoids does; the basic example is the holonomy groupoid $\textrm{Hol}(M, \F)$ of a foliation being Morita equivalent to $\textrm{Hol}_{T}(M, \F)$; hence, staying within the world of \'etale groupoids, Haefliger's philosophy is: the leaf space $M/\F$ is represented by an \'etale groupoid (namely $\textrm{Hol}_{T}(M, \F)$ for some complete transversal $T$), which is well-defined up to Morita equivalence.\\

\textbf{Differentiable stacks:}
Stacks originate in algebraic geometry, where they are used to model moduli spaces that are not well-defined otherwise (note the similarity with leaf spaces). The topological and smooth versions of the theory 
were studied in more detail only later (see for example \cite{behrend_xu,heinloth,metzler}), and it was immediately noticed that the resulting notion of ``differentiable stack'' can be represented by Lie groupoids, two Lie groupoids representing the same stack if and only if they are Morita equivalent. Actually, a large part of the existing literature views (by definition) differentiable stacks as Morita equivalence classes of Lie groupoids; the Morita class (stack) of a Lie groupoid $\G$ over $M$ is then suggestively denoted by $M//\G$. In our opinion, this is a rather  obvious extension of Haefliger's philosophy from \'etale to general Lie groupoids (but, sadly enough, Haefliger is often forgotten by the literature on differentiable stacks. %, despite the fact that relationships with the algebro-geometric point of view were made explicit already in \cite{Moerd-Top}). 
In this paper we use the groupoid point of view on stacks; with these in mind, what we do here is again rather straightforward: just extend Haefliger's approach to transverse integration from \'etale groupoids to general Lie groupoids.\\

\textbf{Non-commutative spaces:} Another approach to ``singular spaces'' is provided by Connes' non-commutative geometry \cite{Connes}; while standard spaces (e.g. compact Hausdorff spaces) are fully characterized by their (commutative) algebra of continuous scalar-valued functions, the idea is that non-commutative algebras (possibly with extra structure) should be interpreted as algebras of functions on a ``non-commutative space''. Therefore, when one deals with a singular space $X$, one does not look at its points or at its topological/smooth structure (usually ill-behaved) but one tries to model it via a non-commutative algebra. In general, this modelling step is not precisely defined  and it very much depends on the specific $X$ one looks at. However, in most examples, one uses groupoids as an intermediate step.

More precisely: one first realizes $X$ as the orbit space of a Lie groupoid $\G$ (i.e., one makes sense of $X$ as a differentiable stack) and then one appeals to a standard construction that associates to a Lie groupoid $\G$ over $M$ a (usually non-commutative) algebra, namely the \textbf{convolution algebra} of the groupoid. Before we recall its definition let us mention that, intuitively, the convolution algebra of $\G$ should be thought of as a ``non-commutative model'' for the algebra of (compactly supported) functions on the singular orbit space $M/\G$ (or better: of the stack $M//\G$). For that reason, we will denote the convolution algebra by $\mathcal{N}\mathcal{C}_{c}^{\infty}(M//\G)$.  
 As a vector space, it is simply $C_{c}^{\infty}(\G)$ - the space of compactly supported smooth functions on $\G$. The convolution product is, in principle, given by a convolution formula
\begin{equation}\label{int-conv-prod} 
`` (u_1\star u_2)(g)= \int_{g_1g_2= g} u_1(g_1)u_2(g_2) '' .
\end{equation}
The simplest case when this makes sense is when $\G$ is \'etale, case in which 
\[ (u_1\star u_2)(g)= \sum_{g_1g_2= g} u_1(g_1)u_2(g_2).\]
For general Lie groupoids, it is customary to choose a full Haar system (see Definition \ref{def-Haar-system}) or, even better, to take advantage of the smooth structure and start with a full Haar density $\rho\in C^\infty(M,\mathcal{D}_A)$ (Haar densities are recalled in the appendix - see Definition \ref{def-Haar-density}). Using the induced integration (of functions) along the $s$-fibers, one can now make precise sense of (\ref{int-conv-prod}) as:
\begin{equation}
\label{star-rho} 
(u_1\star_{\rho} u_2)(g)= \int_{s^{-1}(x)} u_1(gh^{-1})u_2(h)\  d\mu_{\stackrel{\rightarrow}{\rho}}(h) \ \ \ \ \textrm{where}\ x= s(g).
\end{equation}
It is not difficult to see that the choice of the full Haar density does not affect the isomorphism class of the resulting convolution algebra.
However, at the price of becoming a bit more abstract, but keeping (\ref{int-conv-prod}) as intuition, one can proceed intrinsically as follows (cf. e.g. \cite{Connes}): consider the bundle $\mathcal{D}_{A}^{1/2}$ of half densities, pull it back to $\G$ via $s$ and $t$ and define
\begin{equation}\label{CcG-intr} 
\mathcal{N}\mathcal{C}_{c}^{\infty}(M//\G):= C_{c}^{\infty}(\G, t^*\mathcal{D}_{A}^{1/2}\otimes s^*\mathcal{D}_{A}^{1/2}).
\end{equation}
Given $u_1$ and $u_2$ in this space, we look at the expression (\ref{int-conv-prod}). For 
\[ x\stackrel{g_1}{\leftarrow} z  \stackrel{g_2}{\leftarrow} y\]
such that $g= g_1g_2$, $u(g_1)u(g_2)$ makes sense as a tensor product; since 
\[ \mathcal{D}_{A, z}^{1/2} \otimes \mathcal{D}_{A, z}^{1/2}= \mathcal{D}_{A, z}\cong \mathcal{D}_{T_{g_2}(s^{-1}(y))},\]
(where we use the right translations for the last identification), we see that 
\begin{equation}
\label{ug1-ug2} 
u(g_1)v(g_2)\in   \mathcal{D}_{A, x}^{1/2}\otimes \mathcal{D}_{T_{g_2}(s^{-1}(y))} \otimes \mathcal{D}_{A, y}^{1/2};
\end{equation}
hence, when $g= g_1g_2$ is fixed, we deal with a density in the $g_2\in s^{-1}(y)$ argument; integrated, this gives rise to 
\[ (u_1\star u_2)(g) \in \mathcal{D}_{A, x}^{1/2}\otimes \mathcal{D}_{A, y}^{1/2},\]
hence to a new element $u_1\star u_2$ in our space (\ref{CcG-intr}).

%%%%%%%%%%%%%%%%%%%%%%%%%%%
%%%%%%%%%%%%%%%%%%%%%%%%%%%
%%%%%%%%%%%%%%%%%%%%%%%%%%%
%%%%%%%%%%%%%%%%%%%%%%%%%%%
%%%%%%%%%%%%%%%%%%%%%%%%%%%
\subsection{Transverse integration \`a la Haefliger}\label{subsec-transverse integration Haefliger}
%%%%%%%%%%%%%%%%%%%%%%%%%%%
%%%%%%%%%%%%%%%%%%%%%%%%%%%
%%%%%%%%%%%%%%%%%%%%%%%%%%%
%%%%%%%%%%%%%%%%%%%%%%%%%%%
%%%%%%%%%%%%%%%%%%%%%%%%%%%
The notion of transverse measure and transverse integration is best understood in the case of foliated manifolds $(M, \F)$; the idea is that such a measure 
 should measure the size of transversals to the foliation in a way that is invariant under holonomy transformations  (so that, morally, they measure not the transversals but the subspaces that they induce in the leaf space). This idea can be implemented either by working directly with set-measures, as done e.g. by Plante \cite{Plante}, or by working with the dual picture (via the Riesz theorem) in terms of linear functionals, as done by Haefliger \cite{haefliger-minimal}. 

Here we follow Haefliger's approach, but applied to a general Lie groupoid $\G$ over a manifold $M$. The main point is to define  what deserves to be called ``the space of compactly supported smooth functions on the orbit space $M/\G$''. It is instructive to first think about the meaning of ``the space of smooth functions on the orbit space $M/\G$''; indeed, there is an obvious candidate, namely
\[ C^{\infty}(M//\G):= C^{\infty}(M)^{\G-\textrm{inv}},\]
the space of smooth functions on $M$ that are $\G$-invariant (i.e., constant on the orbits of $\G$). In this way, intuitively, a ``smooth function on $M/\G$'' is represented by its pull-back to $M$. One may try to add compact supports to the previous discussion (to define  $C^{\infty}_{c}(M//\G)$), but one encounters a serious problem: smooth invariant functions may fail to have compact support (this happens already for actions of groups). Moreover, this approach would be too naive; the point is that the theory with compact supports should not be seen just as a subtheory of the one without support conditions, but as a dual theory (think e.g. of DeRham cohomology and its version with compact supports, which are related by Poincare duality).

\begin{example}[A simple example that illustrates the general philosophy]\label{exemplification}\rm \ 
Probably the best example is that of an action groupoid $\G= \Gamma\ltimes M$ associated to an action of a discrete group $\Gamma$ on a manifold $M$. As for smooth functions, there is an obvious approach to  ``measures on the orbit space'': measures on $M$, $\mu: C_{c}^{\infty}(M)\rightarrow \mathbb{R}$, which are $\Gamma$-invariant in the sense that
\[ \mu(\gamma^*(f))= \mu(f) \ \ \ \forall\ f\in C^{\infty}_{c}(M), \ \gamma\in \Gamma,\]
where $\gamma^{*}$ denotes the induced action on smooth functions ($\gamma^*(u)(x)= u(\gamma\cdot x)$). Therefore, if we want to represent such measures as linear functionals, we have to consider the quotient of $C_{c}^{\infty}(M)$ by the linear span of elements of type $f- \gamma^*(f)$. This is a construction that is defined more generally, for actions of $\Gamma$ on a vector space $V$; the resulting quotient is the space $V_{\Gamma}$ of co-invariants (dual to the space of invariants $V^{\Gamma}$ - cf. e.g. \cite{weibel}). Therefore we arrive 
 at the following model for the space of compactly supported smooth functions on the orbit space:
\begin{equation}\label{Cc-discrete-group-actions} 
C_{c}^{\infty}(M//\Gamma):= C_{c}^{\infty}(M)_{\Gamma},
\end{equation}
and invariant measures correspond to (positive) linear functionals on $C_{c}^{\infty}(M//\Gamma)$.

When the action is free and proper, then $B= M/\Gamma$ is itself a manifold, and the integration over the fibers of the canonical projection $\pi$,
\[ \pi_{!}: C_{c}^{\infty}(M)\rightarrow C_{c}^{\infty}(B), \ \ \pi_{!}(u)(b)= \sum_{x\in \pi^{-1}(b)} u(x),\]
descends to an isomorphism of our model (\ref{Cc-discrete-group-actions}) with $C_{c}^{\infty}(B)$. In contrast, the invariant part of $C_{c}^{\infty}(M)$ is trivial if $\Gamma$ is infinite.  

For general actions, $C_{c}^{\infty}(M)^{\Gamma}$ has a nicer description if $\Gamma$ is finite; actually, in that case, $V^{\Gamma}\cong V_{\Gamma}$ for any $\Gamma$-vector space $V$
(where the isomorphism is induced by the canonical projection and the inverse by averaging). Also, although $M/\Gamma$ may fail to be a manifold, it is a locally compact Hausdorff space, and there is a satisfactory definition for the spaces of smooth functions on $M/\Gamma$: 
\[ C^{\infty}(M/\Gamma):= \{ f\in C(M/\Gamma) : f\circ \pi\ \textrm{is\ smooth}\}, \]
\[ C_{c}^{\infty}(M/\Gamma):= \{f\in C^{\infty}(M/\Gamma) : \textrm{supp}(f)-\textrm{compact}\}.\]
We deduce that, if $\Gamma$ is finite, 
\[ C_{c}^{\infty}(M//\Gamma)\cong C_{c}^{\infty}(M)^{\Gamma} \cong C_{c}^{\infty}(M/\Gamma),\]
hence the elements of $C_{c}^{\infty}(M//\Gamma)$ can be seen as compactly supported functions on $M/\Gamma$. 
This phenomena is at the heart of our separate section on proper groupoids. 
\end{example}

\begin{example}[\'Etale groupoids; leaf spaces]\rm \ The definition (\ref{Cc-discrete-group-actions}) (and the heuristics behind it) easily extend to general \'etale groupoids $\G$ over $M$. One defines
% one defines 
\begin{equation}\label{C-non-c-Hefliger}
C^{\infty}(M//\G):= C^{\infty}(M)^{\G-\textrm{inv}},
\end{equation}
% the space of smooth functions on $M$ that are $\G$-invariant (i.e. constant on the orbits of $\G$). 
or, equivalently but easier to dualise,
\[  C^{\infty}(M//\G)= \textrm{Ker}(s^*- t^*: C^{\infty}(M)\rightarrow C^{\infty}(\G)),\]
where $s^*(u)= u\circ s$ and similarly for $t^*$. Dual to $s^*$ we have:
\[ % \begin{equation} \label{s-int-fiber}
s_{!}: C_{c}^{\infty}(\G) \rightarrow C_{c}^{\infty}(M), \ \ s_{!}(u)(x)= \sum_{s(g)= x} u(g)
% \end{equation}
\]
and similarly $t_{!}$ and then one defines 
\begin{equation}\label{Cc-Hefliger}
C^{\infty}_{c}(M//\G):= \textrm{Coker} (s_{!}- t_{!}: C_{c}^{\infty}(\G) \rightarrow C_{c}^{\infty}(M)).
\end{equation}
This space, as well as (\ref{C-non-c-Hefliger}), is invariant under Morita equivalences; this is implicit in Haefliger's work and will be discussed in full generality later in the paper.

This construction can be applied to the \'etale holonomy groupoids associated to a foliation $(M, \F)$: for any complete transversal $T$ one considers $C_{c}^{\infty}(T//\F):=  C_{c}^{\infty}(T//\textrm{Hol}_{T}(M, \F))$.  These are precisely the spaces denoted $\Omega^{0}_{c}(T/\F)$ by Haefliger \cite{haefliger-minimal} and used to handle transverse integration; our definition is basically the same as his, just that we use the groupoids explicitly. Morita invariance shows that, up to canonical isomorphisms, $C_{c}^{\infty}(T//\F)$ does not depend on the choice of $T$; it serves as a model for ``$C_{c}^{\infty}(M/\F)$''. Of course, transverse measures are now understood as positive linear functionals on $C_{c}^{\infty}(T//\F)$. One point that is not clarified by Haefliger and which may serve as motivation for extending the theory to more general groupoids is whether there is an intrinsic description of ``$C_{c}^{\infty}(M/\F)$'' defined directly using objects that live on $M$.
\end{example}

%%%%%%%%%%%%%%%%%%%%%%%%%%%
%%%%%%%%%%%%%%%%%%%%%%%%%%%
%%%%%%%%%%%%%%%%%%%%%%%%%%%
%%%%%%%%%%%%%%%%%%%%%%%%%%%
%%%%%%%%%%%%%%%%%%%%%%%%%%%
\subsection{Transverse measures: the general case}\label{subsec-Transverse measures: the general case}
%%%%%%%%%%%%%%%%%%%%%%%%%%%
%%%%%%%%%%%%%%%%%%%%%%%%%%%
%%%%%%%%%%%%%%%%%%%%%%%%%%%
%%%%%%%%%%%%%%%%%%%%%%%%%%%
%%%%%%%%%%%%%%%%%%%%%%%%%%%

To extend the previous discussion to general Lie groupoids $\G$ over $M$ one only needs to make sense of the integration over the $s$ and $t$-fibers of $\G$. This issue is identical to the one we encountered when making sense of the convolution product (the last part of Subsection \ref{Making sense of ``orbit spaces''; transverse objects}); so, it is not surprising that one can proceed as there: either fix a strictly positive density $\rho\in C^{\infty}(M, \mathcal{D}_A)$ and define a version of $C_{c}^{\infty}(M//\G)$ that depends on $\rho$ but whose isomorphism class does not depend on $\rho$, or provide a more abstract but intrinsic version of $C_{c}^{\infty}(M//\G)$. We prefer to start with the choice-free approach. 

Hence, as in Subsection \ref{Making sense of ``orbit spaces''; transverse objects}, we use the Lie algebroid $A$ of $\cG$, the associated density bundle $\mathcal{D}_A$ and its pull-backs to $\G$ via the source and target maps.

One has a canonical integration over the $s$-fibres map
\begin{equation}\label{S-!} 
s_{!}: C_{c}^{\infty}(\G, t^*\mathcal{D}_A\otimes s^* \mathcal{D}_A)\rightarrow C_{c}^{\infty}(M, \mathcal{D}_A);
\end{equation}
indeed, for $u$ in the left hand side, its restriction to an $s$-fiber $s^{-1}(x)$ is a section of 
\[ ( t^{*} \mathcal{D}_A)|_{s^{-1}(x)}\otimes \mathcal{D}_{A, x}\cong  \mathcal{D}_{T(s^{-1}(x))}\otimes \mathcal{D}_{A, x}.\]
This allows us to integrate $u|_{s^{-1}(x)}$ and obtain an element in $\mathcal{D}_{A, x}$; varying $x$, this gives rise to $s_{!}(u)$. 
A similar reasoning makes sense of $t_{!}$. 

\begin{definition} The \textbf{intrinsic model for compactly supported smooth functions} on the orbit space is defined as
\[ C_{c}^{\infty}(M//\G):=  C_{c}^{\infty}(M, \mathcal{D}_A)/\textrm{Im}(s_{!}- t_{!}).\]
A \textbf{transverse measure for} $\G$ is 
a positive linear functional on $C_{c}^{\infty}(M//\G)$ or, equivalently, a linear map
\[ \mu: C_{c}^{\infty}(M, \mathcal{D}_A) \rightarrow \mathbb{R}\]
which is positive and which satisfies the invariance condition
\begin{equation}\label{transv-meas-invariance}
\mu\circ s_{!}= \mu\circ t_{!} .
\end{equation}
\end{definition}

\begin{remark}[some intuition]\label{cpct-supp-intuition}\rm \ 
The intuition is that elements $\rho \in C_{c}^{\infty}(M, \mathcal{D}_A)$ represent ``compactly supported smooth functions on $M/\G$'' as follows: while sections $\rho$ of $\cD_A$ can be interpreted as invariant family of densities $\{\rho^x\}$ (as in the Appendix - see (\ref{dens-ind-s-fiber})), their integrals (when defined) will give an invariant function of $x$, hence a function on $M/\G$: 
\[ \textrm{Av}(\rho): M/\G\rightarrow \mathbb{R}, \  \textrm{Av}(\rho)(\mathcal{O}_x)= \int_{s^{-1}(x)} \rho^x .\]
% \overrightarrow{\rho}|_{s^{-1}(x)}.\]
Of course, there are problems to make this precise, and that is the reason for working on $M$. 
Nevertheless, we will see that such problems will disappear in one important case: for proper Lie groupoids. That is the subject of Section \ref{sec-proper}. 
\end{remark}

\begin{remark}[the relationship with the non-commutative approach] It is instructive to compare our more classical approach with the non-commutative one (see Subsection \ref{Making sense of ``orbit spaces''; transverse objects}). In the non-commutative approach one  looks at linear maps $\tau$ on $\mathcal{NC}_{c}^{\infty}(M//\G)$ (see (\ref{CcG-intr})) which are \textbf{traces}, i.e., which satisfy:
\begin{equation}\label{trace-condition}
\tau(u\star v)= \tau(v \star u), \ \ \forall\ u, v\in \mathcal{NC}_{c}^{\infty}(M//\G).
\end{equation}
A general well-known phenomena is that classical approaches are recovered as ``localization at units'' of the non-commutative ones (see e.g. \cite{BN}). Hence, one expects that our notion of transverse measure induces, by restricting to units, such traces. And, indeed, this fits perfectly with our approach (and even with our use of $\mathcal{D}_A$). First of all, for $u\in \mathcal{NC}_{c}^{\infty}(M//\G)$, its restriction to units gives elements
\[ u(1_x)\in \mathcal{D}_{A, x}^{1/2} \otimes \mathcal{D}_{A, x}^{1/2} = \mathcal{D}_{A, x},\]
so that $u|_{M}\in C_{c}^{\infty}(M, \mathcal{D}_A)$. Hence any linear $\mu$ as before induces
\[ \tilde{\mu}: \mathcal{NC}_{c}^{\infty}(M//\G) \rightarrow \mathbb{R}, \ u\mapsto \mu(u|_{M}).\]
\end{remark}

\begin{proposition}\label{equiv-to-trace} $\mu$ satisfies the invariance condition (\ref{transv-meas-invariance}) if and only if $\tilde{\mu}$ satisfies the tracial condition (\ref{trace-condition}).
\end{proposition}

\begin{proof} Applying the definition of $u\star v$ at points $1_x$ and comparing the resulting formula with the one defining $s_{!}$ we find that 
\[ (u\star v)|_{M}= s_{!}(\phi),\]
where $\phi(g)= u(g^{-1})v(g)$ defines an element in $C_{c}^{\infty}(\G, t^*\mathcal{D}_A\otimes s^*\mathcal{D}_A)$.
Similarly, or as a consequence, we have $(v\star u)|_{M}= t_{!}(\phi)$. This explains the statement (strictly speaking, one has to check that any compactly supported $\phi$ can be written as $\phi(g)= u(g^{-1})v(g)$ for some compactly supported $u$ and $v$. Trivializing the density bundles this becomes a question about functions, which is clear: just take $v= \phi$ and $u\in C_{c}^{\infty}(\G)$ any function which is $1$ on the inverse of the support of $\phi$.)
\end{proof} 

\begin{remark}[the more down to earth approach]\label{remark-s-shriek}  \rm \ If we choose a strictly positive density $\rho\in C^\infty(M, \mathcal{D}_A)$ (see the Appendix), then one can use $\rho$ to trivialize all the density bundles. Then $s_{!}$ becomes identified with 
 \[ s_{!}^{\rho}: C^{\infty}_{c}(\G) \rightarrow C_{c}^{\infty}(M), \ \  s_{!}^{\rho}(u)(x)=\int_{s^{-1}(x)} u(g) \stackrel{\rightarrow}{\rho}(g)\]
(for $\stackrel{\rightarrow}{\rho}$, see (\ref{dens-ind-s-fiber}) in the Appendix); similarly for $t_{!}$, hence $C_{c}^{\infty}(M//\G)$ is identified with its more concrete (but $\rho$-dependent)  model 
\[ C_{c}^{\infty}(M//\G, \rho):= \textrm{Coker} (s_{!}^{\rho}- t_{!}^{\rho}: C_{c}^{\infty}(\G) \rightarrow C_{c}^{\infty}(M)).\]
\end{remark}

Let us state right away one of the most basic properties of our definition: Morita invariance. This will follow from a more general result - Theorem \ref{thm-Mor-inv}.

\begin{theorem}\label{theorem-ME} Any Morita equivalence between two groupoids $\G$ (over $M$) and $\H$ (over $N$) gives rise to an  isomorphism between $C_{c}^{\infty}(M//\G)$ and $C_{c}^{\infty}(N//\H)$ and induces a 1-1 correspondence between the transverse measures of $\G$ and of $\H$.
\end{theorem}

\begin{example}[the case of submersions]\label{ex-case-submersions}\rm \ When looking at ``transverse notions'', i.e.,  notions that morally live on the orbit space, the ``test case'' is provided by the groupoids whose orbit spaces $B$ are already smooth or, more precisely, by the groupoids that are (Morita) equivalent to smooth manifolds $B$. Such groupoids are associated to smooth submersions $\pi: P\rightarrow B$. Explicitly, any such submersion $\pi$ gives rise to a Lie groupoid $\G(\pi)$: 
\[ \G(\pi)= P\times_{B} P= \{ (p, q)\in P\times P\ :\ \pi(p)= \pi(q)\},\]
with the source and target being the second, respectively the first, projection, and the multiplication $(p, q)\cdot (q, r)= (p, r)$. The Lie algebroid of $\G(\pi)$ is the sub-bundle of $TP$ consisting of vectors that are tangent to the fibers of $\pi$, 
\[\F(\pi):= \textrm{Ker}(d\pi)\subset TP.\]
In this case the intuition mentioned in Remark \ref{cpct-supp-intuition} works without problems and tells us that we should look at the fiber integration map 
\begin{equation}
\label{int-over-fiber-again}
\int_{\pi-\textrm{fibers}}: C_{c}^{\infty}(P, \mathcal{D}_{\F(\pi)})\rightarrow C_{c}^{\infty}(B),
\end{equation}
Even more, it motivates the use of the density bundle starting with the question: how can one represent 
compactly supported smooth functions on $B$ by functions on $P$, in a canonical way? Of course, the main problem is then to understand the kernel of this map. The definition of $C_{c}^{\infty}(M//\G)$ suggests the answer: look at  
\[ s_{!}-t_{!}% = \int_{s-\textrm{fibers}}- \int_{t-\textrm{fibers}}
: C_{c}^{\infty}(P\times_{B} P, t^{*}\mathcal{D}_{\F(\pi)}\otimes s^{*}\mathcal{D}_{\F(\pi)})\rightarrow C_{c}^{\infty}(P, \mathcal{D}_{\F(\pi)}).\]

\begin{lemma}\label{conf1} For the groupoid $\G(\pi)$ associated to a submersion $\pi: P\rightarrow B$, the integration over the fiber (\ref{int-over-fiber-again}) is surjective and its kernel is the image of $s_{!}-t_{!}$; therefore it induces an isomorphism
\[ C_{c}^{\infty}(P//\G(\pi)) \cong C_{c}^{\infty}(B) .\]
\end{lemma}

\begin{proof} Let $C^{\infty}_{\pi-c}(P, \mathcal{D}_{\F(\pi)})$ be the space of sections of $\mathcal{D}_{\F(\pi)}$ with fiberwise compact supports. We claim that the integration over the fiber, now viewed as a map
\[ \int_{\pi}:= \int_{\pi-\textrm{fibers}}: C^{\infty}_{\pi-c}(P, \mathcal{D}_{\F(\pi)}) \rightarrow C^{\infty}(B),\]
is surjective. Since $\int_{\pi} \pi^*(f) u= f \int_{\pi} u$, it suffices to show that there exists:
\[ c\in C^{\infty}_{\pi-c}(P, \mathcal{D}_{\F(\pi)}) \ \ \textrm{such\ that}\ \ \int_{\pi} c= 1.\]
To see this, choose an open cover $\{V_i\}_{i\in I}$ of $B$ such that, for each $i$, there exists an open $U_i\subset P$ with the property that $\pi|_{U_i}$ is diffeomorphic to the projection $V_i\times \mathbb{R}^q\rightarrow V_i$. On each $U_i$ choose 
\[ c_i\in C^{\infty}_{\pi|_{U_i}-c}(U_i, \mathcal{D}_{\F(\pi)})\]
such that $\int_{\pi|_{U_i}} c_i= 1$. Choose also a partition of unity $\{\eta_i\}_{i\in I}$ with $\eta_i$ supported in $V_i$. Set now $c= \sum_{i} \eta_i c_i$. 

The surjectivity from the statement follows immediately using $c$: for $f\in C_{c}^{\infty}(B)$, $\pi^*(f)\cdot c$ has compact support and $\int_{\pi} \pi^*(f)\cdot c= f$. Let now $u\in C_{c}^{\infty}(P, \mathcal{D}_{\F(\pi)})$ such that $\int_{\pi-\textrm{fibers}} u= 0$. Construct
\[ v\in  C_{c}^{\infty}(P\times_{B} P, t^{*}\mathcal{D}_{\F(\pi)}\otimes s^{*}\mathcal{D}_{\F(\pi)}), \ v(p, q)= \phi(\pi(q))  u(p)\otimes c(q) \]
where $\phi\in C_{c}^{\infty}(B)$ is chosen to be $1$ on $\pi(\textrm{supp}(u))$. The role of $\phi$ is to make $v$ have compact support. Compute now:
\[ s_{!}(v)(p)= \phi(\pi(b)) c(p) \int_{\pi}u = 0, \ t_{!}(v)(q)= u(p) \phi(\pi(p)) \int_{\pi} c= u(p);\]
therefore $u$ is in the image of $s_{!}-t_{!}$.
\end{proof}
\end{example}

%%%%%%%%%%%%%%%%%%%%%%%%%%%
%%%%%%%%%%%%%%%%%%%%%%%%%%%
%%%%%%%%%%%%%%%%%%%%%%%%%%%
%%%%%%%%%%%%%%%%%%%%%%%%%%%
%%%%%%%%%%%%%%%%%%%%%%%%%%%
%%%%%%%%%%%%%%%%%%%%%%%%%%%
%%%%%%%%%%%%%%%%%%%%%%%%%%%
%%%%%%%%%%%%%%%%%%%%%%%%%%%
%%%%%%%%%%%%%%%%%%%%%%%%%%%
%%%%%%%%%%%%%%%%%%%%%%%%%%%
\section{{\it Geometric} transverse measures (transverse densities)}
\label{gen-II}
%%%%%%%%%%%%%%%%%%%%%%%%%%%
%%%%%%%%%%%%%%%%%%%%%%%%%%%
%%%%%%%%%%%%%%%%%%%%%%%%%%%
%%%%%%%%%%%%%%%%%%%%%%%%%%%
%%%%%%%%%%%%%%%%%%%%%%%%%%%
%%%%%%%%%%%%%%%%%%%%%%%%%%%
%%%%%%%%%%%%%%%%%%%%%%%%%%%
%%%%%%%%%%%%%%%%%%%%%%%%%%%
%%%%%%%%%%%%%%%%%%%%%%%%%%%
%%%%%%%%%%%%%%%%%%%%%%%%%%%

%%%%%%%%%%%%%%%%%%%%%%%%%%%
%%%%%%%%%%%%%%%%%%%%%%%%%%%
%%%%%%%%%%%%%%%%%%%%%%%%%%%
%%%%%%%%%%%%%%%%%%%%%%%%%%%
%%%%%%%%%%%%%%%%%%%%%%%%%%%
\subsection{The transverse volume and density bundles}\label{sub-tr-dens-bundle}
%%%%%%%%%%%%%%%%%%%%%%%%%%%
%%%%%%%%%%%%%%%%%%%%%%%%%%%
%%%%%%%%%%%%%%%%%%%%%%%%%%%
%%%%%%%%%%%%%%%%%%%%%%%%%%%
%%%%%%%%%%%%%%%%%%%%%%%%%%%

We are now interested in {\it geometric} transverse measures, i.e.,  the
 analogues of the measures  induced by densities (Definition \ref{geom-measures-maniflds}).
Looking for linear functionals on $C_{c}^{\infty}(M, \mathcal{D}_A)$ that arise from sections of a vector bundle combined with the canonical integration on $M$ (like for standard densities) one immediately is led to the transverse density bundle:

\begin{definition}\label{def-tr-gpd-dens} For a Lie algebroid $A$ over $M$, the \textbf{the transverse density bundle of $A$} is the vector bundle over $M$ defined by: 
\[ \mathcal{D}_{A}^{\textrm{tr}}:= \mathcal{D}_{A^*}\otimes \mathcal{D}_{TM} .\]
Given a section $\sigma$ of $\mathcal{D}_{A}^{\textrm{tr}}$, \textbf{a decomposition of $\sigma$} is any writing of $\sigma$ of type
$\sigma= \rho^{\vee}\otimes \tau$ 
with $\rho\in C^{\infty}(M,\mathcal{D}_A)$ strictly positive and $\tau$ a density on $M$ (where $\rho^{\vee}$ is the dual section induced by $\rho$).
\end{definition}

Similarly one can define the \textbf{transverse volume and orientation bundles}
\[ \mathcal{V}^{\textrm{tr}}_{A}:= \mathcal{V}_{A^*}\otimes \mathcal{V}_{TM}= \Lambda^{\textrm{top}}A\otimes \Lambda^{\textrm{top}}T^*M,\ \ \ \mathfrak{o}^{\textrm{tr}}_{A}:= \mathfrak{o}_{A^*}\otimes \mathfrak{o}_{TM}\]
and the usual relations between these bundles continue to hold in this setting; e.g.:
\[ \mathcal{D}_{A}^{\textrm{tr}}= |\Lambda^{\textrm{top}}A\otimes \Lambda^{\textrm{top}}T^*M|= |\mathcal{V}^{\textrm{tr}}_{A}| . \]

One of the main properties of these bundles is that they are representations of $A$ and, even better, of $\G$, whenever $\G$ is a Lie groupoid with algebroid $A$; hence they do deserve the name of ``transverse'' vector bundles. We describe the canonical action of $\G$ on the transverse density bundle $\mathcal{D}_{A}^{\textrm{tr}}$; for the other two the description is identical. We have to associate to any arrow $g: x\rightarrow y$ of $\G$ a linear transformation 
\[ g_{*}: \mathcal{D}_{A, x}^{\textrm{tr}}\rightarrow \mathcal{D}_{A, y}^{\textrm{tr}}.\]
The differential of $s$ and the right translations induce a short exact sequence
% The short exact sequence induced by the differential of $s$  and right-translations % (\ref{Rg}) 
\[ 0 \rightarrow A_y \rightarrow T_g\G \stackrel{ds}{\rightarrow} T_{x}M \rightarrow 0 ,\]
which, in turn (cf. item 2 in Subsection \ref{Geometric measures: densities}), induces an isomorphism:  
\begin{equation}\label{DTG-dec} 
\mathcal{D}(T_g\G)\cong \mathcal{D}(A_y)\otimes \mathcal{D}(T_xM).
\end{equation}
Using the similar isomorphism at $g^{-1}$ and the fact that the differential of the inversion map gives an isomorphism $T_g\G\cong T_{g^{-1}}\G$, we find an isomorphism
\[ \mathcal{D}(A_y)\otimes \mathcal{D}(T_xM)\cong \mathcal{D}(A_x)\otimes \mathcal{D}(T_yM).\]
and therefore an isomorphism:
\[ \mathcal{D}(A_{x}^{*})\otimes \mathcal{D}(T_xM)\cong \mathcal{D}(A_{y}^{*})\otimes \mathcal{D}(T_yM),\]
and this defines the action $g_*$ we were looking for (it is straightforward to check that this defines indeed an action). 

\begin{definition}  A \textbf{transverse density} for the Lie groupoid $\G$ is any $\G$-invariant section of the transverse density bundle 
$\mathcal{D}_{A}^{\textrm{tr}}$. When we want to stress the structure of $\G$-representation present on $\mathcal{D}_{A}^{\textrm{tr}}$, we will denote it by $\mathcal{D}_{\G}^{\textrm{tr}}$.
\end{definition}

%%%%%%%%%%%%%%%%%%%%%%%%%%%
%%%%%%%%%%%%%%%%%%%%%%%%%%%
%%%%%%%%%%%%%%%%%%%%%%%%%%%
%%%%%%%%%%%%%%%%%%%%%%%%%%%
%%%%%%%%%%%%%%%%%%%%%%%%%%%
\subsection{Transverse densities as transverse measures}\label{subsec-Transverse densities as transverse measures}
%%%%%%%%%%%%%%%%%%%%%%%%%%%
%%%%%%%%%%%%%%%%%%%%%%%%%%%
%%%%%%%%%%%%%%%%%%%%%%%%%%%
%%%%%%%%%%%%%%%%%%%%%%%%%%%
%%%%%%%%%%%%%%%%%%%%%%%%%%% 

Returning to the relevance of the transverse density bundle to transverse measures, note that the canonical pairing 
% $\langle \cdot, \cdot \rangle$ 
between $\mathcal{D}_{A}$ and $\mathcal{D}_{A^*}$  and the integration of densities on $M$ allows us to interpret any section 
$\sigma$ of $\mathcal{D}_{A}^{\textrm{tr}}$ as a linear map 
\[ \mu_{\sigma}: C_{c}^{\infty}(M, \mathcal{D}_A) \rightarrow \mathbb{R}.\]
Explicitly, if $\sigma= \rho^{\vee}\otimes \tau$ is a decomposition of $\sigma$,
\[ \mu_{\sigma}(f \cdot \rho)=\int_{M} f\cdot \tau=  \int_{M} f(x)\  d\mu_\tau(x)\ \ \ \forall\ f\cdot \rho\in C_{c}^{\infty}(M, \mathcal{D}_A) .\]
Of course, the question is: when does $\mu_{\sigma}$ define a transverse measure 
(i.e., descends to $C_{c}^{\infty}(M//\G)$)? We have:

\begin{proposition}\label{prop-transverse-measures-densities} For a section $\sigma\in C^{\infty}(M,\mathcal{D}_{\G}^{\mathrm{tr}})$, the following are equivalent:
\begin{enumerate}
\item[1.] $\mu_{\sigma}\circ s_{!}=  \mu_{\sigma}\circ t_{!}$ (i.e.,  $\sigma$ gives rise to a transverse measure, if it is positive).
\item[2.] $\sigma$ is an invariant section of  $\mathcal{D}_{\G}^{\mathrm{tr}}$ (i.e., it is a transverse density).
\end{enumerate}
\end{proposition}

Hence transverse densities appear as transverse measures of geometric type. The proof of the proposition will be given together with that of Proposition \ref{prop-equiv-viewpoints}.

%%%%%%%%%%%%%%%%%%%%%%%%%%%
%%%%%%%%%%%%%%%%%%%%%%%%%%%
%%%%%%%%%%%%%%%%%%%%%%%%%%%
%%%%%%%%%%%%%%%%%%%%%%%%%%%
%%%%%%%%%%%%%%%%%%%%%%%%%%%
\subsection{The compatibility with other view-points} \label{subsec-The compatibility with other view-points}
%%%%%%%%%%%%%%%%%%%%%%%%%%%
%%%%%%%%%%%%%%%%%%%%%%%%%%%
%%%%%%%%%%%%%%%%%%%%%%%%%%%
%%%%%%%%%%%%%%%%%%%%%%%%%%%
%%%%%%%%%%%%%%%%%%%%%%%%%%%

Here we give two more equivalent characterizations of the invariance of a section $\sigma\in C^{\infty}(M,\mathcal{D}_{A}^{\textrm{tr}})$ which show the compatibility of this condition with the points of view of 
measure groupoids \cite{hahn} and of non-commutative geometry \cite{Connes}. However, unlike the previous proposition, 
the next one depends on a decomposition $\sigma= \rho^{\vee} \otimes \tau$ (see Definition \ref{def-tr-gpd-dens}) that we now fix. 
In this situation, we have:
\begin{enumerate}
\item[1.] an induced measure $\mu_{\G}$ (to be defined) on the manifold $\G$ and, from the point of view of measure groupoids, the main condition to require is the invariance of $\mu_{\G}$ under the inversion map. Of course, since we are in the geometric setting, $\mu_{\G}= \mu_{\rho_{\G}}$ will be associated to a density $\rho_{\G}$ on $\G$. Finally,   
$\rho_{\G}$ is defined as the density whose value at an arrow $g: x\rightarrow y$ of $\G$ equals to  $\rho_y\otimes \tau_x$ modulo the isomorphism (\ref{DTG-dec}). 
\item[2.] the concrete realization $(C_{c}^{\infty}(\G), \star_{\rho})$ (using $\rho$) of the convolution algebra and, using $\tau$, the associated integration over the units
\[ \tilde{\mu}_{\tau}: C_{c}^{\infty}(\G)\rightarrow \mathbb{R}, \ u\mapsto \mu_{\tau}(u|_{M})= \int_{M} u(1_x) \tau(x).\]
The interesting condition is, again, the trace condition.
% \[ \tilde{\tau}(u\star_{\rho} v)   = \tilde{\tau} (v\star_{\rho} u)\ \ \ \forall\ u, v\in C_{c}^{\infty}(\G).\]
\end{enumerate}

\begin{proposition}\label{prop-equiv-viewpoints} Given $\sigma= \rho^{\vee} \otimes \tau$ as above, the following are equivalent:
\begin{enumerate}
\item[2.] $\sigma$ is invariant (i.e., it is a transverse density).
\item[3.] the measure $\mu_{\rho_{\G}}$ on $\G$ is invariant under the inversion map.
\item[4.] $\tilde{\mu}_{\tau}$ is a trace on the algebra $(C_{c}^{\infty}(\G), \star_{\rho})$.
\end{enumerate}
\end{proposition}

\begin{proof}(of Propositions \ref{prop-transverse-measures-densities} and \ref{prop-equiv-viewpoints}) We show that that $1$-$4$ from the last two propositions are all equivalent. 
In order to handle the invariance of $\sigma$, we write more explicitly the action of $\G$ on $\mathcal{D}_{\G}^{\textrm{tr}}$ from Section \ref{sub-tr-dens-bundle}. If $g:x\to y$ is an arrow of $\G$, we use the canonical identifications 
\begin{align*}\mathcal{D}(A_{x}^{*})\otimes \mathcal{D}(A_{y}^{*})\otimes \mathcal{D}(A_y)\otimes \mathcal{D}(T_xM) &\stackrel{c_{23}}{\to}\mathcal{D}_{A, x}^{\textrm{tr}}\\
\mathcal{D}(A_{x}^{*})\otimes \mathcal{D}(A_{y}^{*})\otimes \mathcal{D}(A_x)\otimes \mathcal{D}(T_yM)
&\stackrel{c_{13}}{\to}\mathcal{D}_{A, y}^{\textrm{tr}},
\end{align*}
where $c_{ij}$ denotes the map given by contraction of the $i$-th and $j$-th factors of the tensor product in the left-hand side. Since $\rho$ is strictly positive, under these identifications $\sigma_x\in \mathcal{D}_{A, x}^{\textrm{tr}}$ corresponds to $\rho^{\vee}_x \otimes \rho^\vee_y \otimes  \rho_y \otimes \tau_x$ and  $\sigma_y$ corresponds to $\rho^{\vee}_x \otimes \rho^\vee_y \otimes  \rho_x \otimes \tau_y$. The action $g_{*}: \mathcal{D}_{A, x}^{\textrm{tr}}\rightarrow \mathcal{D}_{A, y}^{\textrm{tr}}$, when written in terms of these identifications, is given by the tensor product of the identity map of $\mathcal{D}(A_{x}^{*})\otimes \mathcal{D}(A_{y}^{*})$ with the map ${di}_g$ induced by the differential of the inversion of $\G$ at $g$. Therefore, $\sigma$ is an invariant section if and only if for all $g\in \G$ with $g:x\to y$, 
\[di_g(\rho_y \otimes \tau_x)=\rho_x \otimes \tau_y.\]
This is the same as the condition that $i^*(\rho_{\G})= \rho_{\G}$ which, of course, is equivalent to the fact that the associated measure $\mu_{\rho_{\G}}$ is invariant under the inversion of $\G$. This proves the equivalence of $2$ and $3$.

The equivalence between $1$ and $4$ is just a rewriting of Proposition \ref{equiv-to-trace} via the identification of the intrinsic convolution algebra with $(C_{c}^{\infty}(\G), \star_{\rho})$. We are left with the equivalence of $3$ and $4$. 
Rewrite $3$ as:
\begin{equation}\label{cond-inv-in-proof} 
\mu_{\rho_{\G}}(u)= \mu_{\rho_{\G}}(u^*) 
\end{equation}
for all $u\in C_{c}^{\infty}(\cG)$, where $u^*(g)= u(g^{-1})$. 
 Using the definition of $\rho_{\G}$  we obtain the formula for the associated integration 
\[ \mu_{\rho_{\G}}(f)= \int_{\G} f(g) \rho_{\G}(g)= \int_{\G} f(g) \rho_{t(g)}\otimes \tau_{s(g)}= \int_{M}  \left( \int_{s^{-1}(x)} f(g)\  d\mu_{\stackrel{\rightarrow}{\rho}}(g)\right) \tau_x .\]

As in the proof of Proposition \ref{equiv-to-trace}, it suffices to require (\ref{cond-inv-in-proof}) on elements of type $f= u^*v$ (hence $f(g)= u(g^{-1})v(g)$). But, on such elements, the previous formula combined with the definition (\ref{star-rho}) of $\star_{\rho}$ gives $\int_{M} (u\star_{\rho} v)(1_x) \tau_x= \tilde{\mu}_{\tau}(u\star_{\rho} v)$; then, on $f^*= v^* u$ we obtain $\tilde{\mu}_{\tau}(v\star_{\rho} u)$. Hence $3$ is equivalent to $\tilde{\mu}_{\tau}$ being a trace.  
\end{proof}

\begin{remark} The non-commutative view-point indicates that one should look also at higher versions of traces; this is related to the Hochschild and cyclic cohomology of the convolution algebra \cite{Connes} (see also \cite{BN}) from which traces emerge in degree zero. On the other hand also the invariant sections of $\mathcal{D}_{A}^{\textrm{tr}}$ show up as degree zero elements in a  cohomology: the differentiable cohomology of $\G$ with coefficients in $\mathcal{D}_{A}^{\textrm{tr}}$. It is natural to expect that, via restriction at units, the differentiable cohomology in higher degrees will give rise to higher (non-commutative) traces; such higher traces will be ``geometric'' in the same way that measures induced by densities are. Our comments are very much in line with the work of Pflaum-Posthuma-Tang \cite{PPT}. 
\end{remark}

%%%%%%%%%%%%%%%%%%%%%%%%%%%
%%%%%%%%%%%%%%%%%%%%%%%%%%%%%
%%%%%%%%%%%%%%%%%%%%%%%%%%%
%%%%%%%%%%%%%%%%%%%%%%%%%%%%%%
\section{Intermezzo: the case of proper Lie groupoids}
\label{sec-proper}
%%%%%%%%%%%%%%%%%%%%%%%%%%%
%%%%%%%%%%%%%%%%%%%%%%%%%%%%%
%%%%%%%%%%%%%%%%%%%%%%%%%%%
%%%%%%%%%%%%%%%%%%%%%%%%%%%%%%

In this section we will show that ``the general nonsense'' from the previous sections takes a much more concrete (but not so obvious) form in the case of proper groupoids. Throughout this section we fix a proper groupoid $\G$ over $M$, and we consider the orbit space 
\[ B:= M/\G, \ \ \ \textrm{with\ quotient\ map}\ \ \pi: M\rightarrow B.\]
Properness implies that $B$ is Hausdorff and locally compact; however, $B$ has much more structure. The upshot of this section is that the abstract transverse measures discussed in the previous two sections descend to (standard) measures on $B$. We start with the existence result for transverse densities.

\begin{proposition}\label{existence-tr-dens}
Any proper Lie groupoid admits strictly positive transverse densities.
\end{proposition}

\begin{proof} This is an immediate application of averaging, completely similar to our illustrative Lemma \ref{illustration-average}: using any proper Haar system $\mu$, if one starts with a strictly positive section $\rho\in C^{\infty}(M,\mathcal{D}_{A}^{\textrm{tr}})$, then 
\[ \textrm{Av}_{\mu}(\rho)_x:= \int_{s^{-1}(x)} g^*\rho_{t(g)}\ d\mu^x(g) \]
remains strictly positive and it is equivariant: for $a: x\rightarrow y$, 
\[  a^*  \textrm{Av}_{\mu}(\rho)_{y}= \int_{s^{-1}(y)} a^*g^*\rho_{t(g)} \ d\mu^{y}(g);\]
using $a^*g^*= (ga)^*$, $t(g)= t(ga)$ and invariance of $\mu$, the last expression is
\[ \int_{s^{-1}(y)} (ga)^*\rho_{t(ga)}\ d\mu^{y}(g)=  \int_{s^{-1}(x)} g^*\rho_{t(g)}\ d\mu^x(g)= \textrm{Av}_{\mu}(\rho)_x.\]
Hence the existence of invariant transverse densities and measures follows from the existence of proper Haar systems (Proposition \ref{equiv-proper}).
\end{proof}

%%%%%%%%%%%%%%%%%%%%%%%%%%%
%%%%%%%%%%%%%%%%%%%%%%%%%%%
%%%%%%%%%%%%%%%%%%%%%%%%%%%
%%%%%%%%%%%%%%%%%%%%%%%%%%%
%%%%%%%%%%%%%%%%%%%%%%%%%%%
\subsection{Classical measures on the orbit space $B$}\label{subsec-Classical measures on the orbit space $B$}
%%%%%%%%%%%%%%%%%%%%%%%%%%%
%%%%%%%%%%%%%%%%%%%%%%%%%%%
%%%%%%%%%%%%%%%%%%%%%%%%%%%
%%%%%%%%%%%%%%%%%%%%%%%%%%%
%%%%%%%%%%%%%%%%%%%%%%%%%%%

We have already mentioned that $B=M/\G$ is a quite nicely behaved space. Although it is not a manifold in general, in many respects it behaves like one. For instance, one has a well-behaved algebra of smooth functions:
\[ C^{\infty}(B):= \{ f: B \rightarrow \mathbb{R}\ :\ f\circ \pi \in C^{\infty}(M)\},\]
(and one can show that one can recover the topological space $B$ out of the algebra  $C^{\infty}(B)$ as its spectrum - see \cite{thesis}). The algebra of compactly supported smooth functions is defined as usual, by just adding the compact support condition:
\[ C^{\infty}_{c}(B):= \{ f\in C^{\infty}(B)\ :\ \ \textrm{supp}(f)- \textrm{compact}\}.\]
Of course, one can define similarly $ C^{k}_{c}(B)$ for any integer $k\geq 0$; for $k= 0$, since $B$ is endowed with the quotient topology, one recovers the usual space of compactly supported continuous functions on the space $B$. As in the case of smooth manifolds (Proposition \ref{1to1measures}), one has:

\begin{proposition}\label{1to1measures-orbi}
If $\G$ is a proper Lie groupoid over $M$, $B= M/\G$, then
\begin{enumerate}
\item[1.] $C_{c}^{\infty}(B)$ is dense in $C_{c}(B)$.
\item[2.] any positive linear functional on $C_{c}(B)$ is automatically continuous; and the same holds for $C_{c}^{\infty}(B)$ with the induced topology.  
\item[3.] the construction $I\mapsto I|_{C_{c}^{\infty}(B)}$ induces a 1-1 correspondence:
\[ \{\textrm{measures\ on}\ B\} \stackrel{\sim}{\longleftrightarrow} \ \{ \textrm{positive\ linear\ functionals\ on}\ C_{c}^{\infty}(B)\} .\]
\end{enumerate}
\end{proposition}

\begin{proof} 
The density statement is well-known, but here is a less standard argument: use Stone-Weierstrass (cf. \cite[Thm.\ 10.A.2]{deitmar}). When $B$ is compact this is straightforward. In general, consider $\mathcal{A}= C_{c}^{\infty}(B)$ as a subspace of the space of continuous functions vanishing at infinity \[C_0(B):=\{f\in C(B)\ :\ \forall\ \epsilon>0\ ,\ \exists\ K\subset B\ \mathrm{compact,\ such\ that}\ |f_{|B\backslash K}|<\epsilon \};\]  the Stone-Weierstrass theorem for locally compact Hausdorff spaces (as it follows from the one for compact spaces applied to the one-point compactification of $B$ and the unitization of $\mathcal{A}$) tells us that $\mathcal{A}$ is dense in $C_0(B)$ with respect to the sup-norm, provided it is:
\begin{enumerate}[1.]
\item a sub-algebra of $C_0(B)$.
\item point-separating: given $x, y\in B$ distinct, there exists $f\in \mathcal{A}$ with $f(x)\neq f(y)$.
\item non-vanishing at any point: for any $x\in B$, there exists $f\in \mathcal{A}$ with $f(x)\neq 0$. 
\end{enumerate}
This is clearly satisfied by our $\mathcal{A}$ - our main illustration of the averaging procedure (Lemma \ref{illustration-average} in the Appendix) is precisely the fact that $C^{\infty}(M//\G)$ separates closed subsets. Thus we deduce that for any $f\in C_0(B)$ one can find a sequence $(f_n)_{n\geq 1}$ in $C_{c}^{\infty}(B)$ that converges to $f$ in the sup-norm. If $f$ has compact support $K$ then, choosing $\phi\in C_{c}^{\infty}(B)$ with $\phi|_{K}= 1$, we have that $\{\phi\cdot f_n\}_{n\geq 1}$ converges uniformly to $\phi\cdot f= f$ and the support of all these functions are inside the compact $\textrm{supp}(\phi)$; hence the last convergence holds in $C_{c}(B)$. 

For part 2 recall that the continuity of a function $\mu$ on $C_{c}(B)$ means that for every compact $K\subset B$, there exists a constant $C_K$ such that 
\[ |\mu(f)| \leq C_{K} ||f||_{\textrm{sup}} \ \ \ \forall\ f\in C_{K}(B).\]
To see that this is implied by the positivity of $\mu$, for $K\subset B$ compact, we choose a function $\phi_K\in C_{c}(B)$ that is $1$ on $K$, for any $f\in C_{K}(B)$ we have that
\[ - ||f||_{\textrm{sup}} \phi_K \leq f \leq  ||f||_{\textrm{sup}} \phi_K\]
hence, by the positivity of $\mu$, 
\[ |\mu(f)|\leq C_K  ||f||_{\textrm{sup}},\ \ \ \textrm{where}\ C_K= \mu(\phi_K).\]
Exactly the same argument (just that one chooses $\phi_K$ smooth) works for $C_{c}^{\infty}(B)$. 

With 1 and 2 at hand, 3 is basically the standard result that continuous functionals on a dense subspace of a locally convex vector space extend uniquely to continuous linear functionals on the entire space; starting with $\mu$ on $C_{c}^{\infty}(B)$, the extension $\tilde{\mu}$ is (and must be given by) $\tilde{\mu}(f)= \lim_{n\to \infty} \mu(f_n)$ where $\{f_n\}_{n\geq 1}$ is a sequence in $C_{c}^{\infty}(B)$ converging to $f$ (the continuity of $\mu$ implies that the sequence $\{\mu(f_n)\}_{n\geq 1}$ is Cauchy, hence the limit is well defined and is easily checked to be independent of the choice of the sequence). What does not follow right away from the standard argument is the positivity of $\tilde{\mu}$, but this is arranged by a simple trick: for $f\in C_{c}(B)$ one can find a sequence of smooth functions $g_n$ such that $g_{n}^{2}\to f$ (because $\sqrt{f}\in C_{c}(B)$), and then $\tilde{\mu}(f)$ is the limit of $\mu(g_{n}^{2})\geq 0$.
\end{proof}

%%%%%%%%%%%%%%%%%%%%%%%%%%%
%%%%%%%%%%%%%%%%%%%%%%%%%%%
%%%%%%%%%%%%%%%%%%%%%%%%%%%
%%%%%%%%%%%%%%%%%%%%%%%%%%%
%%%%%%%%%%%%%%%%%%%%%%%%%%%
\subsection{Transverse measures $\equiv$ classical measures on $B$}\label{subsec-transverse and classical measures}
%%%%%%%%%%%%%%%%%%%%%%%%%%%
%%%%%%%%%%%%%%%%%%%%%%%%%%%
%%%%%%%%%%%%%%%%%%%%%%%%%%%
%%%%%%%%%%%%%%%%%%%%%%%%%%%
%%%%%%%%%%%%%%%%%%%%%%%%%%%

Next, we relate transverse measures to classical measures on $B$. The main point is that the intuition described in Remark \ref{cpct-supp-intuition} can now be made precise; consider the intrinsic averaging:
\begin{equation}\label{intr-average}
\textrm{Av}_{\, \G}= \textrm{Av}: C_{c}^{\infty}(M, \mathcal{D}_A)\rightarrow C_{c}^{\infty}(B), \ 
\textrm{Av}(\rho)(\pi(x))=  \int_{s^{-1}(x)} \overrightarrow{\rho}|_{s^{-1}(x)}.
\end{equation}

\begin{theorem}\label{thm-iso-Cc(B)} The averaging map (\ref{intr-average}) induces an isomorphism
\begin{equation}\label{intr-average2} 
\textrm{Av}_{\, \G}= \textrm{Av}: C_{c}^{\infty}(M//\G) \stackrel{\sim}{\rightarrow} \CC_c(B) .
\end{equation}
In particular, this induces a 1-1 correspondence
\[ \{\textrm{transverse\ measures\ on} \ \G\} \stackrel{\sim}{\longleftrightarrow} \ \{\textrm{measures\ on}\ B=M/\G\}.\]
For the inverse of $\textrm{Av}$, choose any proper normalized Haar density $\rho$ (Def. \ref{proper-dens}) and then, for $f\in C_{c}^{\infty}(B)$, the corresponding element in $C_{c}^{\infty}(M//\G)$ is the one represented by $\pi^*(f)\rho\in C^{\infty}_c(M, \mathcal{D}_A)$.
\end{theorem}  

\begin{proof}
To check that the map $\textrm{Av}$ is well-defined, recall that if two elements of $C_{c}^{\infty}(M, \mathcal{D}_A)$ represent the same class in $C_{c}^{\infty}(M//\G)$, they differ by a density in the image of $(s_{!}-t_{!})$. If, as in Remark \ref{remark-s-shriek}, we fix any strictly positive density $\rho_0\in C^{\infty}(M, \mathcal{D}_A)$ to trivialize all the density bundles, we then check that 

\[\mathrm{Av}((s^\rho_{!}-t^\rho_{!})(u))(\pi(x))=\int_{s^{-1}(x)}\left(\int_{s^{-1}(t(g))}u(g)\overrightarrow{\rho_0}- \int_{s^{-1}(t(g))}u(g^{-1}) \overrightarrow{\rho_0} \right)\overrightarrow{\rho_0},\] which vanishes, being the difference of two double integrals with the same value. This is because both correspond to integrating $u$ over all arrows $g$ whose source (and target) belong to the orbit $\pi(x)$.

Next, choosing any proper normalized Haar density $\rho$, the inverse to $\mathrm{Av}$, given as above by mapping $f\in \CC_c(B)$ into the class of $\pi^*(f)\rho$ in $C_{c}^{\infty}(M//\G)$ is well defined. Indeed, for such an $f$, there is a compact subset $K$ of $M$ such that $\pi(K)$ is the support of $f$. This means that the support of $\pi^*f$ is the saturation of $K$, and the intersection of it with the support of $\rho$ is compact, because $\rho$ is a proper Haar density, so $\pi^*(f)\rho$ has compact support as well.

Finally, to see that this is indeed the inverse to $\mathrm{Av}$, we first check that \begin{align*}\mathrm{Av}([\pi^*(f)\rho])(\pi(x))
&=\int_{s^{-1}(x)}\pi^*(f)(t(g))\overrightarrow{\rho}=f(\pi(x))\int_{s^{-1}(x)}\overrightarrow{\rho}=f(\pi(x)).
\end{align*} Next, suppose that the proper normalized Haar density $\rho$ is of the form $\rho=c\cdot \rho'$, where $c$ is a cut-off function and $\rho'$ is full. Then, for any other $h\cdot \rho'\in C_{c}^{\infty}(M, \mathcal{D}_A)$,
\begin{align*}(h\cdot\rho'-\pi^*&\mathrm{Av}(h\cdot\rho')\rho)(x)
=h(x)\rho'(x) - \left(\int_{s^{-1}(x)}h(t(g))\overrightarrow{\rho'}\right)c(x)\rho'(x)\\
&= \left(\int_{s^{-1}(x)}h(x)c(t(g))\overrightarrow{\rho'} - \int_{s^{-1}(x)}c(x)h(t(g))\overrightarrow{\rho'}\right)\rho'(x),
\end{align*} and this expression represents zero in $C_{c}^{\infty}(M//\G)$. Indeed, using $\rho'$ to trivialize the density bundles, it becomes precisely $(s^{\rho'}_{!}-t^{\rho'}_{!})(s^*h\cdot t^*c)$.\end{proof}

%%%%%%%%%%%%%%%%%%%%%%%%%%%
%%%%%%%%%%%%%%%%%%%%%%%%%%%
%%%%%%%%%%%%%%%%%%%%%%%%%%%
%%%%%%%%%%%%%%%%%%%%%%%%%%%
%%%%%%%%%%%%%%%%%%%%%%%%%%%
\subsection{The measure on $B$ induced by a transverse density}\label{subsec-measure on B from tr density}
%%%%%%%%%%%%%%%%%%%%%%%%%%%
%%%%%%%%%%%%%%%%%%%%%%%%%%%
%%%%%%%%%%%%%%%%%%%%%%%%%%%
%%%%%%%%%%%%%%%%%%%%%%%%%%%
%%%%%%%%%%%%%%%%%%%%%%%%%%% 

The previous section implies that any positive transverse density $\sigma\in C^{\infty}(M,\mathcal{D}_{A}^{\textrm{tr}})$ induces a transverse measure for $\G$, hence a measure on $B$, that we still denote by 
\[ \mu_{\sigma}: C_{c}^{\infty}(B)\rightarrow \mathbb{R}.\]
To make this more explicit, we need the notion of cut-off function for a full Haar density $\rho$ ( Proposition \ref{equiv-proper}): a smooth positive function $c$ on $M$ satisfying:
\begin{enumerate}
\item[1.] the restriction of $s$ to $t^{-1}(\textrm{supp}(c))$ is proper (as a map to $M$).
\item[2.] $\int_{s^{-1}(x)} c(t(g))\ d\mu_{\rho}^{x}(g)= 1$ for all $x\in M$.
\end{enumerate}

\begin{proposition}\label{prop-tr-dens-B} If $\sigma$ is a positive  transverse density for $\cG$, then the induced measure $\mu_{\sigma}$ on $B$ is given by
\begin{equation}\label{eq-mu-sigma}
\int_{B} h(b)\ d\mu_{\sigma}(b)= \int_{M} c(x) h(\pi(x))\ d\mu_{\tau}(x)
\end{equation}
where, to write the right hand side, we choose any decomposition $\sigma= \rho^{\vee}\otimes \tau$ as in Definition \ref{def-tr-gpd-dens}, and we use a cut-off function $c$ for $\rho$.

Moreover, $\mu_{\sigma}$ is uniquely characterized by the formula
\[ \int_M f(x)\ d\mu_{\tau}(x) =\int_B\left(\int_{\O_b}f(y)\ d\mu_{\rho_{\O_b}}(y) \right)\ d\mu_{\sigma}(b),\]
where $\rho_{\O}$ denotes the density induced by $\rho$ on the orbit (cf. Section \ref{Induced measures/densities on the orbits}).
\end{proposition}

Note that it is rather remarkable that the right hand side of the formula \eqref{eq-mu-sigma} for $\mu_{\sigma}$ does not depend on the choice of the decomposition or of the cut-off function. 

\begin{proof} The formula defining $\mu_\sigma$ on $C_c^\infty(B)$ is obtained simply by transferring the definition of $\mu_\sigma$ on $C_{c}^{\infty}(M//\G)$ through the isomorphism described in Theorem \ref{thm-iso-Cc(B)}. 
% $\CC_c(B) \stackrel{\sim}{\rightarrow} C_{c}^{\infty}(M//\G)$ described in Theorem \ref{thm-iso-Cc(B)}, which maps $f\in \CC_c(B)$ to the class of $\pi^*(f)(c\cdot \rho)$ in $C_{c}^{\infty}(M//\G)$.
For the second part we also use the discussion on $\rho_\O$ from \ref{Induced measures/densities on the orbits} to compute:
\begin{align*}\int_B\left(\int_{\O_b}f(y)\ d\mu_{\rho_{\O_b}}(y) \right)\ d\mu_{\sigma}(b)&=\int_M c(x)\left(\int_{\O_x}f(y)\ d\mu_{\rho_{\O_b}}(y)\right)\ d\mu_{\tau}(x)\\
&=\int_M c(x)\left(\int_{s^{-1}(x)}f(t(g))\ d\mu_{\overrightarrow{\rho}}(g)\right)d\mu_{\tau}(x)\\
&=\int_M\left( \int_{s^{-1}(x)}c(s(g))f(t(g))\  d\mu_{\overrightarrow{\rho}}(g)\right)\ d\mu_{\tau}(x).
\end{align*}

The last integral is the same as the integral of the function $(s^*c)\cdot(t^*f)$ over $\G$ with respect to the density $\rho_\G$ induced by $\rho$ and $\tau$. Since  $\sigma= \rho^{\vee}\otimes \tau$ is invariant, Proposition \ref{prop-equiv-viewpoints} implies that $\rho_\G$ is invariant under the inversion, hence
\begin{align*}\int_B\left(\int_{\O_b}f(y)\ d\mu_{\rho_{\O_b}}(y) \right)\  d\mu_{\sigma}(b)&=\int_M\left( \int_{s^{-1}(x)}c(t(g))f(s(g))\  d\mu_{\overrightarrow{\rho}}(g)\right)\ d\mu_{\tau}(x)\\
&=\int_M f(x) \left( \int_{s^{-1}(x)}(c(t(g))\  d\mu_{\overrightarrow{\rho}}(g)\right)\ d\mu_{\tau}(x)\\
&=\int_M f(x)\ d\mu_{\tau}(x)\qedhere
\end{align*}
\end{proof}

%%%%%%%%%%%%%%%%%%%%%%%%%%%
%%%%%%%%%%%%%%%%%%%%%%%%%%%
%%%%%%%%%%%%%%%%%%%%%%%%%%%
%%%%%%%%%%%%%%%%%%%%%%%%%%%
%%%%%%%%%%%%%%%%%%%%%%%%%%%
\subsection{Some interesting consequences}\label{subsec-Some interesting consequences}
%%%%%%%%%%%%%%%%%%%%%%%%%%%
%%%%%%%%%%%%%%%%%%%%%%%%%%%
%%%%%%%%%%%%%%%%%%%%%%%%%%%
%%%%%%%%%%%%%%%%%%%%%%%%%%%
%%%%%%%%%%%%%%%%%%%%%%%%%%%

Here are some immediate but interesting consequences of Theorem \ref{thm-iso-Cc(B)} and the integral formula from Proposition \ref{prop-tr-dens-B}.

\begin{corollary}
With the same notations as above, if $\G$ is compact and $\sigma$ is a transverse density then, for any $h\in C^{\infty}(B)$,  
\[ \int_M h(\pi(x))\ d\mu_{\tau}(x) =\int_B h(b)\cdot \mathrm{Vol}(\O_b, \mu_{\O_b}) \  d\mu_{\sigma}(b).\]
In particular, the volume function $\mathrm{Vol}_{\rho}:  b\mapsto \mathrm{Vol}(\O_b, \mu_{\O_b})$ 
is a continuous function on $B$, smooth in the sense that its pull-back to $M$ is smooth, and its integral with respect to $\mu_{\sigma}$ is $\mathrm{Vol}(M, \mu_{\tau})$. 
\end{corollary}

Note that, in terms of set-measures, this reads as
\[ \mu_{\sigma}= \frac{1}{\mathrm{Vol}_{\rho}} \pi_{!}(\mu_{\tau}).\]

Smoothness of the volume function is a simple consequence of the equality $\mathrm{Vol}(\O_x, \mu_{\O_x})=\mathrm{Vol}(s^{-1}(x), \mu_{\overrightarrow{\rho}})$ (see Section \ref{Induced measures/densities on the orbits}); this function is smooth, by condition $2$ of the definition of a Haar system (Definition \ref{def-Haar-system}).

\begin{corollary} With the same notations as above, if $\G$ is compact  and $\sigma$ is a transverse density, then the volume of $B= M/\G$ with respect to $\mu_{\sigma}$ is given by
\[ \mathrm{Vol}(B, \mu_{\sigma})=\int_M \frac{1}{\mathrm{Vol}(\O_x, \mu_{\O_x})}\  d\mu_{\tau}(x) .\]
\end{corollary}

We see that this expression for the volume of $B$ with respect to $\mu_\sigma$ corresponds to the definition given by Weinstein (cf. Definition 2.3 and Theorem 3.2 in \cite{alan-volume}) for the volume of the stack $M//\G$ with data $(\rho,\tau)$.

%%%%%%%%%%%%%%%%%%%%%%%%%%%
%%%%%%%%%%%%%%%%%%%%%%%%%%%%%
%%%%%%%%%%%%%%%%%%%%%%%%%%%
%%%%%%%%%%%%%%%%%%%%%%%%%%%%%%
\section{Intermezzo: the case of regular Lie groupoids}
\label{sec-regular}
%%%%%%%%%%%%%%%%%%%%%%%%%%%
%%%%%%%%%%%%%%%%%%%%%%%%%%%%%
%%%%%%%%%%%%%%%%%%%%%%%%%%%
%%%%%%%%%%%%%%%%%%%%%%%%%%%%%%

%%%%%%%%%%%%%%%%%%%%%%%%%%%
%%%%%%%%%%%%%%%%%%%%%%%%%%%
%%%%%%%%%%%%%%%%%%%%%%%%%%%
%%%%%%%%%%%%%%%%%%%%%%%%%%%
%%%%%%%%%%%%%%%%%%%%%%%%%%%
\subsection{The general regular case} 
\label{The regular case} 
%%%%%%%%%%%%%%%%%%%%%%%%%%%
%%%%%%%%%%%%%%%%%%%%%%%%%%%
%%%%%%%%%%%%%%%%%%%%%%%%%%%
%%%%%%%%%%%%%%%%%%%%%%%%%%%
%%%%%%%%%%%%%%%%%%%%%%%%%%%

We now have a closer look at the regular case, i.e., when all the orbits of $\G$ have the same dimension. In this case the connected components of the orbits form a regular foliation on the base manifold $M$; we denote it by $\F$. We interpret $\cF$ as an involutive sub-bundle of $TM$; as such, it is the image of the anchor $\sharp: A\to TM$  of the algebroid $A$ of $\G$. 

One special feature of the regular case is that the naive orbit space $M/\cG$ coincides with the orbit space $M/\cE$ of some smaller groupoid $\cE$, namely one that integrates $\cF$ as an algebroid. For instance, when $\cG$ has connected $s$-fibers, then $\cE$ could be any $s$-connected integration of $\cF$, such as the holonomy or the monodromy groupoid of $\cF$. However, for any regular $\cG$ (even without $s$-connected fibers) there is a canonical choice for $\cE$:  
the quotient of $\G$ obtained by dividing out the action of the connected components $\G_{x}^{0}$ of the isotropy groups $\G_{x}$ (for the smoothness, see Proposition 2.5 in \cite{moerdijk}). We denote it by 
\begin{equation}\label{EEE}
\mathcal{E}= \mathcal{E}(\G)
\end{equation}
and call it \textbf{the foliation groupoid associated to $\G$}. 

The previous discussion indicates that there should be a close relationship between transverse measures for $\cG$ and for $\cE= \cE(\cG)$. To see that this is indeed the case, we pass to another special feature of the regular case: the vector bundles 
\[ \mathfrak{g}= \textrm{Ker}(\sharp), \ \ \nu= TM/\mathcal{F}\]
are representations of $\cG$ in a canonical way, called the \textbf{isotropy and the normal representations}, respectively (cf. e.g. \cite{QA}). For the action of an arrow $g: x\rightarrow y$ on $\mathfrak{g}$ we use the adjoint action $\textrm{Ad}_g(a)= g a g^{-1}$ going from the isotropy group of $\G$ at $x$ to the one at $y$; by differentiation, we obtain the action going from $\mathfrak{g}_x$ to $\mathfrak{g}_y$. For the action on $\nu$, start with $v\in \nu_x$ and choose a curve $g(t): x(t)\rightarrow y(t)$ with $g(0)= g$ and such that $\dot{x}(0)\in T_xM$ represents $v$; then $g\cdot v\in \nu_y$ is the vector represented by $\dot{y}(0)\in T_yM$. 

\begin{lemma} One has a canonical isomorphism of representations of $\G$:
\begin{equation}\label{reg-rewrite} 
\mathcal{D}_{\cG}^{\mathrm{tr}} \cong  \mathcal{D}_{\mathfrak{g}^*}\otimes   \mathcal{D}_{\nu} 
% \cong \mathcal{D}_{\mathfrak{g}^*}\otimes \mathcal{D}_{\cF ^{\mathrm{tr}} .
\end{equation}
For $\sigma\in \Gamma(\mathcal{D}^\mathrm{tr}_{A})$, $\kappa\in \Gamma(\cD_{\gg})$ nowhere vanishing and $\beta\in \Gamma(\cD_{\nu})$ we write 
\begin{equation} \label{reg-rewrite2} 
\sigma\equiv \kappa^{\vee}\otimes \beta
\end{equation}
if the two elements correspond to each other by the previous isomorphism. 
\end{lemma}

For the proof, consider the two exact sequences  of vector bundles over $M$:
\[ 0\rightarrow \mathfrak{g}\rightarrow A\stackrel{\rho}{\rightarrow} \mathcal{F}\rightarrow 0,\ \ 
 0\rightarrow \mathcal{F}\rightarrow TM \rightarrow \nu \rightarrow 0.\]
to deduce that there are canonical isomorphism (see 1 and 2 in \ref{Geometric measures: densities}):
\[ \mathcal{D}_{A}\cong \mathcal{D}_{\mathfrak{g}}\otimes \mathcal{D}_{\mathcal{F}}, \quad 
\mathcal{D}_{TM }\cong \mathcal{D}_{\mathcal{F} }\otimes \mathcal{D}_{\nu }.\]
Working out the actions, one obtains the lemma. Note that the lemma
applied to $\cE= \cE(\cG)$ (or just the previous argument) gives $\mathcal{D}_{\cE}^{\textrm{tr}}\cong \mathcal{D}_{\nu}$, hence also to a canonical isomorphism of representations of $\cG$: 
\[ \mathcal{D}_{\cG}^{\mathrm{tr}} \cong  \mathcal{D}_{\mathfrak{g}^*}\otimes \mathcal{D}_{\cE}^{\mathrm{tr}}.\]
Hence, to talk about strictly positive transverse densities for both $\cG$ and $\cE$, we have to assume the existence of 
$\kappa\in \Gamma(\mathcal{D}_{\mathfrak{g}})$ strictly positive and invariant. On the other hand,
since $\mathcal{D}_{A}= \mathcal{D}_{\mathfrak{g}}\otimes \mathcal{D}_{\cF}$, the pairing with $\kappa^{\vee}\in \Gamma(\cD_{\gg^*})$ induces  a map
\[ \kappa^{\vee} : C_{c}^{\infty}(M, \mathcal{D}_A)\rightarrow C_{c}^{\infty}(M, \mathcal{D}_{\cF}).\]
The next result follows now by a straightforward computation:

\begin{proposition}\label{pro-kappa}  Assume that $\kappa\in \Gamma(\mathcal{D}_{\mathfrak{g}})$ is strictly positive and invariant. Then the pairing with $\kappa^{\vee}$ descends to an isomorphism 
\begin{equation}\label{isom-kappa} 
\kappa^{\vee} : C_{c}^{\infty}(M//\G) \stackrel{\sim}{\rightarrow} C_{c}^{\infty}(M//\mathcal{E}),
\end{equation}
and it induces a 1-1 correspondence 
\[ \{\textrm{transverse\ measures\ for} \ \G\} \stackrel{\sim}{\longleftrightarrow} \ \{\textrm{transverse\ measures\ for} \ \cE\}.\]
Moreover, this preserves the geometricity of measures; more precisely, via the relation $\sigma\equiv \kappa^{\vee}\otimes \beta$ (see  (\ref{reg-rewrite2})), $\kappa$ induces a 1-1 correspondence between 
    \begin{itemize}
    \item  invariant sections $\beta$ of $\cD_{\nu}$ (i.e.,  transverse densities for $\mathcal{E}$)
    \item invariant sections $\sigma$ of $\cD_{A}^{\textrm{tr}}$ (i.e., transverse densities for $\G$).
    \end{itemize}
uniquely characterized by the commutative diagram:
\[ \xymatrix{
 C_{c}^{\infty}(M//\G)\ar[rr]^-{\kappa^{\vee}} \ar[rd]_-{\mu_{\sigma}} & & C_{c}^{\infty}(M//\mathcal{E})\ar[ld]^-{\mu_{\beta}}\\
 & \mathbb{R} & }\]
\end{proposition}

%%%%%%%%%%%%%%%%%%%%%%%%%%%
%%%%%%%%%%%%%%%%%%%%%%%%%%%
%%%%%%%%%%%%%%%%%%%%%%%%%%%
%%%%%%%%%%%%%%%%%%%%%%%%%%%
%%%%%%%%%%%%%%%%%%%%%%%%%%%
\subsection{The proper regular case}
\label{The proper regular case}
%%%%%%%%%%%%%%%%%%%%%%%%%%%
%%%%%%%%%%%%%%%%%%%%%%%%%%%
%%%%%%%%%%%%%%%%%%%%%%%%%%%
%%%%%%%%%%%%%%%%%%%%%%%%%%%
%%%%%%%%%%%%%%%%%%%%%%%%%%%

We now assume that $\G$ is both regular and proper. In this case there are two new features that are present:

\begin{itemize}
\item Both $C_{c}^{\infty}(M//\G) $ and $C_{c}^{\infty}(M//\cE)$ are isomorphic to $C_{c}^{\infty}(B)$, by the averaging of Theorem \ref{thm-iso-Cc(B)} applied to $\G$ and $\cE$, respectively.  % ($B= M/\G= M/\cE$). 

\item There is a preferred choice of a strictly positive invariant section $\kappa$: since the isotropy groups $\G_x$ are compact, their Lie algebras $\mathfrak{g}_x$ come with induced Haar densities; we consider the ones associated to the identity components of $\G_{x}$ - they  
define a strictly positive invariant section
\[ \kappa_{\mathrm{Haar}}\in \Gamma(\mathcal{D}_{\mathfrak{g}}).\]
\end{itemize}
Note that choosing the Haar densities associated to the identity components (rather than of the entire $\G_x$'s) is essential for the smoothness of $\kappa_{\mathrm{Haar}}$; indeed, while the bundle consisting of the full isotropy groups may fail to be smooth, passing to identity components does produce a smooth bundle of Lie groups (cf. Proposition 2.5 in \cite{moerdijk}). 
Using the normalization and the invariance conditions for the Haar densities, the following is immediate:

\begin{proposition}\label{weinst-conj0} $\kappa_{\mathrm{Haar}}$ is the only strictly positive invariant section of $\mathcal{D}_{\mathfrak{g}}$ with the property that the isomorphism (\ref{isom-kappa}) becomes the identity after identifying $C_{c}^{\infty}(M//\G)$ and $C_{c}^{\infty}(M//\cF)$ with $C_{c}^{\infty}(B)$. %  (i.e., (\ref{intr-average2} applied to $\G$ and $\mathcal{E}$).
\end{proposition}

Putting this and Proposition \ref{pro-kappa} together we find in particular:

\begin{corollary}\label{cor-reg-proper} The relation $\sigma\equiv \kappa^{\vee}_{\mathrm{Haar}}\otimes \beta$ (see  (\ref{reg-rewrite2})) induces a bijection  between 
    \begin{itemize}
    \item  invariant sections $\beta$ of $\cD_{\nu}$ (i.e.,  transverse densities for $\mathcal{E}$)
    \item invariant sections $\sigma$ of $\cD_{A}^{\textrm{tr}}$ (i.e., transverse densities for $\G$).
    \end{itemize}
Moreover, the measures $\mu_{\sigma}$ and $\mu_{\beta}$ induced on $B$ by $\sigma$ and $\beta$ coincide.
\end{corollary}
 
Or, on a commutative diagram:
\[ \xymatrix{
 & C_{c}^{\infty}(B) & \\
 C_{c}^{\infty}(M//\G)\ar[rr]^{\kappa^{\vee}}_{\sim} \ar[rd]_-{\mu_{\sigma}}
 \ar[ru]^{\mathrm{Av}_{\mathcal{\cG}}}_{\sim} & & C_{c}^{\infty}(M//\mathcal{E})\ar[ld]^-{\mu_{\beta}} \ar[lu]_{\mathrm{Av}_{\mathcal{\cE}}}^{\sim}\\
 & \mathbb{R} & }\]

%%%%%%%%%%%%%%%%%%%%%%%%%%%
%%%%%%%%%%%%%%%%%%%%%%%%%%%%%
%%%%%%%%%%%%%%%%%%%%%%%%%%%
%%%%%%%%%%%%%%%%%%%%%%%%%%%%%%
\section{Intermezzo: the case of symplectic groupoids}
\label{sec-sympl-gpds}
%%%%%%%%%%%%%%%%%%%%%%%%%%%
%%%%%%%%%%%%%%%%%%%%%%%%%%%%%
%%%%%%%%%%%%%%%%%%%%%%%%%%%
%%%%%%%%%%%%%%%%%%%%%%%%%%%%%%

%%%%%%%%%%%%%%%%%%%%%%%%%%%
%%%%%%%%%%%%%%%%%%%%%%%%%%%
%%%%%%%%%%%%%%%%%%%%%%%%%%%
%%%%%%%%%%%%%%%%%%%%%%%%%%%
%%%%%%%%%%%%%%%%%%%%%%%%%%%
\subsection{The general case} 
\label{The general case} 
%%%%%%%%%%%%%%%%%%%%%%%%%%%
%%%%%%%%%%%%%%%%%%%%%%%%%%%
%%%%%%%%%%%%%%%%%%%%%%%%%%%
%%%%%%%%%%%%%%%%%%%%%%%%%%%
%%%%%%%%%%%%%%%%%%%%%%%%%%%

In this section we look at another particular class of groupoids: those that come from Poisson Geometry, i.e., symplectic groupoids \cite{GS,alan-GS}. As mentioned in the introduction, this case and its relevance to compactness in Poisson Geometry \cite{PMCT1,PMCT2,PMCT3} was one of our original motivations for this paper. Recall (cf. e.g. \cite{GS}) that a symplectic groupoid $(\G, \Omega)$ is a Lie groupoid $\G$ endowed with a symplectic form $\Omega\in \Omega^2(\G)$ which is multiplicative. In this case the Lie algebroid $A$ of $\G$ is, as a vector bundle, canonically isomorphic to $T^*M$ by
\[ A\cong T^*M, \ \ \alpha_x\mapsto (X_x\mapsto \Omega_x(\alpha_x, X_x)),\]
where we identify $x\in M$ with $1_x\in \G$. Therefore we obtain, as vector bundles, 
\begin{equation}\label{DTM-sympl-case}  
\mathcal{D}_{A}^{\textrm{tr}}= \mathcal{D}_{TM}\otimes \mathcal{D}_{TM}
\end{equation}
and similarly when replacing $\mathcal{D}$ by $\mathcal{V}$ or $\mathfrak{o}$.

In the remaining part of this section we fix a symplectic groupoid $(\G, \Omega)$ with connected $s$-fibers. Recall that the base $M$ carries an induced Poisson structure, uniquely determined by the fact that the source map is a Poisson map \cite{GS}. For $f\in C^{\infty}(M)$ we will denote by $X_{f}$ the corresponding Hamiltonian vector field on $M$. We say that a density $\tau$ on $M$ is invariant under Hamiltonian flows if 
\[ L_{X_f}(\tau)= 0 \ \ \ \forall\ f\in C^{\infty}(M).\]

\begin{proposition}\label{Poisson-first} The correspondence
\[ \tau\in \Gamma(\mathcal{D}_{TM}) \mapsto \sigma:= \tau\otimes \tau\in \Gamma(\mathcal{D}_{A}^{\textrm{tr}})\]
induces a bijection between strictly positive:
\begin{itemize}
\item transverse densities $\sigma$ for $\G$. 
\item densities $\tau$ on $M$ which are invariant under Hamiltonian flows.
\end{itemize}
Furthermore, $\mathcal{D}_{TM}$ can be made into a representation of $\G$ in a canonical way, uniquely determined by the conditions:
\begin{enumerate}\item[1.] (\ref{DTM-sympl-case}) becomes an isomorphism of representations of $\G$
\item[2.] the action preserves the (fiberwise) positivity of densities.
\end{enumerate}
With respect to this action, a density $\tau$ on $M$ is invariant as a section of $\mathcal{D}_{TM}$ if and only if $\tau$ 
is invariant under Hamiltonian flows.
\end{proposition}

\begin{proof} In general, the action of $\G$ on $\mathcal{D}_{A}^{\textrm{tr}}$ induces, by differentiation, an  infinitesimal action of $A$, i.e., a flat $A$-connection $\nabla$ on $\mathcal{D}_{A}^{\textrm{tr}}$. Using the formula for $\nabla$ from \cite{QA} (but using $\mathcal{D}_{A}^{\textrm{tr}}$ instead of $\mathcal{V}_{A}^{\textrm{tr}}$), in the case where $A=T^*M$ is the Lie algebroid of the symplectic groupoid $\G$, we have that \[\nabla_{df}(\tau_1\otimes\tau_2)=(\cL_{X_f}\tau_1)\otimes \tau_2+\tau_1\otimes(\cL_{X_f}\tau_2). \]

Therefore, if $\tau$ is a strictly positive density on $M$ invariant under Hamiltonian flows, then $\nabla_{df}(\tau\otimes\tau)=0$ for all $f\in C^\infty (M)$. Hence $\tau\otimes\tau$ is invariant with respect to the infinitesimal action of $A$. Since $\G$ is $s$-connected, $\tau\otimes\tau$ is $\G$-invariant as well. Moreover, the mapping $\tau\mapsto \tau\otimes\tau$ is injective on strictly positive densities.

Conversely, if $\sigma$ is a strictly positive transverse density for $\G$,  we can choose a decomposition of the form $\sigma=\tau\otimes u\tau$, where $u\in C^\infty(M)$ is strictly positive. Then $\tau'=\sqrt{u}\tau$ is a strictly positive smooth density on $M$ and $\sigma=\tau'\otimes \tau'$. Since $\sigma$ is a transverse density for $\G$, we have: \[0=\nabla_{df}\sigma=(\cL_{X_f}\tau')\otimes \tau'+\tau'\otimes (\cL_{X_f}\tau')\]

Writing $\cL_{X_f}\tau'=v\tau'$, where $v\in C^\infty(M)$, we conclude that $2v\tau'\otimes\tau'=0$. Since $\tau'$ is strictly positive, $v=0$. Therefore $\tau'$ is invariant under Hamiltonian flows.

In order to make $\mathcal{D}_{TM}$ into a representation of $\G$  as in the statement, pick a strictly positive density $\tau$ on $M$. Then, for any arrow $g:x\to y$ in $\G$, \[(g\cdot\tau_x)\otimes (g\cdot \tau_x)=g\cdot(\tau_x \otimes \tau_x).\] Since the action of $\G$ on $\mathcal{D}_{A}^{\textrm{tr}}$ preserves the positivity of sections, there exists a positive smooth function  $c^\tau\in C^\infty (\G)$, such that \[g\cdot(\tau_x \otimes \tau_x)=c^\tau(g)(\tau_y\otimes\tau_y),\]  for any arrow $g$ of $\G$ from $x$ to $y$. We are then forced to set \[g\cdot \tau_x=\sqrt{c^\tau(g)}\tau_y.\] This defines an action of $\G$ on $\mathcal{D}_{TM}$. We check that it does not depend on the choice of $\tau$: if $\tau'$ is another strictly positive density on $M$, then $\tau'=f\tau$, where $f$ is a positive smooth function on $M$. Then \[g\cdot(\tau'_x\otimes\tau'_x)=f(x)^2g\cdot(\tau_x\otimes\tau_x)=f(x)^2c^{\tau}(g)(\tau_y\otimes \tau_y)=\left(\frac{f(x)}{f(y)}\right)^2c^{\tau}(g)(\tau'_y\otimes \tau'_y),\] and, on the other hand, \[g\cdot(\tau'_x\otimes\tau'_x)=c^{\tau'}(g)(\tau'_y\otimes \tau'_y),\] so $c^{\tau'}(g)=\left(\frac{f(x)}{f(y)}\right)^2c^{\tau}(g).$ Therefore \[\sqrt{c^{\tau'}(g)}\tau'_y=\frac{f(x)}{f(y)}\sqrt{c^{\tau}(g)}\tau'_y=f(x)\sqrt{c^\tau(g)}\tau_y=g\cdot \tau'_x.\]

Finally, since (\ref{DTM-sympl-case}) is now an isomorphism of representations and the action preserves positivity of densities, a density $\tau$ on $M$ is invariant with respect to the representation of $\G$ on $\cD_{TM}$ if and only if $\tau\otimes \tau$ is an invariant section of $\mathcal{D}_{A}^{\textrm{tr}}$. 
\end{proof}

\begin{remark}[on the existence] By the general theory (Proposition \ref{existence-tr-dens}), in the proper case, strictly positive transverse densities do exist. However, it is interesting to search for natural ones, associated to the  symplectic/Poisson geometry that is present in this context. Searching for a canonical $\sigma$ is the same thing as searching for a canonical $\tau$ or, setting $\rho= \tau^{\vee}$, searching for a Haar density 
\[ \rho\in \Gamma(\mathcal{D}_A) \]
(see Definition \ref{def-Haar-density} in the Appendix). With the intuition that such Haar densities are ``smooth'' families of densities along the orbits (see Corollary \ref{cor-rewrite-Haar} in the Appendix),  there is an obvious candidate: the one that uses the Liouville forms of the symplectic leaves! However, the resulting family is not smooth, unless we are in the regular case (and we choose the normalizations carefully). Nevertheless, in the general case one can still make sense of the resulting measure (just that it might not come from a transverse density); see \cite{PMCT3}. 
\end{remark}

Despite the previous remark, under a slightly stronger condition on $\G$, namely that $s: \cG\to M$ is proper, one does obtain a canonical transverse density. First one integrates along the $s$-fibers the Liouville density of the groupoid:
\[ \rho^{M}_{DH}:= \int_{s-\textrm{fibers}} \frac{|\Omega^\textrm{top}|}{\textrm{top}!} \]
(note that, since we do not assume any orientability, although the Liouville density comes from a volume form, integration over the $s$-fibers gives only  a density). As explained in \cite{PMCT1,PMCT2}, $\rho^{M}_{DH}$ is an invariant density on $M$; hence it induces (by the previous proposition) a canonical (strictly positive) transverse density for $\G$,
\[ \sigma_{DH}= \rho^{M}_{DH}\otimes \rho^{M}_{DH} \in \Gamma(\mathcal{D}_{A}^{\textrm{tr}}),\]
and an associated transverse measure $\mu_{DH}$ (here ``DH'' stands for Duistermaat-Heckman).

\begin{remark} Even when the procedure from the previous remark works, the resulting transverse density is different from $\sigma_{DH}$; these two are related by a Duistermaat-Heckman formula (see \cite{PMCT2} and also our next subsection). 
\end{remark}

%%%%%%%%%%%%%%%%%%%%%%%%%%%
%%%%%%%%%%%%%%%%%%%%%%%%%%%
%%%%%%%%%%%%%%%%%%%%%%%%%%%
%%%%%%%%%%%%%%%%%%%%%%%%%%%
%%%%%%%%%%%%%%%%%%%%%%%%%%%
\subsection{The regular case}\label{subsec-symplectic regular case}
%%%%%%%%%%%%%%%%%%%%%%%%%%%
%%%%%%%%%%%%%%%%%%%%%%%%%%%
%%%%%%%%%%%%%%%%%%%%%%%%%%%
%%%%%%%%%%%%%%%%%%%%%%%%%%%
%%%%%%%%%%%%%%%%%%%%%%%%%%%

In the regular case there is yet another description of transverse densities: using the normal bundle 
\[ \nu= TM/\F\]
of the symplectic foliation $\F$ associated to Poisson structure $\pi$ on $M$; recall that, as an involutive sub-bundle of $TM$, $\F$ is the image of $\pi^{\sharp}: T^*M\rightarrow TM$ (the Poisson bivector interpreted as a linear map). As for any foliation, $\nu$ comes with an action of the holonomy groupoid of $\cF$; hence one can talk about sections of 
$\mathcal{D}_{\nu}$ which are invariant under holonomy. 

We are in the setting of Subsection \ref{The regular case}: the symplectic groupoid $\G$ is regular, $\cF$ is the associated foliation and the resulting action of $\G$ on $\nu$ (see the subsection) factors through the holonomy action. One special feature of this case is that $\mathfrak{g}= \textrm{Ker}(\pi^{\sharp})$ is just the dual of $\nu$, as representations. Hence the isomorphism (\ref{reg-rewrite}) becomes
\begin{equation}\label{rewr-Poisson} 
\mathcal{D}_{A}^{\textrm{tr}} \cong \cD_{\nu}\otimes \cD_{\nu},
\end{equation}
an isomorphism of representations of $\G$. More explicitly: the Liouville forms induced by the leafwise symplectic forms define a section which trivializes $\mathcal{D}_{\F}$:
\[ \frac{|\omega^{\textrm{top}}|}{\textrm{top}!} \in \Gamma(\mathcal{D}_{\mathcal{F}})\]
Using again the identification $\mathcal{D}_{TM}= \mathcal{D}_{\F}\otimes \mathcal{D}_{\nu}$, we obtain an isomorphism
\[ \mathcal{D}_{TM}\cong   \mathcal{D}_{\nu}, \ \  \frac{|\omega^{\textrm{top}}|}{\textrm{top}!} \otimes \beta  \longleftrightarrow \beta.\]
%\[ \cD_{A}^{\textrm{tr}}= \mathcal{D}_{TM}\otimes \mathcal{D}_{TM} \cong \mathcal{D}_{\nu}\otimes \mathcal{D}_{\nu}

We now apply Proposition \ref{pro-kappa}; given the fact that $\mathfrak{g}= \mathfrak{\nu}^*$, the (positive) sections $\kappa$ of $\cD_{\mathfrak{g}}$ appearing in the proposition will be written as duals $\beta^{\vee}$ of sections $\beta$ of $\cD_{\nu}$. To avoid confusion, the resulting map $\kappa^{\vee}$  (now given by pairing with $\beta$) will be denoted $\langle \beta, \cdot \rangle$.  Combining the proposition also with Proposition \ref{Poisson-first}, we find:

\begin{proposition} Consider the relations
 \[    \Gamma(\cD_{TM})\ni \tau= \frac{|\omega^{\mathrm{top}}|}{\mathrm{top}!} \otimes \beta \in \Gamma(\cD_{\cF}\otimes \cD_{\nu})\]
\[    \Gamma(\cD_{A}^{\textrm{tr}})\ni \sigma = \tau\otimes \tau\in \Gamma(\mathcal{D}_{TM}\otimes \mathcal{D}_{TM})\]
and $\sigma\equiv \beta\otimes \beta$ modulo (\ref{rewr-Poisson}). These induce bijections between  strictly positive:
\begin{enumerate}
\item[1.] transverse densities $\sigma$ for $\G$.
\item[2.] densities $\tau$ on $M$ invariant under Hamiltonian flows.
\item[3.] sections $\beta$ of $\mathcal{D}_{\nu}$ invariant under holonomy ($=$ transverse densities for $\cE$). 
\end{enumerate}
Moreover, in this case the pairing with $\beta$ descends to an isomorphism 
\[ \langle \beta, \cdot \rangle : C_{c}^{\infty}(M//\G) \stackrel{\sim}{\rightarrow} C_{c}^{\infty}(M//\mathcal{E})\]
which relates the transverse measure $\mu_{\sigma}$ for $\G$ 
with the transverse measure $\mu_{\beta}$ for $\mathcal{E}$ through the commutative diagram
\[ \xymatrix{
 C_{c}^{\infty}(M//\G)\ar[rr]^-{\langle \beta, \cdot \rangle}_-{\sim} \ar[rd]_-{\mu_{\sigma}} & & C_{c}^{\infty}(M//\mathcal{E})\ar[ld]^-{\mu_{\beta}}\\
 & \mathbb{R} & }\]
\end{proposition}

%%%%%%%%%%%%%%%%%%%%%%%%%%%
%%%%%%%%%%%%%%%%%%%%%%%%%%%
%%%%%%%%%%%%%%%%%%%%%%%%%%%
%%%%%%%%%%%%%%%%%%%%%%%%%%%
%%%%%%%%%%%%%%%%%%%%%%%%%%%
\subsection{The proper regular case}\label{subsec-symplectic proper regular case} 
%%%%%%%%%%%%%%%%%%%%%%%%%%%
%%%%%%%%%%%%%%%%%%%%%%%%%%%
%%%%%%%%%%%%%%%%%%%%%%%%%%%
%%%%%%%%%%%%%%%%%%%%%%%%%%%
%%%%%%%%%%%%%%%%%%%%%%%%%%%
Under the condition that the symplectic groupoid $\G$ is both regular and proper we can further specialize the discussion from Subsection \ref{The proper regular case} and use Corollary \ref{cor-reg-proper}. In this case we have at our disposal the Haar densities associated to the identity components of the isotropy groups $\G_x$, which give rise to a strictly positive section
\[ \beta_{\mathrm{Haar}}\in \Gamma(\cD_{\nu}).\]
This is just the dual of $\kappa_{\mathrm{Haar}}$ from Subsection \ref{The proper regular case} (as above, in the Poisson case we pass from $\mathfrak{g}= \nu^*$ to $\nu$).

\begin{corollary}\label{corr-sympl-gpd-affine} Any proper regular symplectic groupoid $(\G, \Omega)$ carries a canonical transverse density: the $\sigma$ that corresponds to $\beta_{\mathrm{Haar}}\otimes \beta_{\mathrm{Haar}}$ modulo the isomorphism (\ref{rewr-Poisson}) or, equivalently, corresponding to the density on $M$ given by 
\[  \frac{|\omega^{\textrm{top}}|}{\textrm{top}!}\otimes \beta_{\mathrm{Haar}}.\]
Moreover, the measure $\mu_{\sigma}$ induced by $\sigma$ on $B= M//\G$ coincides with the measure $\mu_{\beta}$ induced by $\beta$ on $B= M//\cE$. 

\[ \xymatrix{
 & C_{c}^{\infty}(B) & \\
 C_{c}^{\infty}(M//\G)\ar[rr]^{\langle \beta_{\mathrm{Haar}}, \cdot\rangle}_{\sim} \ar[rd]_-{\mu_{\sigma}= \mu_{\mathrm{aff}}}
 \ar[ru]^{\mathrm{Av}_{\mathcal{\cG}}}_{\sim} & & C_{c}^{\infty}(M//\mathcal{E})\ar[ld]^-{\mu_{\beta_{\mathrm{Haar}}}} \ar[lu]_{\mathrm{Av}_{\mathcal{\cE}}}^{\sim}\\
 & \mathbb{R} & }\]
\end{corollary}

\begin{remark} Let us also point out that, in the case of symplectic groupoids, the Haar density $\beta_{\mathrm{Haar}}$  has a nice Poisson geometric interpretation. For instance, if the symplectic groupoid has 1-connected $s$-fibers, then the variation of the leafwise symplectic areas gives rise (because of properness) to a lattice in $\nu$ (a transverse integral affine structure), see e.g. \cite{alan-volume}; then $\beta_{\mathrm{Haar}}$ is just the corresponding density. If also the leaves are 1-connected - which ensures that $B= M/\G$ is smooth, then this induces an integral affine structure on $B$ and $\beta_{\mathrm{Haar}}$ is the canonical density associated to it. A similar interpretation holds for general symplectic groupoids $\G$ - see \cite{PMCT2}. For this reason, the resulting measure on $B$ is called the affine measure induced by $\G$, denoted $\mu_{\mathrm{aff}}$. This is related to the Duistermaat-Heckman measure from Subsection \ref{The general case} by a Duistermaat-Heckman formula - see \cite{PMCT2}. 
\end{remark}

One can use Proposition \ref{prop-tr-dens-B} to obtain a Weyl-type integration formula for $\mu_{\mathrm{aff}}$. 
One obtains the following result from \cite{PMCT2} (but with a different proof):

\begin{corollary} For the affine measure $\mu_{\mathrm{aff}}$ on $B$, denoting by $\mu_{M}$ the measure on $M$ induced by the corresponding density on $M$ ($\tau$ above), one has 
\[ \int_M f(x)\ d\mu_{M}(x) =\int_B \iota(b) \left(\int_{\O_b}f(y)\ d\mu_{\O_b}(y) \right)\ d\mu_{\mathrm{aff}}(b),\]
for all $f\in C_{c}^{\infty}(M)$, where $\mu_{\O}$ denotes the Liouville measure of the symplectic leaf $\mathcal{O}$, $\iota(b)= \iota(x)$
is the number of connected components of $\G_x$ with $x\in \mathcal{O}_b$ (any). 
\end{corollary}

\begin{proof} 
We just have to be careful with computing the resulting densities on the orbits: they arise when looking at the principal $\G_x$-bundles $t: s^{-1}(x)\rightarrow \mathcal{O}_x$ and decomposing the resulting density on $s^{-1}(x)$ using the Haar density associated to $\G_x$ (see (\ref{app-Haar-subtelty}) in the Appendix). This differs from the density $\beta_{\mathrm{Haar}}$ that we used before precisely by the factor $\iota(x)$ (see Example \ref{ex-Haar0}). And the final conclusion is that the resulting measure on $\mathcal{O}_x$ is $\iota(x)$ times the Liouville measure $\mu_{\mathcal{O}_x}$. 
\end{proof}

\begin{corollary} If $\G$ is compact then the affine volume of $B= M/\G$ is given by
\[ \mathrm{Vol}(B, \mu_{\mathrm{aff}})=\int_M \frac{1}{\iota(x) \mathrm{Vol}(\O_x, \mu_{\O_x})}\ d\mu_{M}(x) .\]
\end{corollary}

Finally, we point out that Proposition \ref{weinst-conj0} immediately implies Conjecture 5.2 from \cite{alan-volume} (even without the simplifying assumptions from {\it loc.cit}). 

\begin{corollary} Let $\sigma$ be a transverse density for the regular proper symplectic groupoid $(\G, \Omega)$ (denoted $\lambda$ in \cite{alan-volume}). Write $\sigma= \tau\otimes \tau$ with $\tau$ an invariant density on $M$ and write $\tau= 
\frac{|\omega^{\textrm{top}}|}{\textrm{top}!}\otimes \beta$ with $\beta\in \Gamma(\cD_{\nu})$ invariant. Then $\mu_{\sigma}= \mu_{\beta}$ (an equality of measures on $B$) if and only if $\beta= \beta_{\mathrm{Haar}}$.
\end{corollary}

%%%%%%%%%%%%%%%%%%%%%%%%%%%
%%%%%%%%%%%%%%%%%%%%%%%%%%%
%%%%%%%%%%%%%%%%%%%%%%%%%%%
%%%%%%%%%%%%%%%%%%%%%%%%%%%
%%%%%%%%%%%%%%%%%%%%%%%%%%%
\section{Stokes formula and (Ruelle-Sullivan) currents}\label{sec-Stokes formula and (Ruelle-Sullivan) currents}
%%%%%%%%%%%%%%%%%%%%%%%%%%%
%%%%%%%%%%%%%%%%%%%%%%%%%%%
%%%%%%%%%%%%%%%%%%%%%%%%%%%
%%%%%%%%%%%%%%%%%%%%%%%%%%%
%%%%%%%%%%%%%%%%%%%%%%%%%%%
We now return to the general theory. In this section we point out that, as in the case of our motivating example of foliations, transverse measures for groupoids give rise to closed $r$-currents on the base manifold $M$, where $r$ is the dimension of the $s$-fibers (or the rank of the Lie algebroid); one advantage of Haefliger's approach is that it makes such constructions rather obvious. 

%%%%%%%%%%%%%%%%%%%%%%%%%%%
%%%%%%%%%%%%%%%%%%%%%%%%%%%
%%%%%%%%%%%%%%%%%%%%%%%%%%%
%%%%%%%%%%%%%%%%%%%%%%%%%%%
%%%%%%%%%%%%%%%%%%%%%%%%%%%
\subsection{Stokes formula and Poincare duality for usual densities}\label{subsec-Stokes formula and Poincare duality for usual densities} 
%%%%%%%%%%%%%%%%%%%%%%%%%%%
%%%%%%%%%%%%%%%%%%%%%%%%%%%
%%%%%%%%%%%%%%%%%%%%%%%%%%%
%%%%%%%%%%%%%%%%%%%%%%%%%%%
%%%%%%%%%%%%%%%%%%%%%%%%%%%

The main property of the canonical integration of densities (\ref{eq-can-int}), which distinguishes it from other linear functionals on $C^{\infty}_c(M, \mathcal{D}_M)$, is the Stokes formula. To state it for general densities, one first reinterprets the sections of $\mathcal{D}_M$ as top-forms with values in the orientation bundle, 
\begin{equation}\label{int-dens-bdle} 
C^{\infty}_{c}(M, \mathcal{D}_{TM})= \Omega^{\textrm{top}}_{c}(M, \mathfrak{o}_{M}),
\end{equation}
The orientation bundle comes with a flat connection.
By a flat vector bundle we mean a vector bundle $E$ together with a fixed flat connection 
\[ \nabla: C^{\infty}(M,TM)\times C^{\infty}(M,E)\rightarrow C^{\infty}(M,E), \ (X, e)\mapsto \nabla_{X}(e).\]
They provide  the geometric framework for local coefficients; the main point for us is that, for such $E$, one has the spaces of $E$-valued forms
\begin{equation}\label{flat-conn}
 \Omega^{\bullet}(M, E)= C^{\infty}(M,\Lambda^{\bullet}T^*M\otimes E),
\end{equation}
the flat connection $\nabla$ gives rise to a DeRham operator $d_{\nabla}$ on $\Omega^{\bullet}(M, E)$ given explicitly by the standard Koszul-type formula:
\begin{align} \label{Koszul-form}
\d_{\nabla}\omega(X_1, \ldots , X_{k+1})=&\sum_{i}(-1)^{i+1} \nabla_{X_i}(\omega(X_1, \ldots , \widehat{X_i}, \ldots , X_{k+1}))+ \\
& + \sum_{i< j} (-1)^{i+j}\omega([X_i, X_j], \ldots , \widehat{X_i}, \ldots, \widehat{X_j}, \ldots , X_{k+1}), \nonumber
\end{align}
and the flatness of $\nabla$ is equivalent to $d_{\nabla}^{2}= 0$. Therefore one can talk about DeRham cohomology with coefficients in $E$, denoted $H^{\bullet}(M, E)$.
Back to densities, one uses the DeRham differential with coefficients in $\mathfrak{o}_M$ 
\[ d: \Omega^{\textrm{top}-1}_{c}(M, \mathfrak{o}_{M})\rightarrow \Omega^{\textrm{top}}_{c}(M, \mathfrak{o}_{M}),\]
and the Stokes formula for the canonical integration reads, via (\ref{int-dens-bdle}):
\[ \int_{M} d \omega= 0 \ \ \ \textrm{for\ all}\ \omega\in \Omega^{\textrm{top}-1}_{c}(M, \mathfrak{o}_{M}).\]
Equivalently, $\int_{M}$ descends to a linear map
\[ \int_{M}:  H^{\mathrm{top}}_{c}(M, \mathfrak{o}_{M}) \rightarrow \mathbb{R}.\]
The fact that the domain is always $1$-dimensional shows that the Stokes formula characterizes the integration of densities uniquely, up to multiplication by scalars.

%%%%%%%%%%%%%%%%%%%%%%%%%%%
%%%%%%%%%%%%%%%%%%%%%%%%%%%
%%%%%%%%%%%%%%%%%%%%%%%%%%%
%%%%%%%%%%%%%%%%%%%%%%%%%%%
%%%%%%%%%%%%%%%%%%%%%%%%%%%
\subsection{Stokes for transverse measures}\label{subsec-Stokes for transverse measures}
%%%%%%%%%%%%%%%%%%%%%%%%%%%
%%%%%%%%%%%%%%%%%%%%%%%%%%%
%%%%%%%%%%%%%%%%%%%%%%%%%%%
%%%%%%%%%%%%%%%%%%%%%%%%%%%
%%%%%%%%%%%%%%%%%%%%%%%%%%%

For a Lie algebroid $A$, the basic constructions that allows us to talk about the Stokes formula have an obvious $A$-version, mainly by replacing the tangent bundle $TM$ by $A$; one obtains the
notion of $A$-flat vector bundle $(E, \nabla)$ 
(implement the mentioned replacement in (\ref{flat-conn})), 
also known as representations of $A$, $A$-differential forms with values in $E$, $\Omega^{\bullet}(A, E)= C^{\infty}(M,\Lambda^{\bullet} A^*\otimes E)$, DeRham differential given by the same Koszul formula as above, the cohomology of $A$ with coefficients in $E$, $H^{\bullet}(A, E)$; for details, see e.g. \cite{VanEst,QA}. Of course, considering only compactly supported sections one obtains the cohomology with compact supports (with coefficients in an arbitrary representation $E$), denoted
\[ H^{\bullet}_{c}(A, E). \]
Via the anchor map $\sharp: A\rightarrow TM$, any flat vector bundle over $M$ can be seen as a representation of $A$ (this applies in particular to vector bundles of type $\mathfrak{o}_{E}$ where $E$ is any vector bundle over $M$). % $\sharp$ induces a map in cohomology
% \[ \sharp^*: H^{\bullet}(M, E) \rightarrow H^{\bullet}(A, E).\]
One has the analogue of (\ref{int-dens-bdle}) describing the domain of definition for transverse measures:
\[ C^{\infty}_c(M, \mathcal{D}_A)= \Omega^{\textrm{top}}_{c}(A, \mathfrak{o}_A)\]
where $\textrm{top}$ stands for the rank of $A$. Therefore we will use the flat vector bundle $\mathfrak{o}_{A}$ and the DeRham operator:
\[ d_{A}:  \Omega^{\textrm{top}-1}_{c}(A, \mathfrak{o}_{A})\rightarrow \Omega^{\textrm{top}}_{c}(A, \mathfrak{o}_{A}).\]

\begin{proposition}\label{pre-VE-iso-deg-0}
Any transverse measure $\mu$ for $\G$ a Lie groupoid satisfies the Stokes formula
\[ \mu(d_A\omega)= 0 \ \ \ \forall\ \omega\in \Omega^{\mathrm{top}-1}_{c}(A, \mathfrak{o}_{A}).\]
Equivalently, but at the level of cohomology: the quotient map 
\[ \Omega^{\mathrm{top}}_{c}(A, \mathfrak{o}_A)= C^{\infty}_c(M, \mathcal{D}_A)\rightarrow C_{c}^{\infty}(M//\G)\]
descends to a surjection
\begin{equation}\label{VE-hidden} 
VE: H^{\mathrm{top}}_{c}(A, \mathfrak{o}_A) \rightarrow C_{c}^{\infty}(M//\G) .
\end{equation}
\end{proposition}

The proof will be given in Subsection \ref{The missing proofs}. 
Similarly to the classical case, this implies that a transverse measure $\mu$ descends to a linear functional on $H^{\textrm{top}}_{c}(A, \mathfrak{o}_A)$ hence, 
composing with the pairing induced by wedge products, one obtains 
\[ \langle \cdot, \cdot \rangle_{\mu}:  H^{k}(A) \times H^{\mathrm{top}-k}_{c}(A, \mathfrak{o}_A) \stackrel{\wedge}{\rightarrow} H^{\mathrm{top}}_{c}(A, \mathfrak{o}_A)\stackrel{\mu}{\rightarrow} \mathbb{R},\]
called \textbf{the Poincare pairing induced by $\mu$}. In many interesting examples this pairing is non-degenerate; however, that cannot happen in general since $VE$ may fail to be an isomorphism
(but see Theorem \ref{VE-iso-deg-0} below).

%%%%%%%%%%%%%%%%%%%%%%%%%%%
%%%%%%%%%%%%%%%%%%%%%%%%%%%
%%%%%%%%%%%%%%%%%%%%%%%%%%%
%%%%%%%%%%%%%%%%%%%%%%%%%%%
%%%%%%%%%%%%%%%%%%%%%%%%%%%
\subsection{Reformulation in terms of (Ruelle-Sullivan) currents}\label{subsec-Reformulation in terms of (Ruelle-Sullivan) currents}
%%%%%%%%%%%%%%%%%%%%%%%%%%%
%%%%%%%%%%%%%%%%%%%%%%%%%%%
%%%%%%%%%%%%%%%%%%%%%%%%%%%
%%%%%%%%%%%%%%%%%%%%%%%%%%%
%%%%%%%%%%%%%%%%%%%%%%%%%%% 

The previous discussion on the integration of densities can be slightly reformulated using currents \cite{derham}. 
Recall that a \textbf{$p$-current} on a manifold $M$ is a continuous linear map
\[ \xi: \Omega_{c}^{p}(M) \rightarrow \mathbb{R}, \]
where the continuity is in the distributional sense, i.e., it uses the inductive limit topology arising from writing
\[  \Omega_{c}^{p}(M)= \cup_{K-\textrm{compact}}\Omega_{K}^{p}(M),\]
where the space $\Omega_{K}^{p}(M)$ of $p$-forms supported in $K$ is endowed with the topology of uniform convergence for all derivatives. The resulting spaces $\Omega_{p}(M)$ of $p$-currents fit into a chain complex
\[ \ldots  \stackrel{d^*}{\rightarrow} \Omega_{1} (M) \stackrel{d^*}{\rightarrow} \Omega_{0}(M)\]
where $d^{*}(\xi)= \xi\circ d$. By construction, the resulting homology $H_{\bullet}(M)$ is in duality with the compactly supported DeRham cohomology; moreover, this pairing induces a canonical isomorphism
\begin{equation}\label{DeRham-duality} 
H_{\bullet}(M) \cong H^{\bullet}_{c}(M)^{*} .
\end{equation}
Similarly one talks about currents on $M$ with values in a flat vector bundle $E$ and the homology $H_{\bullet}(M, E)$.
Back to densities we see that the canonical integration becomes a $\textrm{top}$-current on $M$ with coefficients in $\mathfrak{o}_{M}$, $\int_{M}\in \Omega_{\textrm{top}}(M, \mathfrak{o}_{M})$, 
and the Stokes formula says that this is a closed current. Hence it gives rise to a completely canonical homology class
\begin{equation}\label{fund-class-gen} 
\left[\int_{M}\right]\in H_n(M, \mathfrak{o}_{M}) ,
\end{equation}
where $M$ is assumed to be connected, of dimension $n$. When $M$ is compact and oriented, the orientation trivializes $\mathfrak{o}_{M}$ and $H_{\bullet}(M)$  is canonically isomorphic to singular homology (associate to a singular $p$-chain $\sigma: \Delta^p\rightarrow \mathbb{R}$ the $p$-current $C_{\sigma}(\omega)= \int_{\Delta^p} \sigma^*\omega$); with these identifications, $[\int_{M}]$  becomes the standard fundamental class $[M]\in H_n(M)$ of the compact oriented manifold $M$.

\begin{remark} \rm \ Note also that, via the isomorphism (\ref{DeRham-duality}), the Poincare duality can now be stated in terms of currents as an isomorphism
\begin{equation}\label{P-d-currents} 
H^{\bullet}(M)\cong H_{n-\bullet}(M, \mathfrak{o}_{M}).
\end{equation}
In turn, this follows easily by a sheaf-theoretic argument. Indeed, the complex computing $H_{\bullet}(M, \mathfrak{o}_{M})$ can be arranged into a chain complex augmented by $\mathbb{R}$,
\[ \mathbb{R}\rightarrow \Omega_n(M, \mathfrak{o}_{M})\rightarrow \Omega_{n-1}(M, \mathfrak{o}_{M}) \rightarrow \ldots \ ,\]
where the first map takes a scalar $\lambda$ to $\lambda \cdot \int_{M}$. This complex is local, i.e., can be seen as a complex of sheaves over $M$; as such, it is actually a resolution of $\mathbb{R}$ by fine sheaves (basically the Poincare lemma), therefore it computes the  
cohomology of $M$, giving rise to the Poincare duality isomorphism (\ref{P-d-currents}). 
\end{remark}

The basic constructions from the previous discussion have an obvious generalization to algebroids, allowing us to talk about the space $\Omega_{p}(A, E)$ of $A$-currents of degree $p$ with coefficients in a representation $E$ of $A$ and the resulting homology $H_{\bullet}(A, E)$. With this, a transverse measure $\mu$ 
can be interpreted as a top $A$-current, $\mu\in \Omega_{\textrm{top}}(A, \mathfrak{o}_A)$, and the Stokes formula can be reformulated as:

\begin{corollary} (reformulation of Proposition \ref{pre-VE-iso-deg-0}) Any transverse measure, interpreted as a top $A$-current, is closed.
\end{corollary}

In particular, when $A$ is oriented, composing with the anchor $\sharp$ induces a map
\[ \sharp_*: \Omega_{r}(A) \rightarrow \Omega_r(M) \ \ \ \ (r= \textrm{the\ rank\ of}\ A),\]
hence any transverse measure $\mu$ of $\G$ induces an $r$-current on $M$, called \textbf{the Ruelle-Sullivan current} on $M$ induced by $\mu$. 
Of course, the standard notion \cite{ruelle} is obtained in the case of foliations. 
Finally, we have the following converse of the previous corollary. 

\begin{proposition} \label{VE-iso-deg-0}
If $\G$ is a Lie groupoid with connected $s$-fibers then (\ref{VE-hidden}) is an isomorphism. As a consequence, transverse measures on $\G$ are the same thing as closed $\textrm{top}$ $A$-currents with coefficients in $\mathfrak{o}_A$,
% \[ \mu: \Omega^{\textrm{top}}_{c}(A, \mathfrak{o}_A) \rightarrow \mathbb{R}, \] 
\[ \mu\in \Omega_{\mathrm{top}}(A, \mathfrak{o}_A), \] 
which are positive in the sense that, if $\rho\in \Omega_{c}^{\mathrm{top}}(A, \mathfrak{o}_A)= C^{\infty}_c(M, \mathcal{D}_A)$ is positive as a density, then $\mu(\rho)\geq 0$. 
\end{proposition}

The proof of this result and of Proposition \ref{pre-VE-iso-deg-0} are given in the next section.

%%%%%%%%%%%%%%%%%%%%%%%%%%%
%%%%%%%%%%%%%%%%%%%%%%%%%%%
%%%%%%%%%%%%%%%%%%%%%%%%%%%
%%%%%%%%%%%%%%%%%%%%%%%%%%%
%%%%%%%%%%%%%%%%%%%%%%%%%%%
%%%%%%%%%%%%%%%%%%%%%%%%%%%
%%%%%%%%%%%%%%%%%%%%%%%%%%%
%%%%%%%%%%%%%%%%%%%%%%%%%%%
%%%%%%%%%%%%%%%%%%%%%%%%%%%
%%%%%%%%%%%%%%%%%%%%%%%%%%%
\section{Cohomological insight and the Van Est isomorphism}
\label{sec:cohomological insight}
%%%%%%%%%%%%%%%%%%%%%%%%%%%
%%%%%%%%%%%%%%%%%%%%%%%%%%%
%%%%%%%%%%%%%%%%%%%%%%%%%%%
%%%%%%%%%%%%%%%%%%%%%%%%%%%
%%%%%%%%%%%%%%%%%%%%%%%%%%%
%%%%%%%%%%%%%%%%%%%%%%%%%%%
%%%%%%%%%%%%%%%%%%%%%%%%%%%
%%%%%%%%%%%%%%%%%%%%%%%%%%%
%%%%%%%%%%%%%%%%%%%%%%%%%%%
%%%%%%%%%%%%%%%%%%%%%%%%%%%

Another way to look at the intrinsic model $C_{c}^{\infty}(M//\G)$ is via a compactly supported version of differentiable cohomology; then Propositions \ref{pre-VE-iso-deg-0} and \ref{VE-iso-deg-0} will become part of a
compactly supported version of the Van Est isomorphism \cite{VanEst}. Here we give the details. 

%%%%%%%%%%%%%%%%%%%%%%%%%%%
%%%%%%%%%%%%%%%%%%%%%%%%%%%
%%%%%%%%%%%%%%%%%%%%%%%%%%%
%%%%%%%%%%%%%%%%%%%%%%%%%%%
%%%%%%%%%%%%%%%%%%%%%%%%%%%
\subsection{Differentiable cohomology with compact supports}\label{Differentiable cohomology with compact supports} 
%%%%%%%%%%%%%%%%%%%%%%%%%%%
%%%%%%%%%%%%%%%%%%%%%%%%%%%
%%%%%%%%%%%%%%%%%%%%%%%%%%%
%%%%%%%%%%%%%%%%%%%%%%%%%%%
%%%%%%%%%%%%%%%%%%%%%%%%%%%

A version of differentiable cohomology with compact supports was briefly sketched in Remark 4 of \cite{VanEst}. While {\it loc.cit}  made use of a Haar system, in our context, there is a clear intrinsic model. For each $k\geq 0$ integer, we denote by $\G_k$ the manifold consisting of strings 
\begin{equation}\label{k-strings} 
x_0 \stackrel{g_1}{\leftarrow} x_1 \ldots x_{k-1} \stackrel{g_k}{\leftarrow} x_k
\end{equation} 
of $k$ composable arrows of $\G$. Recall that these spaces are related by the face maps

\[  \delta_i: \G_k\rightarrow \G_{k-1}, \ \delta_i(g_1, \ldots, g_k)= \left\{ \begin{array}{lll} 
       (g_2, \ldots , g_k) & \mbox{if $i= 0$}\\
       (g_1, \ldots , g_ig_{i+1}, \ldots , g_k) & \mbox{if $1\leq i\leq k-1$}\\
       (g_1, \ldots , g_{k-1}) & \mbox{if $i= k$} 
                                                        \end{array}
                                             \right. 
\]
On each of the spaces $\G_k$ one considers the line bundle 
\[ \mathcal{D}^{k+1}= p_{0}^{*}\mathcal{D}_A \otimes \ldots \otimes p_{k}^{*}\mathcal{D}_A\]
where $p_i: \G_k\rightarrow M$ ($0\leq i\leq k$) takes a $k$-string (\ref{k-strings}) to $x_i$. The spaces $C^{\infty}_{c}(\G_k, \mathcal{D}^{k+1})$ come with the differential
\[ \delta= \sum_{i= 0}^{} (-1)^i \delta_{i\, !}: C^{\infty}_{c}(\G_k, \mathcal{D}^{k+1}) \rightarrow C^{\infty}_{c}(\G_{k-1}, \mathcal{D}^{k}).\]
Since we end up with a chain complex we will talk about the \textbf{differentiable homology} of $\G$ and to use the notation $H_{\bullet}^{\textrm{diff}}(\G)$. However, it is sometimes useful to think of it as a cohomology with compact supports; then we use the notation
\[ H^{\bullet}_{c}(\G):= H_{r-\bullet} (\G),\]
where $r$ is the rank of the Lie algebroid of $\G$. For instance, in this way, the Van Est map will become a map between cohomologies with compact support, preserving the degree.
With these, $C_{c}^{\infty}(M//\G)$ is simply the homology in degree zero:
\[ C_{c}^{\infty}(M//\G)= H_{0}^{\textrm{diff}}(\G)= H^{r}_{c}(\G). \]
Let us also remark that, as in the case of differentiable cohomology, there is also a version $H^{\bullet}_{c}(\G, E):= H_{r-\bullet} (\G, E)$ with coefficients in a representation $E$ of $\G$. The only subtlety is that (still as in the case of differentiable cohomology) one has to use the action of $E$ in order to define the differential. For instance, in the lowest degree, while the integration over the $s$-fibers with coefficients, 
\[ s_{!}^{E}: C_{c}^{\infty}(\G, s^*E\otimes t^*\mathcal{D}_A\otimes s^* \mathcal{D}_A)\rightarrow C_{c}^{\infty}(M, E\otimes \mathcal{D}_A),\]
is defined exactly as (\ref{S-!})), for $t_{!}^{E}$ one first composes with the isomorphism $s^*E\cong t^*E$ induced by the action of $\G$ on $E$ (so that $s_{!}^{E}$ and $t_{!}^{E}$ are defined on the same space). Similarly in higher degrees.

\begin{example}[the case of submersions, continued]\label{ex-case-submersions-c}\rm \ 
Let us continue the discussion for the groupoid $\G(\pi)$ associated to a submersion $\pi: P\rightarrow B$ (see Example \ref{ex-case-submersions}). For this groupoid, the representations of $\G(\pi)$ are simply pull-backs of vector bundles $E$ over $B$ endowed with the tautological action (the arrow from $x$ to $y$, which exists only when $\pi(x)= \pi(y)$, sends $(\pi^*E)_x= E_{\pi(x)}$ to  $(\pi^*E)_y= E_{\pi(y)}$ by the identity map). Note that this can also be seen as a consequence of the fact that $\G(\pi)$ is Morita equivalent to $B$ viewed as a groupoid with only unit arrows. 

While Lemma \ref{conf1} (and its obvious version with coefficients) can be seen as a computation of degree zero homology, we have the following (itself a particular case of Morita invariance, but used in the proof of the main theorem):

\begin{lemma}\label{conf2} For the groupoid $\G(\pi)$ associated to a submersion $\pi: P\rightarrow B$, 
\[ H^{k}_{c}(\G(\pi), E)= H_{r- k}(\G(\pi), E)\cong 
\left\{
\begin{array}{cc} 
C_{c}^{\infty}(B, E)  &  \textrm{if}\ k= r \\ 
0  &  \textrm{otherwise}
\end{array}
\right.\] where $r$ is the rank of the Lie algebroid of $\G$.
\end{lemma}

This lemma is it the compactly supported version of Lemma 1 from \cite{VanEst}, and it follows by the same type of arguments as there. 
\end{example}

\begin{theorem}\label{thm-Mor-inv}  The homology $H_{\bullet}(\G, \cdot)$ is Morita invariant. 
\end{theorem}

\begin{proof} The proof goes exactly as the proof of the Morita invariance of differentiable cohomology from \cite{VanEst}: given a Morita equivalence between $\G$ (over $M$) and $\H$ (over $N$), i.e., a principal bi-bundle $P$ (with $\alpha: P\rightarrow M$, $\beta: P\rightarrow N$) one forms a double complex $C_{\bullet, \bullet}(P)$ together with quasi-isomorphisms
\[ C_{\bullet}(\G) \stackrel{\int_{\alpha}}{\leftarrow} \textrm{tot}(C_{\bullet, \bullet}(P))  \stackrel{\int_{\beta}}{\rightarrow} C_{\bullet}(\H);\]
the fact that the two maps are quasi-isomorphism is ensured by the fact that each column $C_{p, \bullet}(P)$ comes with an augmentation $\int_{\alpha}: C_{p, \bullet}(P) \rightarrow C_{p}(\G)$ compatible with the differentials, and similarly for the rows. Explicitly, $C_{p, q}(P)$ is built on the subspace $P_{p, q}\subset \G^p\times P\times \H^q$ consisting of elements $(g_1, \ldots, g_p, x, h_1, \ldots, h_q)$ with the property that the product $g_1\ldots g_pxh_1\ldots h_q$ is defined. 
On each such space one has a line bundle $\mathcal{D}_{p, q}$ defined as follows: one pulls-back the line bundle $\mathcal{D}_{\G}^{p+1}$ that appears in the definition of $C_{\bullet}(\G)$ via the projection $P_{p, q}\rightarrow \G_p$, similarly for $\mathcal{D}_{\H}^{q+1}$; then, on $P$, one has the pull-back algebroid $C= \alpha^{!}A\cong \beta^{!}B$ (see e.g. Example 5 in \cite{VanEst}) and one pulls-back $\mathcal{D}:= \mathcal{D}_{C}$ to $P_{p, q}$; with these, 
\[ \mathcal{D}_{p, q}= \mathcal{D}_{\G}^{p+1} \otimes \mathcal{D}\otimes \mathcal{D}_{\H}^{q+1}.\]
define $C_{p, q}(P)= \Gamma_{c}(\mathcal{D}_{p, q})$. 
%  is the space of compactly supported sections of this bundle. 
Fixing $p$, the differential of the $p$-column is defined so that it becomes the differentiable complex with compact supports of the action groupoid associated to the right action of $\H$ on $P_{p, 0}$; since this action is principal, the previous lemma implies that the cohomology of the $p$-column is zero everywhere, except in degree zero where it is $C_p(\G)$, with the isomorphism induced by $\int_{\alpha}$. Similarly for the rows, but using $\H$ instead of $\G$. Hence the conclusion follows.  
\end{proof}

\begin{remark}\rm \ Of course, the same holds in the presence of coefficients. Also, with arguments similar to the ones from \cite{VanEst}, these isomorphisms are compatible with respect to the tensor product of bi-bundles (hence they are functorial). Even more, if $P$ is only principal as an $\H$-bundle, i.e., a generalized morphism from $\G$ to $\H$, then a careful look at the previous argument (especially at the coefficients) gives rise to the ``integration over the $P$-fibers map'' (which is functorial in $P$). 
\end{remark}

\begin{remark} [Co-invariants] \label{remark-coinvariants}\rm \ The definition of $C_{c}^{\infty}(M//\G)$ can also be thought of as making sense of ``the space of $\G$-coinvariants associated to $C_{c}^{\infty}(M)$'' (see Example \ref{exemplification}). This viewpoint becomes important when looking for compactly supported versions of statements with no conditions on the support. For that purpose, we use a more suggestive notation and we extend the notion to a slightly larger setting. First of all, for a representation $E$ of $\G$ define the space of coinvariants $C_{c}^{\infty}(M, E)_{\G-\textrm{coinv}}$ (dual to the obvious space $C^{\infty}(M, E)^{\G-\textrm{inv}}$ of invariants) as
\[ C^{\infty}_c(M, E)_{\G-\textrm{coinv}}:= H_{0}(\G, E).\]
Morally, while representations $E$ of $\G$ can be thought of as vector bundles $E/\G$ over $M/\G$, the space $C^{\infty}(M, E)^{\G-\textrm{inv}}$ plays the role of ``$C^{\infty}(M/\G, E/\G)$'' and similarly  $C^{\infty}_{c}(E)_{\G-\textrm{coinv}}$ plays the role of ``$C^{\infty}_c(M/\G, E/\G)$''. These heuristics become precise in the case of the groupoid $\G(\pi)$ associated to a submersion $\pi: P\rightarrow B$ when, by Lemma \ref{conf1}, as the fiber integration induces an isomorphism
\begin{equation}\label{sub-case-extra} 
C^{\infty}_c(\pi^*E)_{\G(\pi)-\textrm{coinv}} \cong C^{\infty}_c(B, E). 
\end{equation}
We will use the same notations in a slightly more general context: when $\G$ acts on a manifold $P$ (say from the right) and $E$ is a $\G$-equivariant vector bundle over $P$.
To put ourselves in the previous setting, one considers the associated action groupoid $P\rtimes \G$ and interprets $E$ as a representation of it. The resulting space of $P\rtimes \G$-coinvariants will still be denoted by $C^{\infty}_{c}(P, E)_{\G-\textrm{coinv}}$. 

\begin{corollary}\label{cor-coinv-princ} If $\pi: P\rightarrow B$ is a principal $\G$-bundle over a manifold  $B$, then $\G$-equivariant vector bundles on $P$ are necessarily of type $\pi^*E$ with $E$ a vector bundle over $B$ (so that the action on $\pi^*E$ is the tautological action) and 
\[ C^{\infty}_{c}(P, \pi^*E)_{\G-\mathrm{coinv}} \cong C^{\infty}_c(B, E).\]
Moreover, in degrees $k>0$, $H_{k}(P\rtimes \G, \pi^*E)= 0$. 
\end{corollary} 

\begin{proof} In this situation one has a diffeomorphism $P\times_{M}\G \cong P\times_{B} P$, given by $(p, g)\mapsto (p, pg)$, and this identifies the action groupoid with the groupoid $\G(\pi)$. The statement becomes that of Lemma \ref{conf2}. 
\end{proof}
\end{remark}

Next we mention another important property of the differentiable cohomology with compact supports: the Van Est theorem. 

\begin{theorem}\label{thm-Van-Est} For any Lie groupoid $\G$, with Lie algebroid denoted by $A$, and for any representation $E$ of $\G$, there is a canonical map
\[ VE^{\bullet}: H^{\bullet}_{c}(A, E\otimes \mathfrak{o}_A) \rightarrow H^{\bullet}_{c} (\G, E) .\]
Moreover, if $k\in \{0, 1, \ldots, r-1\}$ where $r$ is the rank of $A$ and if the fibers of $s$ are homologically $k$-connected, then $VE^{0}, \ldots , VE^{k}$ are isomorphisms. The same is true for $k\geq r$ except for the case when some $s$-fiber is compact and orientable. 
\end{theorem}

\begin{example}[the case of submersions, continued]\label{ex-case-submersions-coh}\rm \ 
Let us continue our discussion on the basic example associated to a submersion $\pi: P\rightarrow B$. In this case the algebroid is $A= \textrm{Ker}(d\pi)= \mathcal{F}(\pi)$ - the foliation induced by $\pi$. While the resulting differentiable cohomology vanishes except in degree $r$ (the rank of $A$), we are looking at a vanishing result for the foliated cohomology with compact supports. For trivial coefficients, the complex under discussion is 
\[ (\Omega^{\bullet}_{c}(\F(\pi), \mathfrak{o}_{\F(\pi)}), d_{\pi})\]
where $d_{\pi}= d_{\mathcal{F}(\pi)}$ is now just DeRham differentiation along the fibers of $\pi$. One may think that the role of the orientation bundle is to make the top degree
\[ \Omega^{\textrm{top}}_{c}(\F(\pi), \mathfrak{o}_{\F(\pi)})= C^{\infty}_{c}(P, \mathcal{D}_{\F(\pi)}), \]
the domain of the canonical fiber integration $\int_{\pi}$ with values in $C^{\infty}(B)$. For the version with coefficients, necessarily of type $\pi^*E$ for some vector bundle $E$ over $B$ (cf. Example \ref{ex-case-submersions-c}), $(\Omega^{\bullet}_{c}(\F(\pi), \mathfrak{o}_{\F(\pi)}\otimes \pi^*E), d_{\pi})$ uses the flat $\F(\pi)$-connection which is the infinitesimal counterpart of the tautological action of $\G(\pi)$ on $\pi^*E$; this is uniquely characterized by the condition that sections of type $\pi^*s$ are flat. Although the following is a particular case of the Theorem \ref{thm-Van-Est}, it will be used to prove the theorem.  Note that it clarifies why the degree $r$ is special: for an $r$-dimensional manifold $F$, $H^{0}_{c}(F, \mathfrak{o}_F)$ is non-zero if and only if $F$ is compact and orientable.

\begin{lemma}\label{conf3} For the foliation $\F(\pi)$ associated to a submersion $\pi: P\rightarrow B$, if the fibers of $\pi$ are homologically $k$-connected, where $k\in\{0, 1, \ldots, r-1\}$, then the following sequence is exact:
\begin{small}
\[  \Omega^{r-k-1}_{c}(\F(\pi), \mathfrak{o}_{\F(\pi)}\otimes \pi^*E)\stackrel{d_{\pi}}{\longrightarrow} \ldots \stackrel{d_{\pi}}{\longrightarrow} % \Omega^{r-1}_{c}(\F(\pi), \mathfrak{o}_{\F(\pi)})\stackrel{d_{\pi}}{\longrightarrow}
\Omega^{r}_{c}(\F(\pi), \mathfrak{o}_{\F(\pi)}\otimes \pi^*E)\stackrel{\int_{\pi}}{\longrightarrow} C_{c}^{\infty}(B, E) \rightarrow 0.\]
\end{small}

(hence the corresponding compactly supported cohomology is zero in degrees $r-k$, \ldots, $r-1$ and is $C_{c}^{\infty}(B, E)$ in degree $r$). The same is true for $k= r$  except for the case when some fiber of $\pi$ is compact and orientable. 
\end{lemma}

\begin{proof}
This is just the version of compact supports of Theorem 2 from \cite{VanEst} applied to the zero Lie algebroid; the first few lines of the proof in {\it loc.cit.} adapt immediately to compact supports. Alternatively, one can first show that the Poincare pairing (with respect to any positive measure on $B$) is non-degenerate and obtain our lemma as a consequence of the result from \cite{VanEst}. 
\end{proof}
\end{example}

%%%%%%%%%%%%%%%%%%%%%%%%%%%
%%%%%%%%%%%%%%%%%%%%%%%%%%%
%%%%%%%%%%%%%%%%%%%%%%%%%%%
%%%%%%%%%%%%%%%%%%%%%%%%%%%
%%%%%%%%%%%%%%%%%%%%%%%%%%%
% \subsection{Algebroid differential forms as invariant forms along the $s$-fibers of the groupoid: the compactly supported version}
\subsection{The compactly supported $A$-DeRham complex via forms along $s$-fibers}\label{subsec-The compactly supported A-DeRham complex via forms along s-fibers}
%%%%%%%%%%%%%%%%%%%%%%%%%%%
%%%%%%%%%%%%%%%%%%%%%%%%%%%
%%%%%%%%%%%%%%%%%%%%%%%%%%%
%%%%%%%%%%%%%%%%%%%%%%%%%%%
%%%%%%%%%%%%%%%%%%%%%%%%%%% 

The definition of algebroid cohomology often comes with the remark that the defining complex can be interpreted as the complex of right-invariant forms along the $s$-fibers of $\G$. This makes the complex $\Omega^{\bullet}(A)$ into a subcomplex of the DeRham complex associated to the foliation $\F(s):=\Ker(ds)$ induced by $s$ (fiberwise differential forms): 
\[ t^*: (\Omega^{\bullet}(A), d_A)\hookrightarrow  (\Omega^{\bullet}(\F(s)), d_s),\]
(Hence $d_s= d_{\F(s)}$ is the DeRham differential of $\F(s)$ viewed as a Lie algebroid or, equivalently, DeRham differential along the fibers of $s$).  Recalling that the right translations allow us to extend a section $\alpha$ of $A$ to a vector field $\overrightarrow{\alpha}$ on $\G$ (tangent to the $s$-fibers) and to identify $t^*A$ with $\F(s)$ (so that $t^*\alpha$ corresponds to $\overrightarrow{\alpha}$), the inclusion above identifies $\omega\in \Omega^{\bullet}(A)$ with the foliated form $\overrightarrow{\omega}$ defined by $\overrightarrow{\omega}(\overrightarrow{\alpha_1}, \ldots )= \omega(\alpha_1, \ldots )$; the foliated forms of type 
$\overrightarrow{\omega}$ are precisely the foliated forms that are right-invariant; i.e., $\omega\mapsto \overrightarrow{\omega}$ gives an isomorphism
\[ t^*: (\Omega^{\bullet}(A), d_A)\cong  (\Omega^{\bullet}(\F(s)), d_s)^{\G-\textrm{inv}}.\]
For the later discussions, it is worth keeping in mind the structure that allows us to talk about $\G$-invariance: one has a space $P= \G$ together with a vector bundle $E= \F(s)$ over $P$, and $\G$ acts on both from the right: on $P$ by right translations, while the action on $\F(s)= t^*A$ is the tautological one ($(t^*A)_{ag}= (t^{*}A)_{a}$). 

We would like to understand the compactly supported counterpart of the previous discussion; the main point is that, in the resulting dual picture, subcomplexes (like $\Omega^{\bullet}(A)$ sitting as a sub-complex of $(\Omega^{\bullet}(\F(s)), d_s)$) will turn into quotient complexes, the inclusions will turn to into integrations over fibers and invariants into coinvariants. Let us also allow as coefficients any representation $E$ of $A$. 
The relevant complex at the level of $\G$ is then 
\[ (\Omega^{\bullet}_{c}(\F(s), t^*E \otimes s^*\mathcal{D}_A), d_{s}),\]
where $\F(s)$ is the foliation induced by the source map $s$. Regarding the coefficients, we have mentioned above that pull-backs by $s$ are canonically representations of $\F(s)$; for $t^*E$, the $\F(s)$-action is the pull-back of the action of $A$ on $E$: $\nabla_{\overrightarrow{\alpha}}(t^*e)= t^* \nabla_{\alpha}(e)$ for $\alpha\in C^\infty(M,A)$ and $e\in C^\infty(M,E)$). Recalling the identifications $\F(s)= t^*A$, $s^*A= \F(t)$, we see that 
\[ \Omega^{\bullet}_{c}(\F(s), t^*E\otimes s^*\mathcal{D}_A)= C_{c}^{\infty}(\G, t^*(\Lambda^{\bullet}A^*\otimes E)\otimes \mathcal{D}_{\F(t)})\]
hence the integration along the $t$-fibers makes sense as a map
\begin{equation}\label{eq-lemma-coinv-DeRham} 
\int_{t}: \Omega^{\bullet}_{c}(\F(s), t^{*}E\otimes s^*\mathcal{D}_A)\rightarrow C_{c}^{\infty}(M, \Lambda^{\bullet}A^*\otimes E)= \Omega_{c}^{\bullet}(A, E).
\end{equation}
Note also that we are in the position of talking about the $\G$-coinvariants of $\Omega^{\bullet}_{c}(\F(s), t^{*}E)$, in the sense of Remark \ref{remark-coinvariants} where, as above, $\G$ acts on $\G$ from the right and $\F(s)$ is viewed as an  equivariant $\G$-bundle over $\G$.

\begin{lemma}\label{lemma-coinv-DeRham}  The map $\int_{t}$ is a surjective morphism of cochain complexes which descends to an isomorphism
\[  \Omega^{\bullet}_{c}(\F(s), t^{*}E\otimes s^*\mathcal{D}_A)_{\G-\mathrm{coinv}} \stackrel{\sim}{\longrightarrow} \Omega_{c}^{\bullet}(A, E).\]
\end{lemma}

\begin{proof} The fact that $\int_{t}$ descends to a degreewise isomorphism follows from Corollary \ref{cor-coinv-princ} since the action of $\G$ on $\G$ is principal, with quotient map $t: \G\rightarrow M$. The main issue is the compatibility with the differentials. For that we use the Koszul formula for the differential $d_A$ (and the analogous formula for $d_s$): 
\begin{align*} 
\d_A\omega(\alpha_1, \ldots , \alpha_{k+1})=&\sum_{i}(-1)^{i+1} \nabla_{\alpha_i}(\omega(\alpha_1, \ldots , \widehat{\alpha_i}, \ldots , \alpha_{k+1}))+\\
& + \sum_{i< j} (-1)^{i+j}\omega([\alpha_i, \alpha_j], \ldots , \widehat{\alpha_i}, \ldots, \widehat{\alpha_j}, \ldots , \alpha_{k+1}),
\end{align*}
where $\nabla$ stands for the action of $A$ on $E$. For $\omega\in \Omega^{\bullet}_{c}(\F(s), t^{*}E\otimes s^*\mathcal{D}_A)$, denote $\overline{\omega}= \int_{t}\omega\in  \Omega_{c}^{\bullet}(A, E)$. Explicitly, for $\alpha_1, \ldots, \alpha_k\in C^{\infty}(M,A)$, $\omega(\overrightarrow{\alpha_1}, \ldots , \overrightarrow{\alpha_k})\in C^{\infty}_c(\G,t^*E\otimes \mathcal{D}_{\F(t)})$ and 
\[ \overline{\omega}(\alpha_1, \ldots , \alpha_k)= \int_{t} \omega(\overrightarrow{\alpha_1}, \ldots , \overrightarrow{\alpha_k}) \in C^{\infty}(M,E).\]
We have to show that $\overline{d_s\omega}= d_{A} \overline{\omega}$. Writing this out using the previous formulas, and using that $[\overrightarrow{\alpha}, \overrightarrow{\beta}]= \overrightarrow{[\alpha, \beta]}$, we see that the identities we have to check will follow provided we prove the commutativity of the following diagram:
\[ \xymatrix{
C^{\infty}_c(\G,t^*E\otimes \mathcal{D}_{\F(t)}) \ar[r]^-{\nabla_{\overrightarrow{\alpha}} }\ar[d]_-{\int_t}     &            C^{\infty}_c(\G, t^*E\otimes \mathcal{D}_{\F(t)} ) \ar[d]^-{\int_t} \\
                                C^{\infty}_c(M,E) \ar[r]_-{\nabla_{\alpha}}                                                                    &            C^{\infty}_c(M,E) }\]
Writing the elements in the right upper corner as $t^*e\otimes \rho$ with $e\in C^{\infty}(M,E)$, $\rho\in C^{\infty}_{c}(M,\mathcal{D}_{\F(t)})$ (which are mapped by $\int_{t}$ to $e\cdot \int_{t} \rho$) we see that one may assume that $E$ is the trivial representation. In this case the bottom $\nabla_{\alpha}$ becomes the Lie derivative along $\sharp\alpha$, the image of $\alpha$ by the anchor map of $A$. For the upper horizontal arrow, $\nabla_{\overrightarrow{\alpha}}$, recall that it is the canonical $\F(s)$-connection on $s^*A= \F(t)$, hence it is uniquely determined by the Leibniz identity and the fact that sections of type $s^*\beta\cong \overleftarrow{\beta}$ are flat (where $\overleftarrow{\beta}$ is the left-invariant vector field induced by $\beta$). On the other hand, since $[\overrightarrow{\alpha}, \overleftarrow{\beta}]= 0$ holds in general, we see that the usual Lie derivative $L_{\overrightarrow{\alpha}}$ (defined as variations along the flow $\phi_{\overrightarrow{\alpha}}^{\epsilon}$) has exactly the same properties; hence $\nabla_{\overrightarrow{\alpha}}= L_{\overrightarrow{\alpha}}$. Therefore, denoting $\overrightarrow{\alpha}= \tilde{V}$, $\sharp{\alpha}= V$, the new diagram becomes
\[ \xymatrix{
C^{\infty}_c(M,\mathcal{D}_{\F(t)}) \ar[r]^-{L_{\tilde{V}}} \ar[d]_-{\int_t}     &            C^{\infty}_c(M,\mathcal{D}_{\F(t)} ) \ar[d]^-{\int_t} \\
                                C_{c}^{\infty}(M) \ar[r]_-{L_{V}}                                                                    &            C_{c}^{\infty}(M) }.\]
This diagram is commutative for any submersion $t: \G\rightarrow M$ between two manifolds and any vector field $\tilde{V}$ on $\G$ that is $t$-projectable to $V$. 
% Since $t$ is locally a projection, a partition of unity argument reduces the problem to a local one, which becomes straightforward, provided one is careful with the way that $L_{\tilde{V}}$ acts on densities. Alternatively, 
To see this, we interpret it as the infinitesimal counterpart of the diagram involving the flows:
\[ \xymatrix{
C^{\infty}_c(M,\mathcal{D}_{\F(t)}) \ar[r]^-{     \phi_{\tilde{V}}^{\epsilon}} \ar[d]_-{\int_t}     &            C^{\infty}_c(M,\mathcal{D}_{\F(t)} ) \ar[d]^-{\int_t} \\
                                C_{c}^{\infty}(M) \ar[r]_-{\phi_{V}^{\epsilon}}                                                                   &            C_{c}^{\infty}(M) }.\]
The commutativity of the last diagram follows immediately from the fact that the flow of $\tilde{V}$ covers the flow of $V$ and from the invariance of integration of densities. 
\end{proof}

%%%%%%%%%%%%%%%%%%%%%%%%%%%
%%%%%%%%%%%%%%%%%%%%%%%%%%%
%%%%%%%%%%%%%%%%%%%%%%%%%%%
%%%%%%%%%%%%%%%%%%%%%%%%%%%
%%%%%%%%%%%%%%%%%%%%%%%%%%%
\subsection{The missing proofs}
\label{The missing proofs}
%%%%%%%%%%%%%%%%%%%%%%%%%%%
%%%%%%%%%%%%%%%%%%%%%%%%%%%
%%%%%%%%%%%%%%%%%%%%%%%%%%%
%%%%%%%%%%%%%%%%%%%%%%%%%%%
%%%%%%%%%%%%%%%%%%%%%%%%%%%
We now fill in the missing proofs. 

\begin{proof} (of Proposition \ref{pre-VE-iso-deg-0})
We have to show that 
\[ \textrm{Im}(d_A: \Omega^{r-1}_{c}(A, \mathfrak{o}_A)\rightarrow \Omega^{r}_{c}(A, \mathfrak{o}_A)) \subset \textrm{Im}(s_{!}- t_{!}).\]
Let $u= d_{A}(v)$ in the left hand side; to show that it is on the right hand side, we will use the following diagram
\begin{small}
\[\xymatrix{
C_{c}^{\infty}(M, \mathcal{D})                                   & C_{c}^{\infty}(\G, \mathcal{D}^2)= \Omega_{c}^{r}(\F(s), \mathfrak{o}_{\F(s)}\otimes s^*\mathcal{D}) \ar[d]_-{t_!}\ar[l]_-{s_!} &  \Omega_{c}^{r-1}(\F(s), \mathfrak{o}_{\F(s)}\otimes s^*\mathcal{D}) \ar[d]_-{t_!}\ar[l]_-{d_s}  \\
                                                                                    & \Omega_{c}^{r}(A, \mathfrak{o}_A)                                                                                                                                                         & \Omega_{c}^{r-1}(A, \mathfrak{o}_A)\ar[l]_-{d_A} 
 }\]
\end{small}

In this diagram, the first horizontal sequence is of the type considered in Lemma \ref{conf3} (hence the composition of the two maps is zero, and the sequence is even exact if the $s$-fibers are connected). The 
vertical maps are the ones of type (\ref{eq-lemma-coinv-DeRham}) (for $E= \mathfrak{o}_A$), they are surjective and the square is commutative by Lemma \ref{lemma-coinv-DeRham}. We can now look at $v$ as an element sitting in the bottom right corner and write it as $\int_{t}\xi$ for some $\xi$ in the upper right corner. Consider then $w= d_s(\xi)$. It follows that $u= t_{!}(w)$ and $s_{!}(w)= 0$. Hence $u$ is in the image of $s_!- t_!$.
\end{proof}

\begin{proof} (of Proposition \ref{VE-iso-deg-0})
Here we have to prove the reverse inclusion. So, let $u\in \textrm{Im}(s_{!}- t_{!})$. Since $\G$ is $s$-connected, it is clear that the previous argument can be reversed to conclude that $u\in \textrm{Im}(d_A)$, provided we can show that we can write $u= t_!w$ for some $w$ that is killed by $s_!$. For that we work on another diagram:
\[
\xymatrix{
C_{c}^{\infty}(\G, \mathcal{D}^2) \ar[d]_-{\delta}   & C_{c}^{\infty}(\G_2, \mathcal{D}^3) \ar[d]_-{\delta'}\ar[l]_-{\delta_{0\, !}}  & \\
C_{c}^{\infty}(M, \mathcal{D})                                   & C_{c}^{\infty}(\G, \mathcal{D}^2)= \ar[d]_-{t_!}\ar[l]_-{s_!} & \hspace{-0.9cm} C^{\infty}_c(\G, t^* \mathcal{D}_A\otimes s^*\mathcal{D}_A)\\
                                                                                    & C^{\infty}_c(M, \mathcal{D}_A) &}
\]
where $\delta= s_{!}- t_{!}$, $\delta'= \delta_{1\, !}- \delta_{2\, !}$ and $\delta_i: \G_2\rightarrow \G$ are given by (cf. Subsection \ref{Differentiable cohomology with compact supports}) 
\[ \delta_0(g, h)= h,\ \delta_1(g, h)= gh, \ \delta_2(g, h)= g.\]

The commutativity of the diagram follows from the functoriality of the fiber integration and the obvious identities $s\circ \delta_0= s\circ \delta_1$, $t\circ \delta_0= s\circ \delta_2$. Lemma \ref{conf2} applied to $E= \mathcal{D}_A$ implies that the left vertical sequence is exact.
%The functoriality of the fiber integration and the obvious identities $s\circ \delta_0= s\circ \delta_1$ and $t\circ \delta_0= s\circ \delta_2$ imply that the diagram is commutative. Moreover, the left vertical sequence is exact by Lemma \ref{conf2} applied to $E= \mathcal{D}_A$ since it is part of the sequence computing the compactly supported differentiable cohomology with coefficients for the action groupoid associated to the right action of $\G$ on $\G$. 
Look now at $u$ sitting in the lowest left corner; the hypothesis is that $u= \delta(v)$ for some $v$. Write $v= \delta^{0}_{!}(\xi)$ for some $\xi\in C_{c}^{\infty}(\G_2, \mathcal{D}^3)$ and consider $w'= \delta'(\xi)$, Then $u= s_!(w')$ and also $t_!(w')= 0$ since $t_!\circ \delta'= 0$. Of course, using the inversion $\iota$ of the groupoid, $w= \iota^*(w')$ will have the desired properties.
\end{proof}

\begin{proof} (of the Van Est isomorphism - Theorem \ref{thm-Van-Est})
For notational simplicity we assume that $E$ is the trivial representation. As in \cite{VanEst}, we use an augmented double complex 
argument; it is obtained by extending the last two diagrams:
\begin{tiny}
\[
\xymatrix{
 \ldots \ar[d]_-{\delta}&  \ldots \ar[d]_-{\delta} & \ldots \ar[d]_-{\delta} & \ldots \\
C_{c}^{\infty}(\G_2, \mathcal{D}^3) \ar[d]_-{\delta}& C_{c}^{\infty}(\G_3, \mathcal{D}^4)\ar[d]_-{\delta}\ar[l]_-{\delta_{0\, !}}                                                                                                               & C_{1, 2}  \ar[l]_-{d_{\delta_0}}\ar[d]_-{\delta} & \ldots \ar[l]_-{d_{\delta_0}}
 \\
C_{c}^{\infty}(\G, \mathcal{D}^2) \ar[d]_-{\delta}   & C_{c}^{\infty}(\G_2, \mathcal{D}^3)\ar[d]_-{\delta}\ar[l]_-{\delta_{0\, !}}                                                                                                             & C_{1, 1} \ar[l]_-{d_{\delta_0}} \ar[d]_-{\delta} & \ldots  \ar[l]_-{d_{\delta_0}}
 \\
C_{c}^{\infty}(M, \mathcal{D})                                   & C_{c}^{\infty}(\G, \mathcal{D}^2)= \Omega_{c}^{r}(\F(s), \mathfrak{o}_{\F(s)}\otimes s^*\mathcal{D}) \ar[d]_-{t_!}\ar[l]_-{\delta_{0\, !}= s_!} &  \Omega_{c}^{r-1}(\F(s), \mathfrak{o}_{\F(s)}\otimes s^*\mathcal{D}) \ar[d]_-{t_!}\ar[l]_-{d_s} & \ldots \ar[l]_-{d_s} \\
                                                                                    & \Omega_{c}^{r}(A, \mathfrak{o}_A)                                                                                                                                                         & \Omega_{c}^{r-1}(A, \mathfrak{o}_A)\ar[l]_-{d_A} &  \ldots \ar[l]_-{d_A}
}
\]
\end{tiny}

The first vertical row and the bottom horizontal one compute the cohomologies from the statement, and they are the augmentations of the rows / columns of the actual double complex $C_{\bullet, \bullet}$. Explicitly, the double complex is defined as:
\[ C_{i, k}= C_{c}^{\infty}(\G_{k+1}, p_{0}^{*}(\Lambda^{r-i} A^*\otimes \mathfrak{o}_A) \otimes p_{1}^{*}\mathcal{D}_A\otimes \ldots \otimes p_{k+1}^{*} \mathcal{D}_A).\]
For the differentials, let us look at the columns and the rows separately. Start with the $i^{th}$ column, for a fixed $i$. We will describe it as the complex computing the compactly supported cohomology of a groupoid $\tilde{\G}$ with coefficients in a representation that depends on $i$. For that, we consider the action of $\G$ on itself from the right, which makes $\G$ into a (right) principal $\G$-bundle over $M$ with 
projection $t: \G\rightarrow M$,
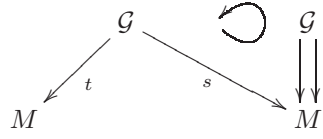
\[
\xymatrix{
  & \G\ar[dl]^-{t}\ar[drr]_-{s} & \ar@(dr, ur)   & \G \ar@<0.25pc>[d] \ar@<-0.25pc>[d]\\
M  & & & M}
\]
Let $\tilde{\G}$  be the resulting action groupoid: $\tilde{\G}= \G_2$ viewed as a groupoid over $\G$ with 
\[ (g, h)\stackrel{s= \delta_1}{\longrightarrow} gh, \quad (g, h)\stackrel{t= \delta_2}{\longrightarrow} g, \quad  (g, h)\cdot (gh, k)= (g, hk).\]
% and with the multiplication 
% \[ (g, h)\cdot (gh, k)= (g, hk).\]
The algebroid $\tilde{A}$ of $\tilde{\G}$ is $\tilde{A}= \mathcal{F}(t)\cong s^*A $. 
The representations of $\tilde{\G}$ that are relevant here are tautological, i.e., pull-backs along the principal bundle projection $t$; more precisely, they are:
\[ E_i= t^*(\Lambda^{r-i} A^*\otimes \mathfrak{o}_A).\]
It is now straightforward to check that, as vector spaces,
\[ C_{i, k}= C_c^k(\tilde{\G}, E_i).\]
This defines the augmented $i^{th}$ column which, by Corollary \ref{cor-coinv-princ}, is exact. 

We now describe the $k^{th}$ row. For that, we first remark that $\delta_{0}: \G_{k+1}\rightarrow \G_k$ is a left principal $\G$-bundle over $M$, with the action defined along $p_0$:
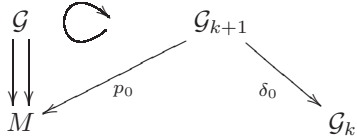
\[
\xymatrix{
 \G \ar@<0.25pc>[d] \ar@<-0.25pc>[d]  & \ar@(dl, ul) & \G_{k+1}\ar[dll]^-{p_0}\ar[dr]_-{\delta_0} &  \\
M & & &  \G_k}
\]
This gives an identification of $p_{0}^{*}(A)$ with the foliation $\cF(\delta_0)$, so that 
\begin{small}\[ C_{i, k}=%  C_{c}^{\infty}(\G_{k+1}, \Lambda^{r-i} \cF(\delta_0)^*\otimes \mathfrak{o}_{\cF(\delta_0)} \otimes p_{1}^{*}\mathcal{D}_A\otimes \ldots \otimes p_{k+1}^{*} \mathcal{D}_A)=
\Omega^{r-i}_{c}(\cF(\delta_0), \mathfrak{o}_{\cF(\delta_0)} \otimes p_{1}^{*}\mathcal{D}_A\otimes \ldots \otimes p_{k+1}^{*} \mathcal{D}_A)= \Omega^{r-i}_{c}(\cF(\delta_0), \mathfrak{o}_{\cF(\delta_0)} \otimes (\delta_0)^*\mathcal{D}_{k+1}).\]
\end{small}
We see that we deal with spaces of the type that appear in Lemma \ref{conf3} - and this interpretation describes the horizontal differentials $d_{\delta_0}$ so that (by the lemma) the augmented rows have vanishing cohomology in the relevant degrees. 

Of course, one still has to check the compatibility of the (augmented) horizontal and the vertical differentials; in low degree, this is contained in the previous two proofs; in arbitrary degrees it is tedious but straightforward (and rather standard- see e.g. \cite{VanEst}); the trickiest part is for the squares in the bottom (involving $t_{!}$), but that was dealt with in Lemma \ref{lemma-coinv-DeRham}.
\end{proof}

%%%%%%%%%%%%%%%%%%%%%%%%%%%
%%%%%%%%%%%%%%%%%%%%%%%%%%%
%%%%%%%%%%%%%%%%%%%%%%%%%%%
%%%%%%%%%%%%%%%%%%%%%%%%%%%
%%%%%%%%%%%%%%%%%%%%%%%%%%%
%%%%%%%%%%%%%%%%%%%%%%%%%%%
%%%%%%%%%%%%%%%%%%%%%%%%%%%
%%%%%%%%%%%%%%%%%%%%%%%%%%%
%%%%%%%%%%%%%%%%%%%%%%%%%%%
\section{Appendix: Haar systems and cut-off functions, revisited}\label{Appendix}
%%%%%%%%%%%%%%%%%%%%%%%%%%%
%%%%%%%%%%%%%%%%%%%%%%%%%%%
%%%%%%%%%%%%%%%%%%%%%%%%%%%
%%%%%%%%%%%%%%%%%%%%%%%%%%%
%%%%%%%%%%%%%%%%%%%%%%%%%%%
%%%%%%%%%%%%%%%%%%%%%%%%%%%
%%%%%%%%%%%%%%%%%%%%%%%%%%%
%%%%%%%%%%%%%%%%%%%%%%%%%%%
%%%%%%%%%%%%%%%%%%%%%%%%%%%

The notion of Haar systems on groupoids, extending the notion of the Haar measures on Lie groups, is well-known \cite{renault}. In some sense, Haar systems are orthogonal to the transverse measures of this paper. They
are used in the paper but are of independent interest. Therefore, we collect some of the basics on Haar systems in this Appendix.  This will also give us the opportunity of clarifying a few points that are perhaps not so well-known.

%%%%%%%%%%%%%%%%%%%%%%%%%%%
%%%%%%%%%%%%%%%%%%%%%%%%%%%
%%%%%%%%%%%%%%%%%%%%%%%%%%%
%%%%%%%%%%%%%%%%%%%%%%%%%%%
%%%%%%%%%%%%%%%%%%%%%%%%%%%
\subsection{Haar systems}\label{subsec-app-Haar systems} 
%%%%%%%%%%%%%%%%%%%%%%%%%%%
%%%%%%%%%%%%%%%%%%%%%%%%%%%
%%%%%%%%%%%%%%%%%%%%%%%%%%%
%%%%%%%%%%%%%%%%%%%%%%%%%%%
%%%%%%%%%%%%%%%%%%%%%%%%%%%

Throughout this appendix $\G$ is a Lie groupoid over a manifold $M$. To talk about right-invariance, one first remarks that the right multiplication by an arrow $g: x\rightarrow y$ of $\G$ is no longer defined on the entire $\G$ (as for groups) but only between the $s$-fibers:
\begin{equation}\label{Rg} 
R_{g}: s^{-1}(y)\rightarrow s^{-1}(x).
\end{equation}
% Hence, in principle, right-invariance is a property for families of objects that live on the $s$-fibers of $\G$. 
Hence, roughly speaking, right-invariance makes sense only for families of objects that live on the $s$-fibers of $\G$.

\begin{definition} \label{def-Haar-system}
A \textbf{smooth Haar system} on a Lie groupoid $\G$ is a family 
\[ \mu= \{\mu^x\}_{x\in M}\] 
of non-zero measures $\mu^{x}$ on $s^{-1}(x)$ which is right-invariant and smooth, i.e.: %with the following properties
\begin{enumerate}
\item[1.] via any right-translation (\ref{Rg}) by an element $g: x\rightarrow y$, $\mu^x$ is pulled-back to $\mu^y$; or, in the integral notation (Remark 
\ref{notation-integrals-0}),
\[ \int_{s^{-1}(y)} f(hg)\  d\mu^{y}(h)= \int_{s^{-1}(x)} f(h)\ d\mu^x(h).\]
\item[2.] for any $f\in C_{c}^{\infty}(\G)$, the function obtained by integration over the $s$-fibers, 
\[ M\ni x\mapsto  \mu^x(f|_{s^{-1}(x)})= \int_{s^{-1}(x)} f(h)\ d\mu^{x}(h)\in \mathbb{R}\]
%(denoted $s_{!}^{\mu}(f)$)
is smooth.
\end{enumerate}
The Haar system is called \textbf{full} if the support each $\mu^x$ is the entire $s^{-1}(x)$. 
% ifthe support is the entire $\G$, we say that $\mu$ is a ful Haar system.
% The Haar system is called full if the support each $\mu^x$ is the entire $\mu^{-1}(x)$.  
\end{definition}

In general, the support of Haar system  $\mu$ is defined as
\[ \textrm{supp}(\mu)= \cup_{x\in M} \textrm{supp}(\mu^x) \subset \G.\]
Due to the invariance of $\mu$, $\textrm{supp}(\mu)$ is right $\G$-invariant, i.e., $ag\in \textrm{supp}(\mu)$ for all $a\in \textrm{supp}(\mu)$ and $g\in \G$ composable. Therefore $\textrm{supp}(\mu)$ is made of $t$-fibers, hence it is determined by its $t$-projection, which is denoted
\[ \textrm{supp}_{M}(\mu):= t(\textrm{supp}(\mu)).\]
More precisely, we have:
\begin{equation}
\label{eq-supports} 
\textrm{supp}(\mu)= t^{-1}(\textrm{supp}_{M}(\mu)).
\end{equation}

\begin{remark} In the existing literature one often restricts to what we call here full Haar systems. However, some important constructions (e.g. already for the averaging process for proper groupoids) are based on systems that are not full. For us, the fullness is replaced by the condition that each $\mu_{x}$ is non-trivial. In turn, our condition 
has consequences on the supports, expressed in one of the following equivalent ways
\begin{itemize}
\item $s|_{\textrm{supp}(\mu)}: \textrm{supp}(\mu)\rightarrow M$ is surjective.
\item the $\G$-saturation of $\textrm{supp}_{M}(\mu)$ is the entire $M$.
\end{itemize}
\end{remark}

%%%%%%%%%%%%%%%%%%%%%%%%%%%
%%%%%%%%%%%%%%%%%%%%%%%%%%%
%%%%%%%%%%%%%%%%%%%%%%%%%%%
%%%%%%%%%%%%%%%%%%%%%%%%%%%
%%%%%%%%%%%%%%%%%%%%%%%%%%%

\subsection{Geometric Haar systems (Haar densities)}\label{subsec-Geometric Haar systems-Appendix}
%%%%%%%%%%%%%%%%%%%%%%%%%%%
%%%%%%%%%%%%%%%%%%%%%%%%%%%
%%%%%%%%%%%%%%%%%%%%%%%%%%%
%%%%%%%%%%%%%%%%%%%%%%%%%%%
%%%%%%%%%%%%%%%%%%%%%%%%%%%

We now pass to Haar systems that are of geometric type, i.e., for which each $\mu^x$ comes from a density on $s^{-1}(x)$; this will bring us to sections
\[ \rho\in C^\infty(M,\mathcal{D}_A)\]
of the density bundle associated to the Lie algebroid $A$ of $\G$. To see this, recall the construction of $A$: it is the vector bundle over $M$ whose fiber above a point $x\in M$ is the tangent space at the unit $1_x$ of the $s$-fiber above $x$ : %(where $s: \G\rightarrow M$ stands for the source map):
\[ A_x= T_{1_x} s^{-1}(x) .\] 
(Globally, $A$ is the restriction to $M$, via the unit map $M\hookrightarrow \G$, of the bundle $T^s\G= \textrm{Ker}(ds)$ of vector tangent to the $s$-fibers, also denoted by $\cF(s)$). With this, the right translation (\ref{Rg}) associated to an arrow $g: x\rightarrow y$ induces, after differentiation at the unit at $y$, an isomorphism:
\[ R_g: A_y\rightarrow T_g (s^{-1}(x))=T^s_g\G.\]
In this way any section $\alpha\in \Gamma(A)$ gives rise to a vector field $\overrightarrow{\alpha}$ on $\G$ and this identifies $\Gamma(A)$ with the space of vector fields on $\G$ that are tangent to the $s$-fibers and invariant under right translations; the Lie algebroid bracket $[\cdot, \cdot]$ on $\Gamma(A)$ is induced by the usual Lie bracket of vector fields on $\G$: $\overrightarrow{[\alpha, \beta]}= [\overrightarrow{\alpha}, \overrightarrow{\beta}]$). 

Stated in the spirit of the previous definition, this reinterprets elements $\alpha \in \Gamma(A)$  as families of vector fields on the $s$-fibers of $\G$, that are right-invariant and smooth. It is clear that the same reasoning applies to sections $\rho\in \Gamma(\cD_A)= C^\infty(M,\mathcal{D}_A)$, so that such sections can be interpreted as families of densities on the $s$-fibers of $\G$, that are right-invariant and smooth. Explicitly, right translating $\rho$ we obtain a (right-invariant) density on the vector bundle $\cF(s)$, 
\[ \overrightarrow{\rho}\in C^\infty(\G, \mathcal{D}_{\cF(s)}),\]
hence a family of densities 
\begin{equation}\label{dens-ind-s-fiber} 
\rho^{x}:= \overrightarrow{\rho}|_{s^{-1}(x)}\in \mathcal{D}(s^{-1}(x))\ \ \ (x\in M).
\end{equation}
Note that, for the usual notion of support of sections of vector bundles, we have the analogue of (\ref{eq-supports}):
\[ \textrm{supp}(\overrightarrow{\rho})= t^{-1}(\textrm{supp}(\rho)).\]
Of course, the two are related.

\begin{definition}\label{def-Haar-density} A \textbf{Haar density} for $\G$ is any positive density $\rho\in C^\infty(M,\mathcal{D}_A)$ of the Lie algebroid $A$ of $\G$ with the property that the $\G$-saturation of its support is the entire $M$.
It is called a full Haar density if it is strictly positive. 
\end{definition}

The previous discussion shows that these correspond to (full) Haar systems of geometric type, i.e., for which each of the measures is induced by a density. For $\rho$ such a Haar density, we will denote by
\[ \mu_{\rho}= \{ \mu_{\rho^x}\}_{x\in M}\]
the corresponding Haar system (given by (\ref{dens-ind-s-fiber})).  

\begin{remark} Sections $\rho\in C^\infty(M, \mathcal{D}_{A})$ can be multiplied by smooth functions $f\in C^{\infty}(M)$. At the level of measures we find
\[ \mu_{f\rho}(g)= f(t(g)) \mu_{\rho}(g).\]
Hence we are led to the operation of multiplying a Haar system $\mu$ by a positive function $f\in C^{\infty}(M)$ which, on $\G$, corresponds to the standard multiplication by $t^*(f)$. 
Note that 
\[ \textrm{supp}_M(f\cdot \mu) \subset \textrm{supp}(f)\]
hence, even if $\mu$ is full, for $f\cdot \mu$ to be a Haar system one still has to require that the $\G$-saturation of the support of $f$ is $M$.
\end{remark}

%%%%%%%%%%%%%%%%%%%%%%%%%%%
%%%%%%%%%%%%%%%%%%%%%%%%%%%
%%%%%%%%%%%%%%%%%%%%%%%%%%%
%%%%%%%%%%%%%%%%%%%%%%%%%%%
%%%%%%%%%%%%%%%%%%%%%%%%%%%
\subsection{Proper Haar systems; proper groupoids}\label{subsec-Proper Haar Systems-Appendix}
%%%%%%%%%%%%%%%%%%%%%%%%%%%
%%%%%%%%%%%%%%%%%%%%%%%%%%%
%%%%%%%%%%%%%%%%%%%%%%%%%%%
%%%%%%%%%%%%%%%%%%%%%%%%%%%
%%%%%%%%%%%%%%%%%%%%%%%%%%%

One of the standard uses of the Haar measure of compact Lie groups is to produce $G$-invariant objects out of arbitrary ones, by averaging. On the other hand it is well-known that the correct generalization of compactness when going from Lie groups to groupoids is properness. Recall here that a Lie groupoid $\G$ over $M$ is called proper if the map 
\[ (s, t): \G\rightarrow M\times M\]
is proper i.e., for every $K, L\subset M$ compacts, the subspace $\G(K, L)\subset \G$ of arrows that start in $K$ and end in $L$ is compact. The point is that averaging arguments do work well for proper groupoids. However, there is a small subtlety due to the fact that we would have to integrate over the $s$-fibers which may fail to be compact even if $\G$ is proper. For that reason, we need Haar systems $\mu= \{\mu^x\}$ in which each $\mu^x$ is compactly supported (see Lemma \ref{cpctly-cupp-measures}). 

\begin{definition} \label{proper-dens} A Haar system on a Lie groupoid $\G$ is said to be \textbf{proper} if 
\[ s|_{\textrm{supp}(\mu)}: \textrm{supp}(\mu)\rightarrow M\] 
is  proper. A Haar density $\rho$ is said to be proper if $\mu_{\rho}$ is proper, i.e., if the restriction of $s$ to $t^{-1}(\textrm{supp}(\rho))$ is proper.
\end{definition}

In this case one can talk about the volume of the $s$-fibers,
\[ \textrm{Vol}(s^{-1}(x), \mu^x)= \int_{s^{-1}(x)}\ d\mu^x \]
and this defines a smooth function on $M$ which is strictly positive; hence, rescaling $\mu$ by it (in the sense of the previous remark), we obtain a new proper Haar system satisfying the extra condition
\[ \textrm{Vol}(s^{-1}(x), \mu^x)= 1.\]
Such proper Haar systems are called normalized. Similarly for Haar densities. 

Note that full Haar systems cannot be proper even for proper non-compact groupoids. Also, while full Haar systems exist on all Lie groupoids, proper ones do not. Actually, we have:

\begin{proposition} \label{equiv-proper} For a Lie groupoid $\G$, the following are equivalent:
\begin{enumerate}
\item[1.] $\G$ is proper.
\item[2.] $\G$ admits a proper Haar system.
\item[3.] $\G$ admits a proper Haar density.
\end{enumerate}
In particular, if $\rho\in C^\infty(M, \mathcal{D}_A)$ is a full Haar density then there exists a function $c\in C^{\infty}(M)$ such that $c\cdot \rho$ is a proper Haar density which, moreover, may be arranged to be normalized. Such $c$ is called a \textbf{cut-off function for $\rho$} (cf. \cite{Tu}).
\end{proposition}

\begin{proof} The fact that $\G$ is proper if a proper Haar system $\mu$ exists will follow from the properties of $\textrm{supp}(\mu)$ (right $\G$-invariance plus the fact that the restriction of $s$ to $\textrm{supp}(\mu)$ is proper and surjective) or, equivalently, those of 
$T:= \textrm{supp}_{M}(\mu)$:
\begin{itemize}
\item the $\G$-saturation of $T$ is the entire $M$ 
\item $\G(K, T)$ (arrows that start in $K$ and end in $T$) is compact if $K$ is.
\end{itemize}
Indeed, the first property implies that for any $K, L\subset M$ one has
\[ \G(K, L)\subset \G(K, T)^{-1} \cdot \G(L, T) \]
(for $g$ in the left hand side, choose any $a$ from $t(g)$ to an element in $T$ and write $g= a^{-1}\cdot (ag)$); therefore, if 
if $K$ and $L$ are compacts, then so will be $\G(K, L)$ (as closed inside a compact). 

Recall that a full Haar density $\rho$ always exists and if we multiply it by a positive function $c$ then the $M$-support of the resulting Haar system coincides with $\textrm{supp}(c\rho)= \textrm{supp}(c)$. Therefore, to close the proof, it suffices to show that if $\G$ is proper then one finds a smooth function $c$ such that its support $T:= \textrm{supp}(c)$ has 
the properties mentioned above. We first construct $T$. Recall that a slice through $x\in M$ is any $T\subset M$ that intersects transversally all the orbits that it meets, intersects the orbit through $x$ only at $x$ and its dimension equals the codimension of that orbit. We construct $T$ as a union of such slices (enough, but not too many). For that we need the basic properties of slices for proper groupoids (cf. e.g. \cite{ivan}): through each point of $M$ one can find a slice and, for any slice $\Sigma$, its saturation $\G \Sigma$ is open in $M$ (in particular, $\pi(\Sigma)$ is open in $M/\G$). Since 
$M/\G$ is paracompact, we find a family $\{\Sigma_{i}^{'}\}_{i\in I}$ of locally compact slices such that $V_{i}^{'}:= \pi(\Sigma_{i}^{'})$ defines a locally finite cover of $M/\G$. As usual, we refine this cover to a new cover $\{V_i\}_{i\in I}$ with $\overline{V}_i\subset V_{i}^{'}$ and write $V_i= \pi(\Sigma_i)$ with $\Sigma_i= \pi^{-1}(V_i)\cap \Sigma_{i}^{'}$. Note that 
\[ \pi(\overline{\Sigma}_i)\subset \overline{\pi(\Sigma_i})= \overline{V}_i\subset V_{i}^{'}= \pi(\Sigma_{i}^{'})\]
hence $\overline{\Sigma}_i$ is contained in the saturation of $\Sigma_{i}^{'}$ (but may fail to be contained in $\Sigma_{i}^{'}$!). 

We claim that $T:= \cup_{i} \overline{\Sigma}_i$ will have the desired properties; the saturation property is clear (even before taking closures) since the $V_i$'s cover $M/\G$. Next, 
the compactness of $\G(K, T)$ when $K$ is compact: if $B(K, \overline{\Sigma}_{i})$ is non-empty, then 
$\pi(K)$ must intersect $\pi(\overline{\Sigma}_{i})$ hence also $V_{i}^{'}$; but, since $\pi(K)$ is compact, it can intersects only a finite number of $V_{i}^{'}$; therefore we find indices $i_1, \ldots , i_k$ such that 
\[ B(K, T)= B(K, \overline{\Sigma}_{i_1}) \cup \ldots \cup B(K, \overline{\Sigma}_{i_1}),\]
hence $B(K, T)$ is compact. Note that the proof shows a bit more: instead of $\Sigma_i$ we can use any $C_i$ relatively compact, with
\[ \Sigma_i\subset C_i\subset \overline{C}_i\subset \G \Sigma^{'}_{i},\]
and then $T= \cup_{i} \overline{C}_i$ still has the desired properties. This is important for constructing $c$ since not every closed subspace can be realized as the support of a smooth function. 
Since each $\overline{\Sigma}_i$ is compact and sits inside the open $\G \Sigma_{i}^{'}$, we find $c_i: M\rightarrow [0, 1]$ smooth, supported in $\G \Sigma_{i}^{'}$ with $c_{i}> 0$ on $\overline{\Sigma}_i$. Hence, in the previous construction we can set $C_i= \{c_i> 0\}$ to define $T$. 
By construction, $\pi(\textrm{supp}(c_i))$ sits inside $V^{'}_i$, hence $\{c_i\}_{i\in I}$ is locally finite and $c= \sum_{i} c_i$ makes sense as a smooth function. Moreover, $\{c\neq 0\}= \cup_{i} \{c_i\neq 0\}$ is a locally finite union, hence
\[ \textrm{supp}(c)= \cup_{i} \textrm{supp}(c_i)= T. \]\end{proof}

Let us also illustrate the averaging technique with one simple example. 

\begin{lemma}\label{illustration-average}
If $\G$ is a proper Lie groupoid over $M$ and $A, B\subset M$ are two closed disjoint subsets that are saturated, then there exists a smooth function $f: M\rightarrow [0, 1]$ that $\G$-invariant (i.e.,  constant on the orbits of $\G$) such that $f|_{A}= 0$, $f|_{B}= 1$.
\end{lemma}

Without the invariance condition, this is a basic property of smooth functions on manifolds. The idea is that, choosing any $f$ as above but possibly non-invariant, one replaces it by its average with respect to a proper normalized Haar system $\mu$:
\[\textrm{Av}_{\mu}(f)(x)= \int_{s^{-1}(x)} f(t(g))\ d\mu^x(g) .\]
While it is clear that $\textrm{Av}(f)$ vanishes on $A$, the normalization implies that it is $1$ on $B$. Also the invariance is immediate: 
if $x$ and $y$ are in the same orbit, i.e., if there exists an arrow $a: x\rightarrow y$ then, using the invariance of $\mu$ and $t(ga)= t(g)$, we find
\[ \textrm{Av}_{\mu}(f)(x)= \int_{s^{-1}(x)} f(t(g))\ d\mu^{x}(g)=  \int_{s^{-1}(y)} f(t(ga))\ d\mu^{y}(g)= \textrm{Av}_{\mu}(f)(y).\]
By similar techniques one proves the existence of invariant metrics on $\G$-vector bundles, or of other geometric structures.

%%%%%%%%%%%%%%%%%%%%%%%%%%%
%%%%%%%%%%%%%%%%%%%%%%%%%%%
%%%%%%%%%%%%%%%%%%%%%%%%%%%
%%%%%%%%%%%%%%%%%%%%%%%%%%%
%%%%%%%%%%%%%%%%%%%%%%%%%%%
\subsection{Induced measures/densities on the orbits}
\label{Induced measures/densities on the orbits}
%%%%%%%%%%%%%%%%%%%%%%%%%%%
%%%%%%%%%%%%%%%%%%%%%%%%%%%
%%%%%%%%%%%%%%%%%%%%%%%%%%%
%%%%%%%%%%%%%%%%%%%%%%%%%%%
%%%%%%%%%%%%%%%%%%%%%%%%%%%
The slogan that Haar systems and densities are related to measure theory along the orbits can be made very precise in the case of proper groupoids: in that case 
they induce (in a canonical fashion) measures/densities on the orbits (and conversely!). To explain this, let $\G$ be a proper groupoid over $M$ and let $\mu=\{\mu^x\}$ be a Haar system for $\G$. Above each $x\in M$ one has the isotropy Lie group $\G_x$ consisting of arrows that start and end at $x$ and the $s$-fiber above $x\in M$ is a principal $\G_x$-bundle over the orbit $\mathcal{O}_x$ of $\G$ through $x$, with projection map 
\begin{equation}\label{princ-=Gx-bdle}
t: s^{-1}(x)\rightarrow \mathcal{O}_x.
\end{equation}
The invariance of the family $\{\mu^{x}\}$ implies that each $\mu^x$ is an invariant measure on this bundle hence, since $\G_x$ is compact, $\mu^x$ corresponds to a measure on  $\mathcal{O}_x$  (cf. Example \ref{inv-meas}), namely $\mu_{\mathcal{O}_x}= t_!(\mu^x)$ or, in the integral formulation, 
\[ \int_{\mathcal{O}_x} f(y)\ d\mu_{\mathcal{O}_x}(y)= \int_{s^{-1}(x)} f(t(g))\ d\mu^x(g).\]
Using again the invariance of $\mu$, we see that each such $\mu_{\mathcal{O}_x}$ depends only on the orbit itself and not on the point $x$ in the orbit. Therefore any Haar system $\mu$ determines (and is determined by) a family of measures on the orbits of $\G$, 
\begin{equation}\label{measure-fam} 
\{ \mu_{\mathcal{O}}% \in \mathcal{D}(\mathcal{O})
\}_{\mathcal{O}- \textrm{orbit\ of}\ \G} .
\end{equation}
Of course, this is compatible with ``geometricity'', so that any Haar density $\rho\in C^\infty(M,\mathcal{D}_A)$ determines (and is determined by) a family of densities on the orbits of $\G$,
\begin{equation}\label{density-fam} 
\{ \rho_{\mathcal{O}}\in \mathcal{D}(\mathcal{O})\}_{\mathcal{O}- \textrm{orbit\ of}\ \G} .
\end{equation}
in such a way that $(\mu_{\rho})_{\mathcal{O}}=\mu_{\rho_{\mathcal{O}}}$. Explicitly, given $\rho$, the density $\rho_{\mathcal{O}}$ on the orbit through $x$ is $t_!(\rho^x)$, where $\rho^x= \overrightarrow{\rho}|_{s^{-1}(x)}$ is the induced density on the $s$-fiber - see (\ref{dens-ind-s-fiber}). Also, the discussion from Example \ref{inv-meas} tells us how to recover $\rho$ from this family (and similarly for $\mu$): consider $\mathfrak{g}_x=$ the isotropy Lie algebra at $x$ (the Lie algebra of $\G_x$), with corresponding Haar density denoted by $\mu_{\mathrm{Haar}}$ and then 
\begin{equation}\label{app-Haar-subtelty} 
\rho(x)\in \mathcal{D}_{A_x}\cong \mathcal{D}_{\mathfrak{g}_x}\otimes \mathcal{D}(T_x\mathcal{O})\ \ \textrm{equals\ to}\ \mu_{\mathrm{Haar}}\otimes \rho_{\mathcal{O}}(x).
\end{equation}

\begin{remark} \rm\ Therefore, for proper groupoids, a Haar system $\mu$ (Haar density $\rho$) is the same thing as a family of measures (densities) on the orbits, satisfying a certain smoothness condition that ensures that the reconstructed $\mu$ is smooth. Working this out we find the condition that
for any $f\in C_{c}^{\infty}(\G)$, the function 
\[ M\ni x \mapsto \int_{\mathcal{O}_x} \int_{\G_{x, y}} f(g)\ d\mu_{x, y}^{\mathrm{Haar}}(g) d\mu_{\rho_{\mathcal{O}_x}}(y) \in \mathbb{R}\]
is smooth. Here $\G_{x, y}$ is the space of arrows from $x$ to $y$; they are fibers of (\ref{princ-=Gx-bdle}), hence they come with a canonical Haar density (cf. Example \ref{inv-meas}) obtained by transporting the Haar density on $\G_x$ and $d\mu_{x, y}^{\mathrm{Haar}}$ stands for the resulting integration.
\end{remark}

\begin{corollary}\label{cor-rewrite-Haar}
If $\G$ is proper, there is a 1-1 correspondence between:
\begin{enumerate}
\item[1.] Haar densities $\rho$ of $\G$.
\item[2.] families (\ref{density-fam}) of non-trivial densities on the orbits of $\G$, smooth in the previous sense.
\end{enumerate}
Similarly for Haar systems.
\end{corollary}

The situation is even nicer in the regular case. For densities, we make use of the foliation on $M$ by the (connected components of the) orbits of $\G$, identified with the vector bundle 
$\F\subset TM$ of vectors tangent to the orbits.%  (which can also be described as the image of the anchor map $\sharp: A\rightarrow TM$). 

\begin{corollary}
If $\G$ is proper and regular, there is a 1-1 correspondence between 
\begin{enumerate}
\item[1.] Haar densities $\rho$ of $\G$,
\item[2.]  positive densities of the bundle $\F$, $\rho_{\F}\in C^\infty(M, \mathcal{D}_{\F})$, 
with the property that the support of $\rho_{\F}$ meets each orbit of $\G$. 
\end{enumerate}
Moreover, $\rho$ is full if and only if $\rho_{\F}$ is strictly positive, and $\rho$ is proper if and only if 
the intersection of the support of $\rho_{\F}$ with each leaf is compact.
\end{corollary}

For Haar systems we obtain:

\begin{corollary}
If $\G$ is proper and regular, there is a 1-1 correspondence between:
\begin{enumerate}
\item[1.] Haar systems $\mu$ of $\G$.
\item[2.] families of non-trivial measures $\mu_{\mathcal{O}}$ on the orbits of 
$\G$ which are smooth in the sense that, for any $f\in C_{c}^{\infty}(M)$, the function
\[ x\mapsto \int_{\mathcal{O}_x} f(y)\ d\mu_{\mathcal{O}_x}(y)\]
is smooth. 
\end{enumerate}
\end{corollary}

The last two corollaries imply that, for proper regular Lie groupoids with connected $s$-fibers, having a Haar system/density for $\G$ is equivalent to having one for $\F$, where the notion of Haar measure/density for a foliation $\F$ is the one described by points 2. of the previous corollaries.

\bibliography{mybiblio}   % name your BibTeX data base
\bibliographystyle{plain}

\end{document}